\documentclass[11pt]{article}

\usepackage{amssymb}
\usepackage{amsmath}
\usepackage{theorem}
\usepackage{epsfig}
\usepackage{verbatim}
\usepackage{graphicx}
\usepackage{subfigure}

\usepackage{color}

\newtheorem{theorem}{Theorem}[section]
\newtheorem{lemma}[theorem]{Lemma}

\newtheorem{proposition}[theorem]{Proposition}
\newtheorem{corollary}[theorem]{Corollary}

\newtheorem{defi}[theorem]{Definition}
\newtheorem{rem}[theorem]{Remark}

\newcommand{\T}{\mathbb{T}}
\newcommand{\F}{\mathcal{F}}
\newcommand{\D}{\displaystyle}

\newcommand{\grad}{\nabla}

\newcommand{\dpt}{\partial_t}
\newcommand{\da}{\partial_{\alpha}}

\newcommand{\la}{\Lambda}

\newcommand{\al}{\alpha}
\newcommand{\ep}{\varepsilon}

\newcommand{\ztil}{\tilde{z}}
\newcommand{\dpa}{\partial^{\bot}_{\alpha}}

\newcommand{\pa}{\partial}

\newcommand{\vp}{\varphi}
\newcommand{\om}{\omega}
\newcommand{\si}{\sigma}
\newcommand{\be}{\beta}
\newcommand{\g}{\gamma}

\newenvironment{proof}{\begin{trivlist} \item[] {\em Proof:}}{\hfill $\Box$
                       \end{trivlist}}

\newenvironment{proofthm}[1]{\begin{trivlist} \item[] {\em Proof of Theorem \ref{#1}:}}{\hfill $\Box$
                       \end{trivlist}}

\makeatletter
\renewcommand*\l@section{\@dottedtocline{1}{0em}{1.5em}}
\renewcommand*\l@subsection{\@dottedtocline{2}{1.5em}{2.3em}}
\renewcommand*\l@subsubsection{\@dottedtocline{3}{3.8em}{3.7em}}
\makeatother

\numberwithin{equation}{section}

\begin{document}

\title{Finite time singularities for the free boundary incompressible Euler equations}
\date{\today}
\author{ Angel Castro, Diego C\'ordoba, Charles Fefferman, \\ Francisco Gancedo and Javier G\'omez-Serrano}

\maketitle


\begin{abstract}

In this paper, we prove the existence of smooth initial data for the 2D free boundary incompressible Euler equations (also known for some particular scenarios as the water wave problem), for which the smoothness of the interface breaks down in finite time into a splash singularity or a splat singularity.


\end{abstract}

\tableofcontents

\maketitle


\section{Introduction}
\label{Section1Introduction}

\subsection{Statement of the Problem}

In this paper, we prove that water waves in two space dimensions can form a singularity in finite time by either of two simple, natural scenarios, which we call a ``splash'' and a ``splat''.

The water wave equations (or 2D incompressible free boundary Euler equations) describe a system consisting of a water region $\Omega(t) \subset \mathbb{R}^{2}$ and a vacuum region $\mathbb{R}^{2} \setminus \Omega(t)$, evolving as a function of time $t$, and separated by a smooth interface

$$ \partial \Omega(t) = \{z(\al,t) : \al \in \mathbb{R}\}.$$

We write $\Omega^{1}(t) = \mathbb{R}^{2} \setminus \Omega(t)$, $\Omega^{2}(t) = \Omega(t)$. The fluid velocity $v(x,y,t) \in \mathbb{R}^{2}$ and the pressure $p(x,y,t) \in \mathbb{R}$ are defined for $(x,y) \in \Omega(t)$. The fluid is assumed to be incompressible and irrotational
\begin{equation}
\label{Charlie11}
\nabla \cdot v = 0, \quad \text{curl } v = 0 \quad \text{ in } \Omega(t),
\end{equation}

and to satisfy the 2D Euler equation
\begin{equation}
\label{Charlie12}
[\partial_t + (v \cdot \nabla_{x})]v(x,y,t) = -\nabla p(x,y,t) - (0,g) \quad \text{ in } \Omega(t),
\end{equation}

where $g > 0$ is a constant, and the term $(0,g)$ takes gravity into account.

Neglecting surface tension, we assume that the pressure satisfies
\begin{equation}
\label{Charlie13}
p = p^*(t) \quad \text{ at } \partial \Omega(t), \text{ where $p^*(t)$ is a function of $t$ alone.}
\end{equation}

Finally, we assume that the interface moves with the fluid, i.e.,
\begin{equation}
\label{Charlie14}
\partial_t z(\al,t) = v(z(\al,t),t) + c^{\#}(\al,t) \da z(\al,t),
\end{equation}

where $c^{\#}(\al,t)$ is an arbitrary smooth function of $\al,t$ (the choice of $c^{\#}$ affects only the parametrization of $\partial \Omega(t)$) and $z(\al,t) = (z_1(\al,t), z_2(\al,t))$.

At an initial time $t_0$, we specify the fluid region $\Omega(t_0)$ and the velocity $v(x,y,t_0)$ $((x,y) \in \Omega(t_0))$, subject to the constraint \eqref{Charlie11}. We then solve equations (\ref{Charlie11}-\ref{Charlie14}) with the given initial conditions, and we ask whether a singularity can form in finite time from an initially smooth velocity $v(\cdot,t_0)$ and fluid interface $\partial \Omega(t_0)$.

The water wave problem comes in three flavors:
\begin{itemize}
\item[$\bullet$] Asymptotically Flat: We may demand that $z(\al,t) - (\al,0) \to 0$ as $\alpha \to \pm \infty$.
\item[$\bullet$] Periodic: We may instead demand that $z(\al,t) - (\al,0)$ is a $2\pi$-periodic function of $\al$.
\item[$\bullet$] Compact: Finally, we may demand that $z(\al,t)$ is a $2\pi$-periodic function of $\al$.
\end{itemize}

To obtain physically meaningful solutions in the Asymptotically Flat and Periodic flavors, we demand that
$$ p(x,y,t) + gy = O(1) \quad \text{ in }\Omega(t)$$

and that

$$ \int_{\Omega(t)} |v(x,y,t)|^2dxdy < \infty \text{ (finite energy), }$$

where we regard $\Omega(t)$ as a subset of $\mathbb{T} \times \mathbb{R}$, $\mathbb{T} = \mathbb{R} / 2\pi \mathbb{Z}$, in the Periodic case.

In this paper, we restrict attention to periodic water waves, although our arguments can be easily modified to apply to the other flavors. (See Remark
\ref{Rem3Settings} below).

Let us summarize some of the previous work on water waves. We discuss the real-analytic case later in this introduction. The existence and Sobolev regularity of water waves for short time is due to S. Wu \cite{Wu:well-posedness-water-waves-2d}. Her proof applies to smooth interfaces that need not be graphs of functions, but \cite{Wu:well-posedness-water-waves-2d} assumes the arc-chord condition

$$ |z(\al,t) - z(\beta,t)| \geq c_{AC}|\al - \beta|, \quad \text{all $\al,\beta \in \mathbb{R}$}.$$

The constant $c_{AC} > 0$ is called the arc-chord constant, which may vary with time.

The issue of long-time existence has been treated in Alvarez-Lannes \cite{AlvarezSamaniego-Lannes:large-time-existence-water-waves}, where well-posedness over large time scales is shown, and several asymptotic regimes are justified. By taking advantage of the dispersive properties of the water-wave system, Wu \cite{Wu:almost-global-wellposedness-2d} proved exponentially large time of existence for small initial data.

In three space dimensions, Wu \cite{Wu:well-posedness-water-waves-3d} proved short-time existence; and Germain et al \cite{Germain-Masmoudi-Shatah:global-solutions-gravity-water-waves}, \cite{Germain-Masmoudi-Shatah:global-solutions-gravity-water-waves-annals} and Wu \cite{Wu:global-wellposedness-3d} proved existence for all time in the case of small initial data.

There are several important variants of the water wave problem. One can drop the assumption that the fluid is irrotational. See Christodoulou-Lindblad \cite{Christodoulou-Lindblad:motion-free-surface}, Lindblad \cite{Lindblad:well-posedness-motion}, Coutand-Shkoller \cite{Coutand-Shkoller:well-posedness-free-surface-incompressible}, Shatah-Zeng  \cite{Shatah-Zeng:geometry-priori-estimates}, Zhang-Zhang
\cite{Zhang-Zhang:free-boundary-3d-euler}. Lannes \cite{Lannes:well-posedness-water-waves} considered the case in which water is moving over a fixed bottom.
Ambrose-Masmoudi \cite{Ambrose-Masmoudi:zero-surface-tension-2d-waterwaves} considered the case where the equations include surface tension, and the limit where the coefficient of surface tension tends to zero. Lannes \cite{Lannes:stability-criterion} discussed the problem of two fluids separated by an interface with small non-zero surface tension.
Alazard et al. \cite{Alazard-Burq-Zuily:water-wave-surface-tension} took advantage of the dispersive properties of the equations to lower the regularity of the initial data.

See also the papers of C\'ordoba et al. \cite{Cordoba-Cordoba-Gancedo:interface-water-waves-2d} and Alazard-Metivier \cite{Alazard-Metivier:paralinearization}.

In the case of large data for the two-dimensional problem (\ref{Charlie11}-\ref{Charlie14}), Castro et al. in \cite{Castro-Cordoba-Fefferman-Gancedo-LopezFernandez:rayleigh-taylor-breakdown}, \cite{Castro-Cordoba-Fefferman-Gancedo-LopezFernandez:turning-waves} showed that there exist initial data for which the interface is the graph of a function, but after a finite time the water wave ``turns over'' and the interface is no longer a graph. For previous numerical simulations showing this turning phenomenon, see Baker et al. \cite{Baker-Meiron-Orszag:vortex-methods-free-surface} and Beale et al. \cite{Beale-Hou-Lowengrub:convergence-boundary-integral}.

Next, we describe a singularity that can form in water waves. We start by presenting what we believe based on numerical simulations; then, we explain what we can prove.

\begin{figure}[ht]
\centering
\subfigure[The initial water region $\Omega(t_0)$.]
{
\includegraphics[width=0.4\textwidth]{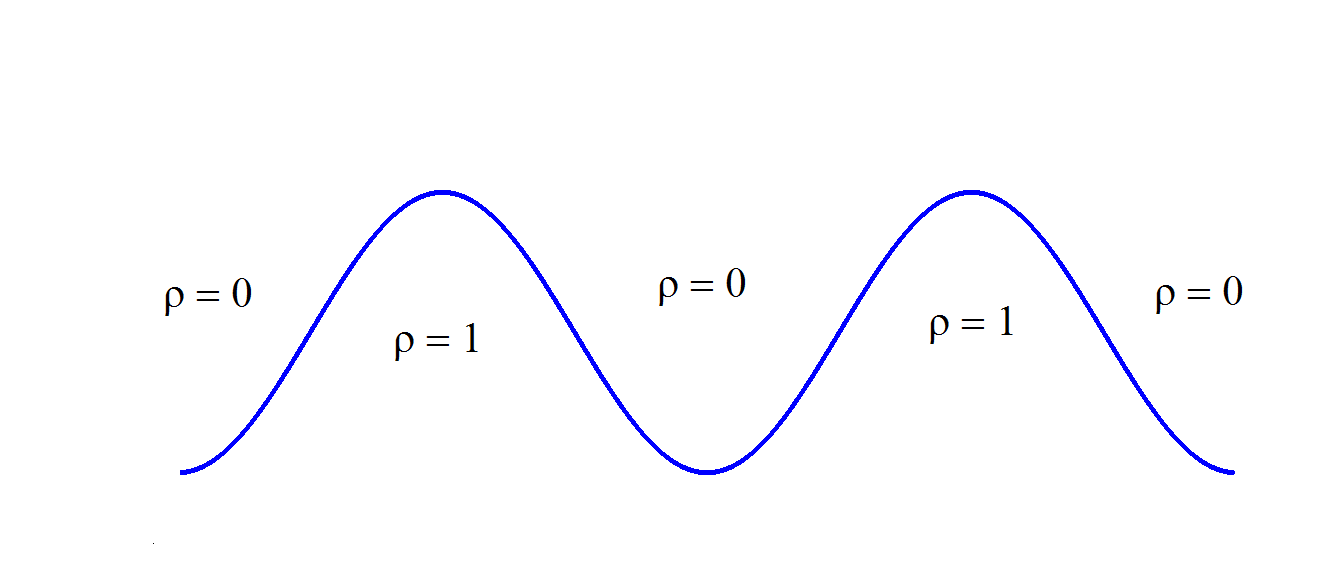}
\label{CharlieFig1a}
}
\subfigure[The water region $\Omega(t_1)$ at a later time $t_1$.]
{
\includegraphics[width=0.4\textwidth]{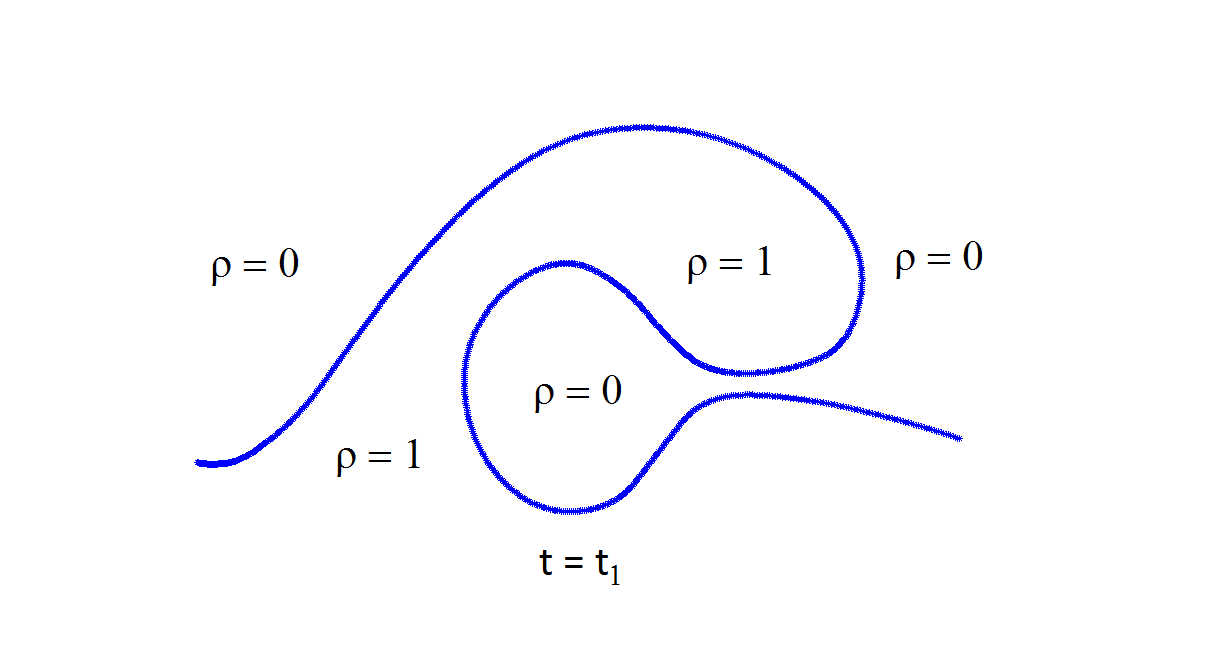}
\label{CharlieFig1b}
}
\subfigure[A ``splash'' forms at time $t_2 > t_1$.]
{
\includegraphics[width=0.4\textwidth]{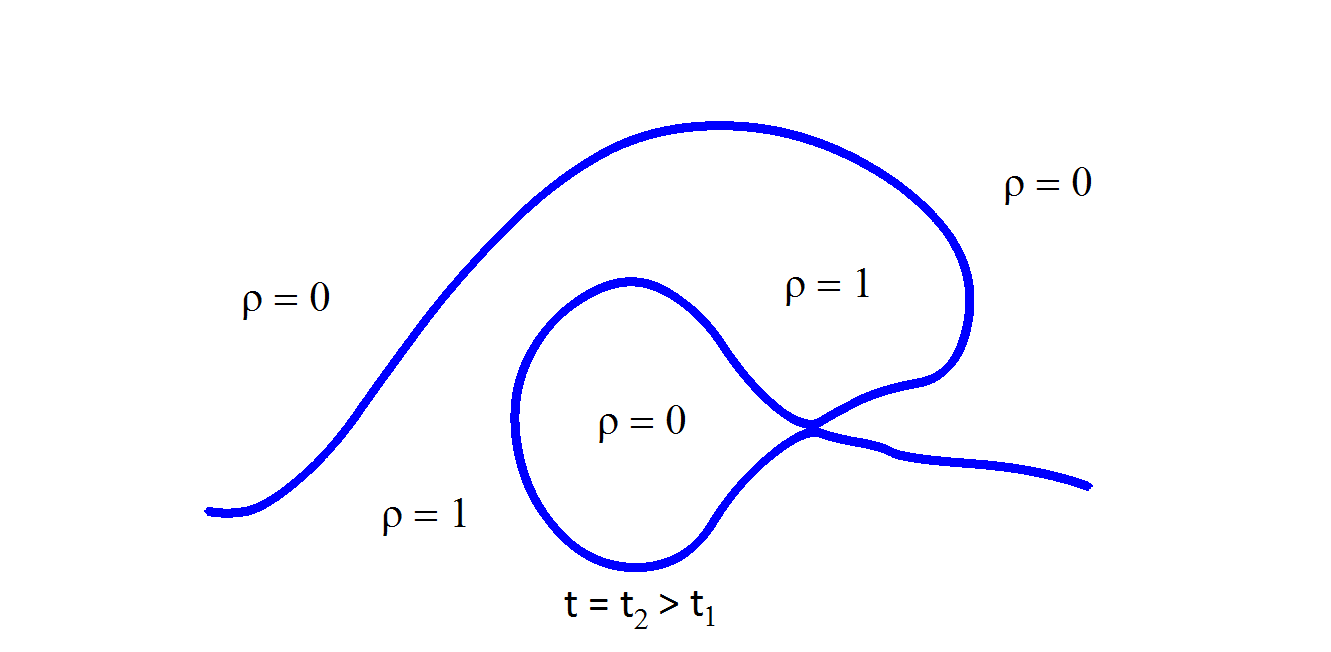}
\label{CharlieFig1c}
}
\caption{Evolution of a ``splash'' singularity.}
\label{splashFig}
\end{figure}

Our simulations show an initially smooth water wave, for which the fluid interface is a graph as in Figure \ref{CharlieFig1a}. At a later time $t_1$, the water wave has ``turned over'' as described in \cite{Castro-Cordoba-Fefferman-Gancedo-LopezFernandez:rayleigh-taylor-breakdown}, \cite{Castro-Cordoba-Fefferman-Gancedo-LopezFernandez:turning-waves}, i.e., the interface is no longer a graph. Finally, in Figure \ref{CharlieFig1c}, the fluid interface self-intersects at a single point \footnote{Here, we regard the fluid interface as sitting inside $\mathbb{T} \times \mathbb{R}$; recall that our water waves are $2\pi$-periodic under horizontal translation.}, but is otherwise smooth. We call this scenario a ``splash'', and we call the single point at which the interface self-intersects, the ``splash point''. Beyond the time $t_2$ pictured in Figure \ref{CharlieFig1c}, there is no physically meaningful solution of (\ref{Charlie11}-\ref{Charlie14}).

Note that the arc-chord condition holds for times $t < t_2$, but the arc-chord constant tends to zero as $t$ tends to $t_2$.

The numerics that led us to Figures \ref{CharlieFig1a}, \ref{CharlieFig1b} and \ref{CharlieFig1c} were performed using the method of Beale-Hou-Lowengrub \cite{Beale-Hou-Lowengrub:growth-rates-linearized}, with special modifications to maintain accuracy up to the splash. In this paper, we use the numerics only as motivation for conjectures, so we omit a detailed discussion of the algorithms used. Actual results from our simulations are shown in Figures \ref{PictureNonTilda}, \ref{PictureSplash} and \ref{PictureNonTildaZoom}. Figures \ref{splashFig} and \ref{splatFig} are cartoons.

Now let us explain what we can prove regarding the splash scenario. Recall that \cite{Castro-Cordoba-Fefferman-Gancedo-LopezFernandez:rayleigh-taylor-breakdown}, \cite{Castro-Cordoba-Fefferman-Gancedo-LopezFernandez:turning-waves} already proved that a water wave may start as in Figure \ref{CharlieFig1a} and later evolve to look like Figure \ref{CharlieFig1b}. In this paper, we prove that a water wave may start as in Figure \ref{CharlieFig1b}, and later form a splash, as in Figure \ref{CharlieFig1c}.

We would like to prove that an initially smooth water wave may start as in Figure \ref{CharlieFig1a}, then turn over as in Figure \ref{CharlieFig1b}, and finally produce a splash as in Figure \ref{CharlieFig1c}. To do so, our plan is to use interval arithmetic \cite{Moore-Bierbaum:methods-applications-interval-analysis} to produce a rigorous computer-assisted proof that, close to the approximate solution arising from our numerics, there exists an exact solution of (\ref{Charlie11}-\ref{Charlie14}) that ends in a splash. The stability result announced in \cite[Theorem 4.1]{Castro-Cordoba-Fefferman-Gancedo-GomezSerrano:splash-water-waves} is a first step in this direction. We are grateful to R. de la Llave for introducing us to interval arithmetic and demonstrating its power.

A variant of the splash singularity is shown in Figures \ref{CharlieFig2a} and \ref{CharlieFig2b}.

\begin{figure}[ht]
\centering
\subfigure[The initial water region]
{
\includegraphics[width=0.4\textwidth]{QualitativeSplashGoodA.png}
\label{CharlieFig2a}
}
\subfigure[The ``splat''.]
{
\includegraphics[width=0.4\textwidth]{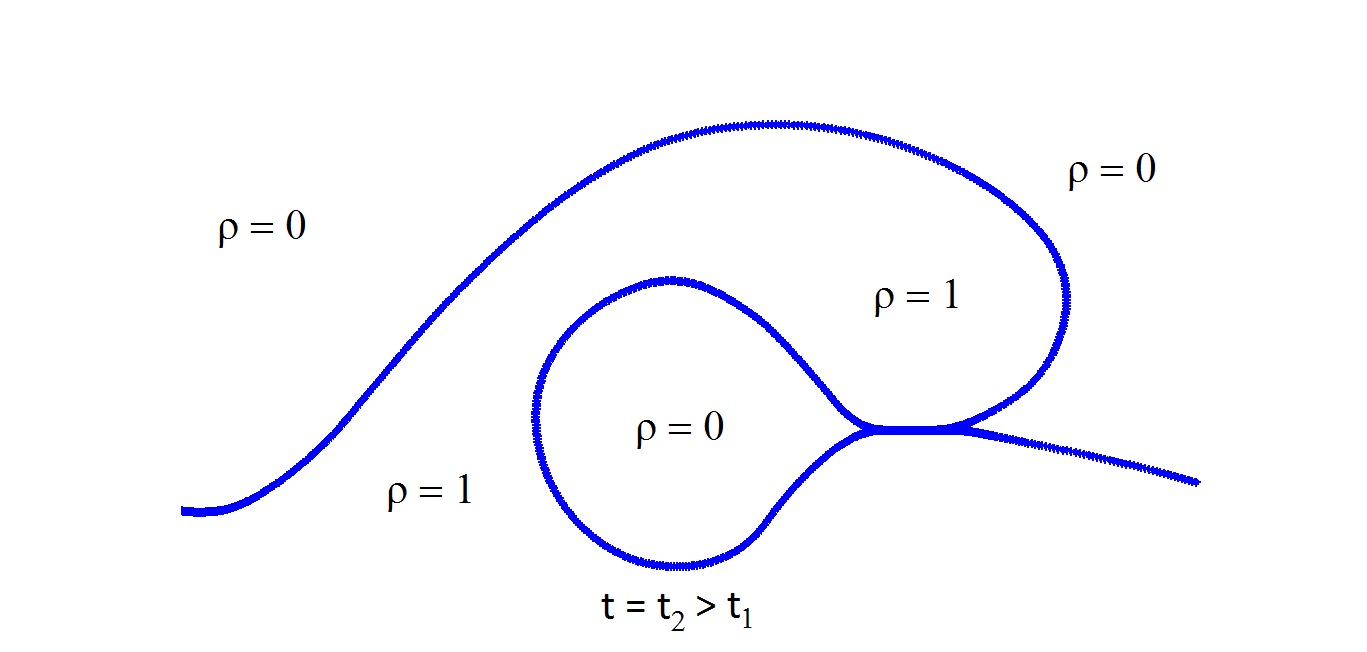}
\label{CharlieFig2b}
}
\caption{Evolution of a ``splat'' singularity.}
\label{splatFig}
\end{figure}

The water wave starts out smooth, as in Figure \ref{CharlieFig2a}, although the interface is not a graph. At a later time, the interface self-intersects along an arc, but is otherwise smooth. Again, no physically meaningful solution of (\ref{Charlie11}-\ref{Charlie14}) exists after the time depicted in Figure \ref{CharlieFig2b}. We call this scenario a ``splat''. In this paper, we prove that water waves can form a splat.

The stability theorem announced in \cite{Castro-Cordoba-Fefferman-Gancedo-GomezSerrano:splash-water-waves} shows that a sufficiently small perturbation of initial conditions that lead to the splash will again lead to a splash. We expect that the analogous statement for a splat is not true.

We make no claim that the splash and the splat are the only singularities that can arise in solutions of the water wave equation.

\subsection{Elementary Potential Theory}
\label{Section1BCharlie}

To formulate precisely our main results, and to explain some ideas from their proofs, we recall some elementary potential theory for irrotational divergence-free vector fields $v(x,y,t)$ defined on a region $\Omega(t) \subset \mathbb{R}^{2}$ with a smooth periodic boundary $\{z(\al,t): \al \in \mathbb{R}\}$ for fixed $t$. We assume that $v$ is smooth up to the boundary and $2\pi$-periodic with respect to horizontal translations. We suppose that $v$ has finite energy.

Such a velocity field $v$ may be represented in several ways:
\begin{itemize}
\item[$\bullet$] We may write $v = \nabla \phi$ for a velocity potential $\phi(x,y,t)$ defined on $\Omega(t)$ and smooth up to the boundary.
\item[$\bullet$] We may also write $v = \nabla^{\perp} \psi = (-\partial_{y} \psi, \partial_{x} \psi)$ for a stream function $\psi$, defined on $\Omega(t)$ and smooth up to the boundary.
\item[$\bullet$] The normal component of $v$ at the boundary, given by
$$ u_{normal}(\al,t) = v(z(\al,t),t) \cdot \frac{(\da z(\al,t))^{\perp}}{|\da z(\al,t)|}$$
uniquely specifies $v$ on $\Omega(t)$. Here, $u^{\perp} = (-u_2, u_1)$ for $u = (u_1, u_2) \in \mathbb{R}^{2}$, and we always orient $\partial \Omega(t)$ so that the normal vector $(\da z(\al,t))^{\perp}$ points into the vacuum region $\mathbb{R}^{2} \setminus \Omega(t)$.

The function $u_{normal}(\al,t)$ satisfies
$$ \int_{\mathbb{T}}u_{normal}(\al,t)|\da z(\al,t)|d\al = 0,$$
but is otherwise arbitrary.

Note that, because $v$ has finite energy, $\phi$ and $\psi$ are $2\pi$-periodic with respect to horizontal translations. (Without the assumption of finite energy, $\phi$ and $\psi$ could be ``periodic plus linear''). The functions $\phi$ and $\psi$ are conjugate harmonic functions.

\item[$\bullet$] There is another way to specify $v$, namely
\begin{equation}
\label{Charlie15}
v(x,y,t) = \frac{1}{2\pi}PV\int_{\mathbb{R}} \frac{(x-z_{1}(\beta,t), y - z_{2}(\beta,t))^{\perp}}{|(x-z_{1}(\beta,t), y - z_{2}(\beta,t))|^{2}} \om(\beta,t) d\beta, \quad ((x,y) \in \Omega(t))
\end{equation}
for a $2\pi$-periodic function $\om(\beta,t)$ called the ``vorticity amplitude''. See \cite{Baker-Meiron-Orszag:vortex-methods-free-surface}.
\end{itemize}

 Formula \eqref{Charlie15} holds only in the interior of $\Omega(t)$. Taking the limit as $(x,y) \to (z_{1}(\al,t),z_{2}(\al,t)) \in \partial \Omega(t)$ from the interior, we find that
 \begin{equation}
 \label{Charlie16}
 v(z(\al,t),t) = BR(z,\om)(\al,t) + \frac{1}{2}\om(\al,t) \frac{\da z(\al,t)}{|\da z(\al,t)|^{2}},
 \end{equation}

 where $BR$ denotes the Birkhoff-Rott integral
 \begin{equation}
 \label{Charlie17}
 BR(z,\om)(\al,t) = \frac{1}{2\pi}P.V.
 \int_{\mathbb{R}} \frac{(z_{1}(\al,t)-z_{1}(\beta,t), z_{2}(\al,t) - z_{2}(\beta,t))^{\perp}}{|(z_{1}(\al,t)-z_{1}(\beta,t), z_{2}(\al,t) - z_{2}(\beta,t))|^{2}} \om(\beta,t) d\beta.
 \end{equation}

 To see that $v$ may be represented as in \eqref{Charlie15}, \eqref{Charlie16}, one applies the Biot-Savart law to a discontinuous extension of $v$ from its initial domain $\Omega(t)$ to all of $\mathbb{R}^{2}$; to make the extension, one solves a Neumann problem in $\mathbb{R}^{2} \setminus \Omega(t)$.

 Thus, our velocity field $v$ admits multiple descriptions. Note that the description in terms of $\omega$ is significantly different from the descriptions in terms of $\phi$, $\psi$ and $u_{normal}$, because we bring in the Neumann problem on $\mathbb{R}^{2} \setminus \Omega(t)$ to justify \eqref{Charlie15} and \eqref{Charlie16}. When $\partial \Omega(t)$ is a ``splash curve'' as in Figure \ref{CharlieFig1c}, there is no problem defining $\phi$ and it is smooth up to the boundary, except that it can take two different values at the splash point, for obvious reasons. The same is true of $\psi$. Similarly, $u_{normal}(\al,t)$ continues to behave well.

 However, there is no reason to believe that $\om(\al,t)$ will be well-defined and smooth for a splash curve, since $\mathbb{R}^{2} \setminus \Omega(t)$ is a somewhat pathological domain. Our numerics suggest that $\max_{\al}|\om(\al,t)| \sim \frac{C}{t_s - t}$, where $t_s$ is the time of the splash.

 Let us apply the above potential theory to the water wave problem. A standard formulation of the problem \cite{Baker-Meiron-Orszag:vortex-methods-free-surface} takes $z(\al,t)$ and $\om(\al,t)$ as unknowns. This has the advantage that at least we know where our unknown functions are supposed to be defined, which is more than we can say for $\phi$, $\psi$ and $u$. Standard computations (see e.g. \cite[Section 2]{Cordoba-Cordoba-Gancedo:interface-water-waves-2d}) show that the water wave problem is equivalent to the following equations

 \begin{equation}
 \label{Charlie17b}
 \partial_t z(\al,t) = BR(z,\om)(\al,t) + \overline{c}(\al,t) \da z(\al,t)
 \end{equation}
 and
 \begin{align}
 \label{Charlie18}
 \partial_t \om(\al,t) = & -2\da z(\al,t) \cdot \partial_t BR(z,\om)(\al,t) \nonumber \\
 & - \da\left(\frac{|\om|^2}{4|\da z|^{2}}\right)(\al,t) + \da\left(\overline{c}(\al,t) \om(\al,t)\right) \nonumber \\
 & + 2\overline{c}(\al,t) \da z(\al,t) \cdot \da BR(z,\om)(\al,t) - 2g \da z_2(\al,t).
 \end{align}

 Here, $\overline{c}(\al,t)$ is a function that we may pick arbitrarily, since it influences only the parametrization of $\partial \Omega(t)$. For future reference, we write down several standard equations that follow from (\ref{Charlie11}-\ref{Charlie14}) by routine computation and elementary potential theory.
 \begin{align}
 \label{Charlie19}
 \Delta_{x} \phi(x,y,t) & = \Delta_{x} \psi(x,y,t) = 0 \quad \text{ in } \Omega(t); \quad \text{ $\phi$ and $\psi$ are harmonic conjugates.} \nonumber \\
 p(x,y,t) & = -\partial_t \phi(x,y,t) - \frac{1}{2}|\nabla \phi(x,y,t)|^{2} - gy \nonumber \\
 \left.\partial_{n} \psi\right|_{z(\al,t)} & = - \frac{\da \Phi(\al,t)}{|\da z(\al,t)|}, \quad \text{ where $\Phi(\al,t) = \phi(z(\al,t),t)$} \nonumber \\
  & \text{ and $n$ is the outward-pointing unit normal to $\partial\Omega(t)$.} \nonumber \\
 \psi(x+2\pi, y,t) & = \psi(x,y,t) \text{ and } \phi(x+2\pi,y,t) = \phi(x,y,t) \nonumber \\
 \psi(x,y,t) & = O(1) \text{ as } y \to -\infty \nonumber \\
 v & = \nabla^{\perp} \psi \text{ in $\Omega(t)$} \nonumber \\
 \partial_t z(\al,t) & = v(z(\al,t),t) + c(\al,t) \da z(\al,t) \nonumber \\
 \partial_t \Phi(\al,t) & = \frac{1}{2}|v(z(\al,t),t)|^{2} + c(\al,t)v(z(\al,t),t) \cdot \da z(\al,t) - gy(\al,t) + p^{*}(t).
 \end{align}

 We may write $u(\al,t)$ to denote $v(z(\al,t),t)$.

 \subsection{Main Results}

 Our main result is the following theorem. For the definition of a splash curve see Definition \ref{defsplash} in Section \ref{SectionDiscussionOmega}. The interface shown in Figure \ref{CharlieFig1c} is an example of a splash curve.
 \begin{theorem}
 \label{localexistencenontilde}
 Let $z^{0}(\al)$ be a splash curve, where the splash point is given by $z^{0}(\al_1) = z^{0}(\al_2)$, $\al_1 \neq \al_2$. Let $u^{0}_{normal}(\al)$ be a scalar function in $H^{4}(\mathbb{T})$, satisfying
 \begin{equation}
 \label{Charlie110}
 \int_{\mathbb{T}}u^{0}_{normal}(\al)|\da z^{0}(\al)|d\al = 0
 \end{equation}
 and
 \begin{equation}
 \label{Charlie111}
 u^{0}_{normal}(\al_1), u^{0}_{normal}(\al_2) < 0.
 \end{equation}
 Then there exist a time $T > 0$; a time-varying domain $\Omega(t)$ defined for $t \in [0,T]$ and a velocity field $v(x,y,t)$ defined for $(x,y) \in \Omega(t)$, $t \in [0,T]$ such that the following hold:
 \begin{equation}
 \label{Charlie112}
 \text{$\Omega(t)$ and $v(x,y,t)$ solve the water wave equations (\ref{Charlie11}-\ref{Charlie14}) for all $t \in [0,T]$.}
 \end{equation}
 \begin{align}
 \label{Charlie113}
& \text{$\partial \Omega(t)$ is given as a parametrized curve $\{z(\al,t): \al \in \mathbb{R}\}$}, \nonumber \\ & \text{with $z(\al,t) - (\al,0)$ $2\pi$-periodic in $\al$ for fixed $t$.}
 \end{align}
 \begin{align}
 \label{Charlie114}
 & z(\al,t) - (\al,0) \in C([0,T], H^{4}(\mathbb{T})) \text{ and } v(z(\al,t),t) \in C([0,T],H^{3}(\mathbb{T}))
 \end{align}
 \begin{equation}
 \label{Charlie115}
 z(\al,0) = z^{0}(\al) \text{ and $u_{normal}(\al,0) = u_{normal}^{0}(\al)$ for all $\al \in \mathbb{R}$.}
 \end{equation}
 \begin{align}
 \label{Charlie116}
& \text{For each $t \in [0,T]$, the curve $\partial \Omega(t)$ satisfies the arc-chord condition,} \nonumber \\
& \text{ but the arc-chord constant tends to zero as $t \to 0$.}
 \end{align}
 \end{theorem}

This result was announced in \cite{Castro-Cordoba-Fefferman-Gancedo-GomezSerrano:splash-water-waves}.

To prove that ``splash singularities'' can form, we note that the water wave equations are invariant under time reversal. Therefore, it is enough to exhibit a solution of the water wave equations that starts as a splash at time zero, but satisfies the arc-chord condition for each small positive time. Theorem \ref{localexistencenontilde} provides such solutions.

Since the curve touches itself it is not clear if the vorticity amplitude is well defined, although the velocity potential remains nonsingular. In order to get around this issue we will apply a transformation from the original coordinates to new ones which we will denote with a tilde. The purpose of this transformation is to be able to deal with the failure of the arc-chord condition. Let us consider the scenario in the periodic setting and then the transformation defined by  $\ztil(\al,t) \equiv P(z(\al,t))$ where $P$ is a conformal map that will be given as:

$$ P(z) = \left(\tan\left(\frac{z}{2}\right)\right)^{1/2}$$
 and the branch of the root will be taken in such a way that it separates the self-intersecting points of the interface. We will also need that the interface passes below the points $(\pm\pi,0)$ (or, equivalently, that those points belong to the vacuum region) in order for the tilde region to lie inside a closed curve and the vacuum region to lie on the outer part. See Figures \ref{PictureNonTilda} and \ref{PictureSplash}. Here $P(z)$ will refer to a 2 dimensional vector whose components are the real and imaginary parts of $P(z_1 + iz_2)$. Its inverse is given by

$$ P^{-1}(w) = i \log \left(\frac{1-iw^2}{1+iw^2}\right) = 2\arctan(w^{2}) \text{ for } w \in \mathbb{C}.$$

In this setting, $P^{-1}(z)$ will be well defined modulo multiples of $2\pi$.
\begin{rem}
Note that $P(z)$ is periodic such that $P(z+2k\pi) = P(z)$. Moreover, $P(z)$ is one-to-one in the water region and single-valued except at the splash point.
\end{rem}

\begin{rem}
Although the transformation to the tilde domain is convenient, the real reason for Theorem \ref{localexistencenontilde} is that
the potential theory inside the water region does not go bad as we approach the splash even
though it goes bad in the vacuum region.
\end{rem}

\begin{figure}
\centering
\includegraphics[scale=0.4]{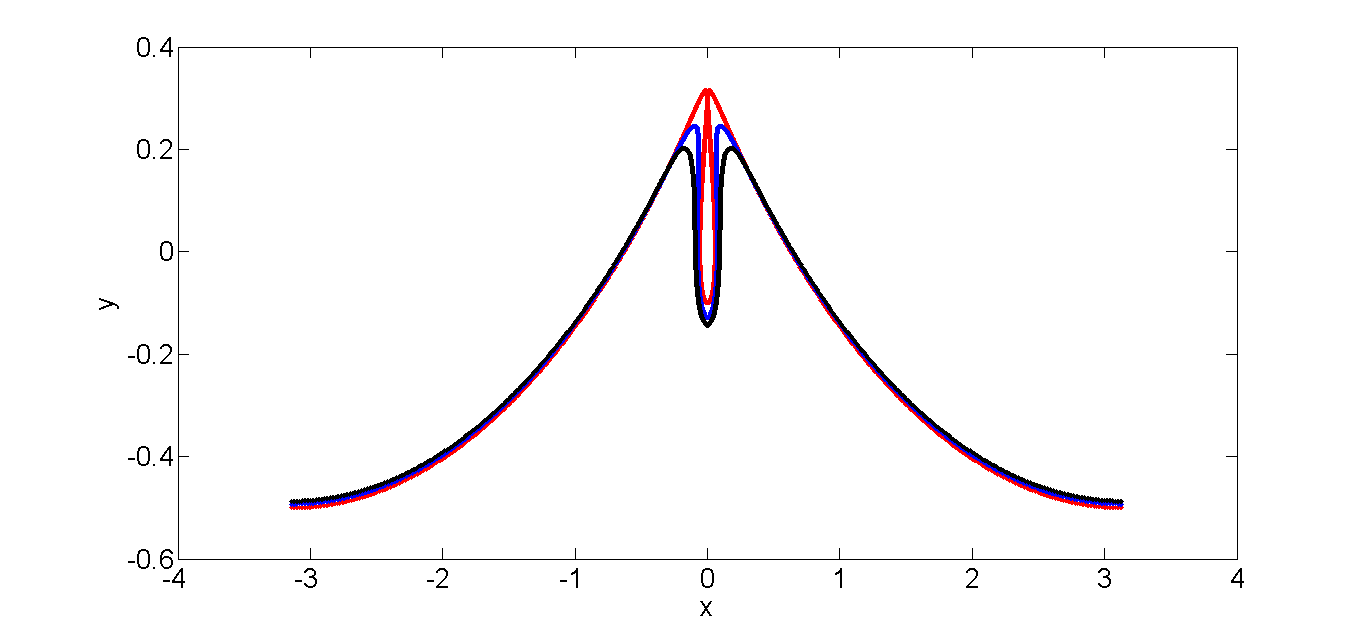}
\caption{Splash singularity at times $t = 0$ (Red - splash), $t = 4 \cdot 10^{-3}$ (Blue - turning) and $t = 7 \cdot 10^{-3}$ (Black - graph).}
\label{PictureNonTilda}
\end{figure}

 We define the following quantities:

$$ \tilde{\psi}(\tilde{x},\tilde{y},t) \equiv \psi(P^{-1}(\tilde{x},\tilde{y}),t), \quad \tilde{\phi}(\tilde{x},\tilde{y},t) \equiv \phi(P^{-1}(\tilde{x},\tilde{y}),t), \quad \tilde{v}(\tilde{x},\tilde{y},t) \equiv \nabla \tilde{\phi}(\tilde{x},\tilde{y},t), $$
$$ \tilde{\Phi}(\al,t) = \tilde{\phi}(\tilde{z}(\al,t),t), \quad \tilde{\Psi}(\al,t) = \tilde{\psi}(\tilde{z}(\al,t),t).$$

Also we define $\tilde{\Omega}(t) = P(\Omega(t))$. Let us note that since $\psi$ and $\phi$ are $2\pi$ periodic, the resulting $\tilde{\psi}$ and $\tilde{\phi}$ are well defined. We do not have problems with the harmonicity of $\tilde{\psi}$ or $\tilde{\phi}$ at the point which is mapped from minus infinity times $i$ (which belongs to the water region) by $P$ since $\phi$ and $\psi$ tend to finite limits at minus infinity times $i$. Also, the periodicity of $\phi$ and $\psi$ causes $\tilde{\phi}$ and $\tilde{\psi}$ to be continuous (and harmonic) at the interior of $P(\Omega^2(t))$.

Let us assume that there exists a solution of \eqref{Charlie19} and that we take $u_{normal} = \frac{\Psi_{\al}}{|z_{\al}|}$ such that $u_{normal}(\al_1), u_{normal}(\al_2) < 0$ for all $0 < t < T$, with $T$ small enough, thus $z(\al,t)$ satisfies the arc-chord condition and does not touch the removed branch from $P(w)$.

The system \eqref{Charlie19} in the new coordinates reads
\begin{align}
\label{tildastream}
\Delta \tilde{\psi}(\tilde{x},\tilde{y},t) & = 0 \quad \text{in $P(\Omega^2(t))$} \nonumber\\
\left.\partial_{n} \tilde{\psi}\right|_{\tilde{z}(\al,t)} & =- \frac{\tilde{\Phi}_{\al}(\al,t)}{|\tilde{z}_{\al}(\al,t)|} \nonumber\\
\tilde{v}&\equiv\nabla^\perp \tilde{\psi}\quad \text{in $P(\Omega^2(t))$}\nonumber\\
\tilde{z}_{t}(\al,t) & = Q^2(\al,t)\tilde{u}(\al,t) + c(\al,t)\tilde{z}_{\al}(\al,t) \nonumber\\
\tilde{\Phi}_{t}(\al,t) & = \frac{1}{2}Q^2(\al,t)|\tilde{u}(\al,t)|^2 + c(\al,t)\tilde{u}(\al,t) \cdot \tilde{z}_{\al}(\al,t) - gP^{-1}_2(\tilde{z}(\al,t))\nonumber\\
\tilde{z}(\alpha,0)&=\tilde{z}^0(\alpha)\nonumber\\
\tilde{\Phi}_\alpha(\alpha,0)&=\tilde{\Phi}_\alpha^0(\alpha) = \Phi_{\al}^{0}(\al),
\end{align}
where $\tilde{u}$ is the limit of the velocity coming from the fluid region in the tilde domain and
$$Q^2(\tilde{z}(\al,t),t) = \left|\frac{dP}{dw}(P^{-1}(\tilde{z}(\al,t)))\right|^{2}, \quad Q^2(\al,t) = \left|\frac{dP}{dw}(z(\al,t))\right|^{2}.$$

We can solve the Neumann problem in the complement of $\tilde{\Omega}(t)$. Therefore we can represent the velocity field $\tilde{v}$ in terms of a vorticity amplitude $\tilde{\omega}$.

We will see that $\tilde{z}$ and $\tilde{\omega}$ satisfy the following equations

\begin{align}\label{zeq}
\tilde{z}_{t}(\al,t) & = Q^2(\al,t)BR(\tilde{z},\tilde{\omega})(\al,t) + \tilde{c}(\al,t)\tilde{z}_{\al}(\al,t).
\end{align}

\begin{align}\label{eqomega}
\tilde{\omega}_{t}(\al,t) & = -2 \partial _{t} BR(\tilde{z},\tilde{\omega})(\al,t) \cdot \tilde{z}_{\al}(\al,t) - |BR(\tilde{z},\tilde{\omega})|^{2} \partial_{\al}Q^{2}(\al,t) \nonumber \\
& - \partial_{\al}\left(\frac{Q^2(\al,t)}{4}\frac{\tilde{\omega}(\al,t)^2}{|\tilde{z}_{\al}(\al,t)|^{2}}\right)
 + 2\tilde{c}(\al,t) \partial_{\al}BR(\tilde{z},\tilde{\omega}) \cdot \tilde{z}_{\al}(\al,t) \nonumber \\
& + \partial_{\al}\left(\tilde{c}(\al,t)\tilde{\omega}(\al,t)\right)
- 2g\partial_{\al}\left(P^{-1}_2(\tilde{z}(\al,t))\right).
\end{align}

\begin{rem}
Equations (\ref{zeq}-\ref{eqomega}) are analogous to (\ref{Charlie17b}-\ref{Charlie18}). In fact, if we  set $Q \equiv 1$ in (\ref{zeq}-\ref{eqomega}) we recover (\ref{Charlie17b}-\ref{Charlie18}).
\end{rem}

Our strategy will be the following: we will consider the evolution of the solutions in the tilde domain and then see that everything works fine in the original domain.

We will have to obtain the normal velocity once given the tangential velocity, and viceversa. To do this, we just have to notice that
$$ \tilde{\Phi}_{\al}(\al,t)  = \tilde{u}(\al,t) \cdot \tilde{z}_{\al}(\al,t) =
BR(\tilde{z},\tilde{\omega})\cdot \tilde{z}_{\al}(\al,t) + \frac{\tilde{\omega}(\al,t)}{2}.$$
From that, we can invert the equation (see \cite{Cordoba-Cordoba-Gancedo:interface-water-waves-2d}) and get $\tilde{\omega}$. Equation \eqref{Charlie16} in the tilde domain then tells us $\tilde{v}$ on the boundary $\partial \tilde{\Omega}(t)$.

We  now note that a solution of the system \eqref{tildastream} in the tilde domain gives rise to a solution of the system \eqref{Charlie19} in the non-tilde domain, by inverting the map $P$. In fact, this will be the implication used in Theorem \ref{localexistencenontilde} (finding a solution in the tilde domain, and therefore in the non-tilde).

\begin{rem}
\label{Rem3Settings}
It is likely that a similar argument works for the other two settings (closed contour and asymptotic to horizontal) by choosing an appropriate $P(w)$ that separates the singularity. For example, for the closed contour we could consider $P_{clo}(z) = \sqrt{z}$, taking the branch so that it separates the singularity, and for the asymptotic to horizontal scenario, it is enough to move the interface such that the water region is entirely contained in the lower halfplane (and the point $-i$ belongs to the vacuum region) and apply the relation $\frac{\tilde{z}+i}{\tilde{z}-i} = \sqrt{\frac{z+i}{z-i}}$.
\end{rem}

We now state the local existence results that lead to the proof of the existence of a splash singularity (Theorem \ref{localexistencenontilde}). To avoid the failure of the arc-chord condition, we will prove the local existence in the tilde domain. This can be done in two different settings, namely in the space of analytic functions and the Sobolev space $H^{s}$.

For the analytic version we define
$$ S_{r}= \{\al + i \eta, |\eta| < r\},$$
$$ \|f\|^2_{L^2(\partial S_r)}=\sum_{\pm}\int_{-\pi}^{\pi}|f(\al\pm ir)|^{2}d\al,$$
$$ \|f\|_r^2=\|f\|^2_{L^2(\partial S_r)}+\|\da^3f\|^2_{L^2(\partial S_r)}, $$

we consider the space
\begin{align*}
H^3(\partial S_r) = \left\{f \text{ analytic in } S_{r}, \|f\|_{r}^{2}<\infty, f \; 2\pi \text{-periodic}\right\}
\end{align*}
and we take $(z_1-\al, z_2,\Phi) \in (H^3(\partial S_r))^3 \equiv X_{r}$.

The first results concerning the Cauchy problem for
small data in Sobolev spaces near the equilibrium point are due to Craig \cite{Craig:existence-theory-water-waves}, Nalimov \cite{Nalimov:cauchy-poisson} and Yosihara \cite{Yosihara:gravity-waves}. Beale et al. \cite{Beale-Hou-Lowengrub:growth-rates-linearized} considered the Cauchy problem in the linearized version.
For local existence with small analytic data see Sulem-Sulem \cite{Sulem-Sulem:finite-time-analyticity-rt}. Our main results regarding local existence in the tilde domain are the following theorems:
\begin{theorem}[Local existence for analytic initial data in the tilde domain]
\label{localexistencetildeAnalytic}
Let $z^{0}(\al)$ be a splash curve and let $u^{0}\cdot\frac{z^0_\alpha}{|z^0_\alpha|}(\alpha)=\frac{\Phi^0_\alpha}{|z^0_\alpha|}(\alpha)$ be the initial tangential velocity such that
$$(z_{1}^{0}(\al) - \al, z_2^{0},(\al), \Phi^0(\alpha)) \in X_{r_0},$$
for some $r_0>0$, and  satisfying:
\begin{enumerate}
\item $ \displaystyle u_{normal}^{0}(\al_1) = u_{normal}(\al_1,0) < 0, \quad u_{normal}^{0}(\al_2) = u_{normal}(\al_2,0) < 0$
\item $\displaystyle \int_{\T} u^{0}_{normal}(\al)|\da z^0(\al)| d\al = 0$.
\end{enumerate}
Then there exist a finite time $T > 0$, $0<r<r_0$, a time-varying curve $\tilde{z}(\al,t)$ and a function $\tilde{\Phi}(\al,t)$ satisfying:
\begin{enumerate}
\item $P^{-1}(\tilde{z}_{1}(\al,t)) - \al, P^{-1}(\tilde{z}_{2}(\al,t))$ are $2\pi$-periodic,
\item $P^{-1}(\tilde{z}(\al,t))$ satisfies the arc-chord condition for all $t \in (0,T]$,
\end{enumerate}
and $\tilde{u}(\al,t)$ with
 $$(\tilde{z}_{1}(\al,t), \tilde{z}_2(\al,t), \tilde{\Phi}(\alpha,t))\in C([0,T],  X_{r})$$
 which provides a solution of the water wave equations \eqref{tildastream} with $\tilde{z}^0(\al)=P(z^0(\al))$ and $\tilde{u}(\al,0)\cdot(\tilde{z}_\alpha)^\perp(\alpha,0)=\tilde{u}^0(\al)\cdot(\tilde{z}^0)^\perp_\alpha(\alpha)$.

\end{theorem}

 The main tool in the proof is an abstract  Cauchy-Kowalewski theorem from \cite{Nirenberg:abstract-cauchy-kowalewski} and \cite{Nishida:theorem-Nirenberg}. For more details see \cite{Castro-Cordoba-Gancedo:naive-vortex-sheet}.

 For the proof of local existence in Sobolev spaces we will take the following $\tilde{c}(\al,t)$:

 \begin{align*}\tilde{c}(\al,t) & = \frac{\al+\pi}{2\pi}\int_{-\pi}^{\pi}(Q^{2}BR(\tilde{z},\tilde{\omega}))_\beta(\beta,t)\cdot\frac{\tilde{z}_{\beta}(\beta,t)}{|\tilde{z}_{\beta}(\beta,t)|^{2}}d\beta \\
& - \int_{-\pi}^{\al}(Q^{2}BR(\tilde{z},\tilde{\omega}))_\beta(\beta,t)\cdot\frac{\tilde{z}_{\beta}(\beta,t)}{|\tilde{z}_{\beta}(\beta,t)|^{2}}d\beta.
\end{align*}

This choice of $\tilde{c}$ will ensure that $|\tilde{z}(\al,t)|$ depends only on $t$. 
We will also define an auxiliary function $\tilde{\varphi}(\al,t)$ analogous to the one introduced in \cite{Beale-Hou-Lowengrub:growth-rates-linearized} (for the linear case) and \cite{Ambrose-Masmoudi:zero-surface-tension-2d-waterwaves} (nonlinear case) which helps us to bound several of the terms that appear:
\begin{equation}\label{fvarsect1}
\tilde{\varphi}(\al,t) = \frac{Q^2(\al,t)\tilde{\omega}(\al,t)}{2|\tilde{z}_{\al}(\al,t)|} - \tilde{c}(\al,t)|\tilde{z}_{\al}(\al,t)|.
\end{equation}

Then, we can prove the following theorem:

\begin{theorem}[Local existence for initial data in Sobolev spaces in the tilde domain]
\label{localexistencetilde}
In the setting of Section \ref{Section1BCharlie}, let $\tilde{z}^{0}(\al)$ be the image of a splash curve by the map $P$ parametrized in such a way that $|\da \tilde{z}^{0}(\al)|$ does not depend on $\alpha$, and such that $\tilde{z}_{1}^{0}(\al), \tilde{z}_2^{0}(\al) \in H^{4}(\T)$. Let $\tilde{\varphi}(\al,0) \in H^{3+\frac{1}{2}}(\T)$ be as in \eqref{fvarsect1} and let $\tilde{\omega}(\al,0) \in H^{2}(\T)$.
Then there exist a finite time $T > 0$, a time-varying curve $\tilde{z}(\al,t)  \in C([0,T];H^{4})$, and functions $\tilde{\omega}(\al,t) \in C([0,T];H^{2})$ and $\tilde{\varphi} \in C([0,T]; H^{3+\frac{1}{2}})$ providing a solution of the water wave equations (\ref{zeq} - \ref{eqomega}).
\end{theorem}

The proof is based on the adaptation of the local existence proof in \cite{Cordoba-Cordoba-Gancedo:interface-water-waves-2d} to the tilde domain.

%

Some of the relevant estimates from \cite{Cordoba-Cordoba-Gancedo:interface-water-waves-2d} obviously hold here as well, with essentially unchanged proofs. We state such results in Lemmas \ref{lemmaBR} and Lemmas \ref{lemmaomt}, $\ldots$, \ref{lemmasigmat} below; and refer the reader to the relevant sections of \cite{Cordoba-Cordoba-Gancedo:interface-water-waves-2d} for the proofs.

However, \cite{Cordoba-Cordoba-Gancedo:interface-water-waves-2d} contains several ``miracles'', i.e., complicated calculations and estimates that lead to simple favorable results for no apparent reason. To see that analogous ``miracles'' occur in our present setting, we have to go through the arguments in detail; see Lemmas \ref{lemmaenergyS} and \ref{lemmaphit}, $\ldots$, \ref{lemmaenergyphiS} below.

We have tried to make it possible to check the correctness of our arguments without extreme effort, and without undue repetitions from \cite{Cordoba-Cordoba-Gancedo:interface-water-waves-2d}.

It would be very interesting to understand a-priori why the ``miracles'' in this paper and in \cite{Ambrose-Masmoudi:zero-surface-tension-2d-waterwaves}, \cite{Cordoba-Cordoba-Gancedo:interface-water-waves-2d} occur. Presumably there is a simple, conceptual explanation, which at present we do not know.


At the end of Section \ref{SectionDiscussionOmega} we will define the notion of a ``splat curve''. The curve depicted in Figure \ref{CharlieFig2b} is an example of a splat curve.

In the statement of Theorem \ref{localexistencetilde}, we may take $\tilde{z}_{0}(\al)$ to be the image of a splat curve under $P$ rather than the image of a splash curve.

The proof of Theorem \ref{localexistencetilde} goes through for this case with trivial changes. Consequently, we obtain an analogue of Theorem \ref{localexistencetildeAnalytic}, with hypothesis 1 replaced by

\underline{Hypothesis $1'$:} $u_{normal}^{0} = u_{normal}(\al,0)$ is negative for all $\al \in I_{1} \cup I_{2}$, where $I_1$, $I_2$ are the intervals appearing in the definition of a splat curve in Section \ref{SectionDiscussionOmega}.

Just as Theorem \ref{localexistencetildeAnalytic} implies the formation of splash singularities for water waves, the above analogue of Theorem \ref{localexistencetildeAnalytic} for splat curves implies

\begin{corollary}[Splat singularity]
There exist solutions of the water wave system that collapse along an arc in finite time, but remain otherwise smooth.
\end{corollary}

\subsection{Further Results}
Here we mention some 
immediate 
consequences of our results which are relevant:
\begin{itemize}
\item[1.] (Splash and Splat singularities for $3D$ water waves) It is possible to extend our results to the periodic three dimensional setting by considering scenarios invariant under translation in one of the coordinate directions. While preparing the final revisions of this manuscript, we noticed that in a very recent arXiv posting \cite{Coutand-Shkoller:finite-time-splash}, Coutand-Shkoller consider additional 3D splash singularities.
\item[2.] (No gravity) The existence of a splash singularity can also be proved in the case where the gravity constant $g$ is equal to zero, as long as the Rayleigh-Taylor condition holds.
\end{itemize}

%
%

\begin{figure}
\centering
\includegraphics[scale=0.4]{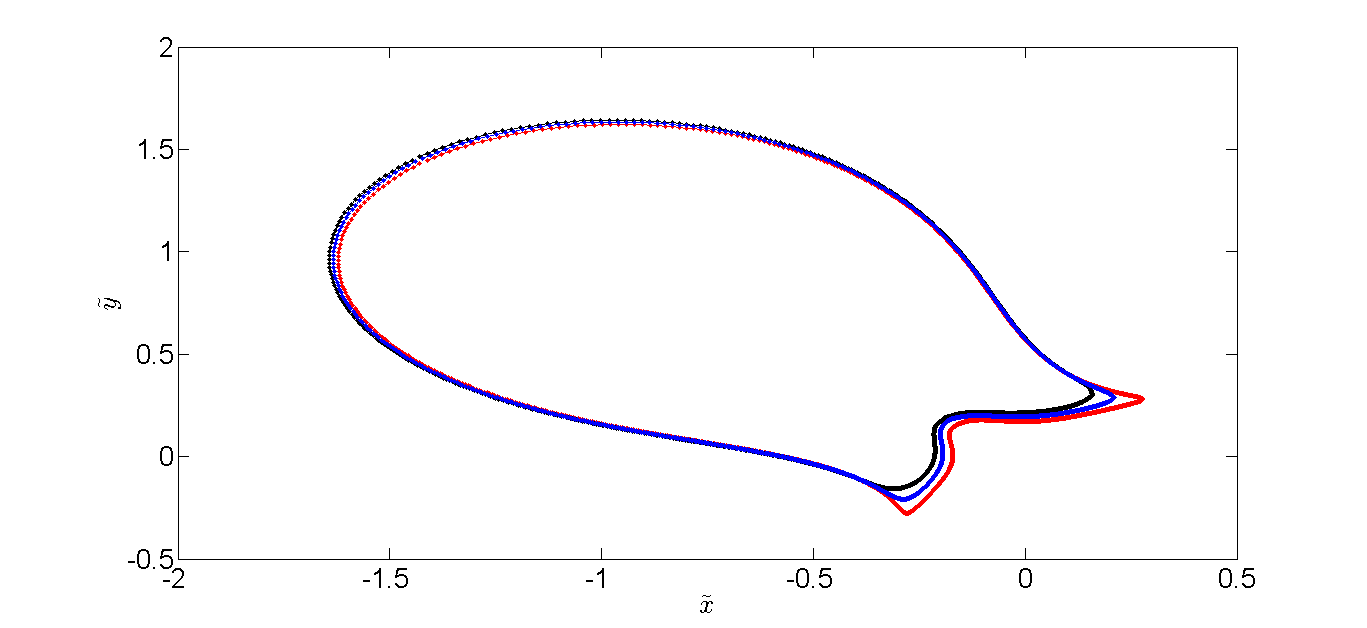}
\caption{Tilde domain at times $t = 0$ (Red - splash), $t = 4 \cdot 10^{-3}$ (Blue - turning) and $t = 7 \cdot 10^{-3}$ (Black - graph).}
\label{PictureSplash}
\end{figure}
\begin{figure}
\centering
\includegraphics[scale=0.4]{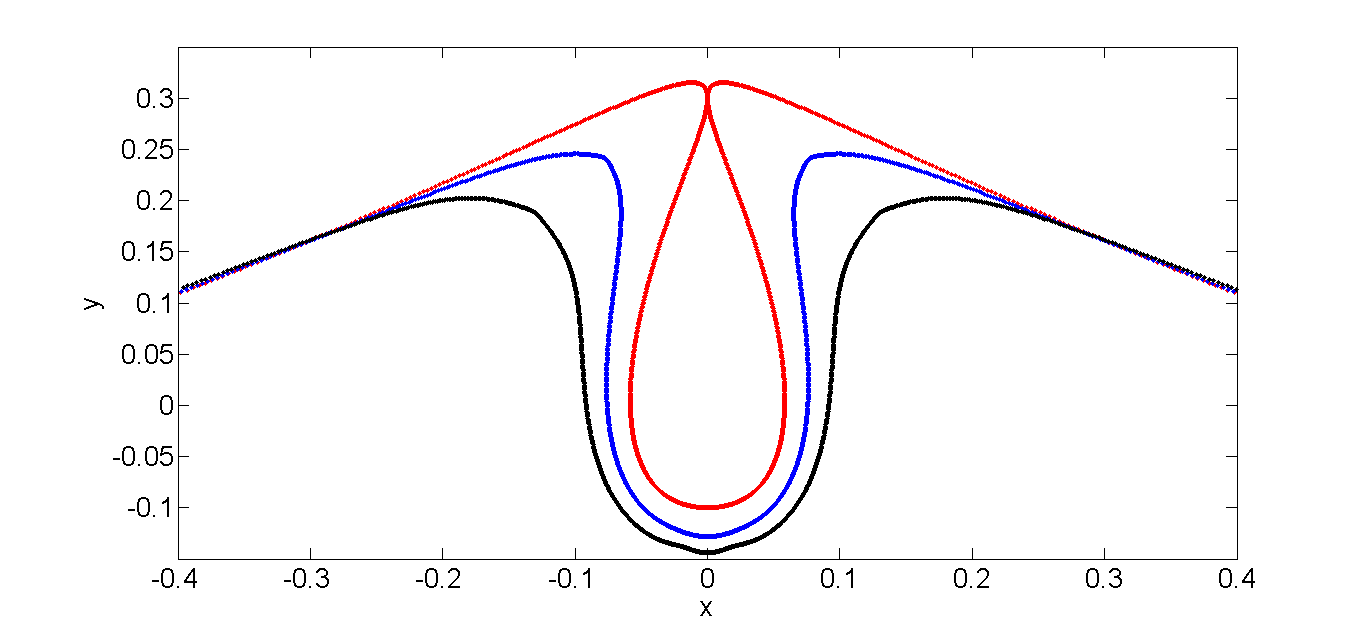}
\caption{Zoom of the splash singularity at times $t = 0$ (Red - splash), $t = 4 \cdot 10^{-3}$ (Blue - turning) and $t = 7 \cdot 10^{-3}$ (Black - graph).}
\label{PictureNonTildaZoom}
\end{figure}

\section{Splash curves: transformation to the tilde domain and back}
\label{SectionDiscussionOmega}
In this section we will rewrite the equations by applying a transformation from the original coordinates to new ones which we will denote by tilde. The purpose of this transformation is to be able to deal with the failure of the arc-chord condition.

For initial data we are interested in considering a self-intersecting curve in one point. More precisely, we will use as initial data \emph{splash curves} which are defined this way:

\begin{defi}
\label{defsplash}
We say that $z(\al) = (z_1(\al),z_2(\al))$ is a \emph{splash curve} if
\begin{enumerate}
\item $z_{1}(\al) - \al, z_2(\al)$ are smooth functions and $2\pi$-periodic.
\item $z(\al)$ satisfies the arc-chord condition at every point except at $\alpha_1$ and $\alpha_2$, with $\alpha_1 < \alpha_2$ where $z(\al_1) = z(\al_2)$ and $|z_{\al}(\al_1)|, |z_{\al}(\al_2)| > 0$. This means $z(\al_1) = z(\al_2)$, but if we remove either a neighborhood of $\al_1$ or a neighborhood of $\al_2$ in parameter space, then the arc-chord condition holds.
\item The curve $z(\alpha)$ separates the complex plane into two regions; a  connected water region and a vacuum region (not necessarily connected). The water region contains each point $x+iy$ for which y is large negative. We choose the parametrization such that the normal vector $n=\frac{(-\pa_\alpha z_2(\alpha), \pa_\alpha z_1(\alpha))}{|\pa_\alpha z(\alpha)|}$ points to the vacuum region. We regard the interface to be part of the water region.
\item We can choose a branch of the function $P$ on the water region such that the curve $\tilde{z}(\al) = (\tilde{z}_1(\al),\tilde{z}_2(\al)) = P(z(\al))$ satisfies:
\begin{enumerate}
\item $\tilde{z}_1(\al)$ and $\tilde{z}_2(\al)$ are smooth and $2\pi$-periodic.
\item $\tilde{z}$ is a closed contour.
\item $\tilde{z}$ satisfies the arc-chord condition.
\end{enumerate}
We will choose the branch of the root that produces that
$$ \lim_{y \to -\infty}P(x+iy) = -e^{-i \pi/4}$$
independently of $x$.
\item $P(w)$ is analytic at $w$ and $\frac{dP}{dw}(w) \neq 0$ if $w$ belongs to the interior of the water region. Furthermore, $(\pm \pi, 0)$ and $(0,0)$ belong to the vacuum region.
\item $\tilde{z}(\al) \neq q^l$ for $l=0,...,4$, where
\begin{equation}\label{points}
q^0=\left(0,0\right),\quad
q^1=\left(\frac{1}{\sqrt{2}},\frac{1}{\sqrt{2}}\right),\quad
q^2=\left(\frac{-1}{\sqrt{2}},\frac{1}{\sqrt{2}}\right),\quad
q^3=\left(\frac{-1}{\sqrt{2}}, \frac{-1}{\sqrt{2}}\right),\quad
q^4=\left(\frac{1}{\sqrt{2}}, \frac{-1}{\sqrt{2}}\right).
\end{equation}
\end{enumerate}
\end{defi}

From now on, we will always work with splash curves as initial data unless we say otherwise. Condition 6 will be used in the local existence theorems and can be proved to hold for short enough time as long as the initial condition satisfies it.  It is also immediate to check that the previous choice of $P$ transforms any periodic interface into a closed curve. Here are two examples of curves which are not splash curves (see Figure \ref{nosplash}).

\begin{figure}[h!]\centering
\includegraphics[width=0.45\textwidth, height=0.3\textwidth]{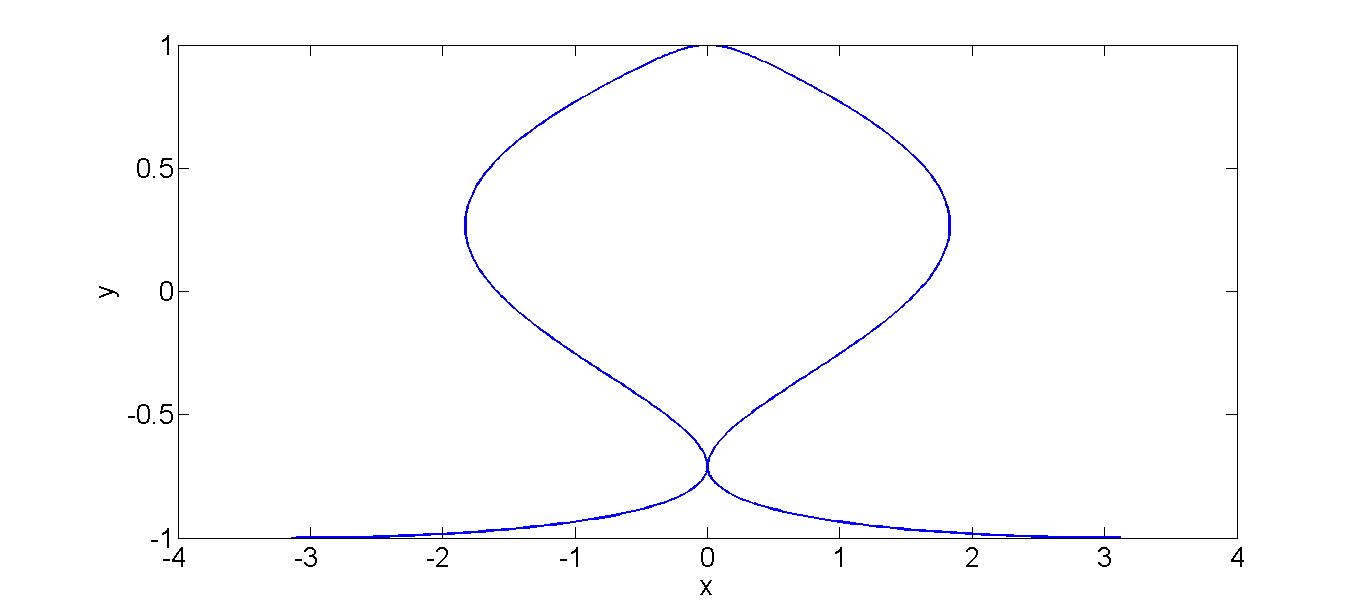}
\includegraphics[width=0.45\textwidth, height=0.3\textwidth]{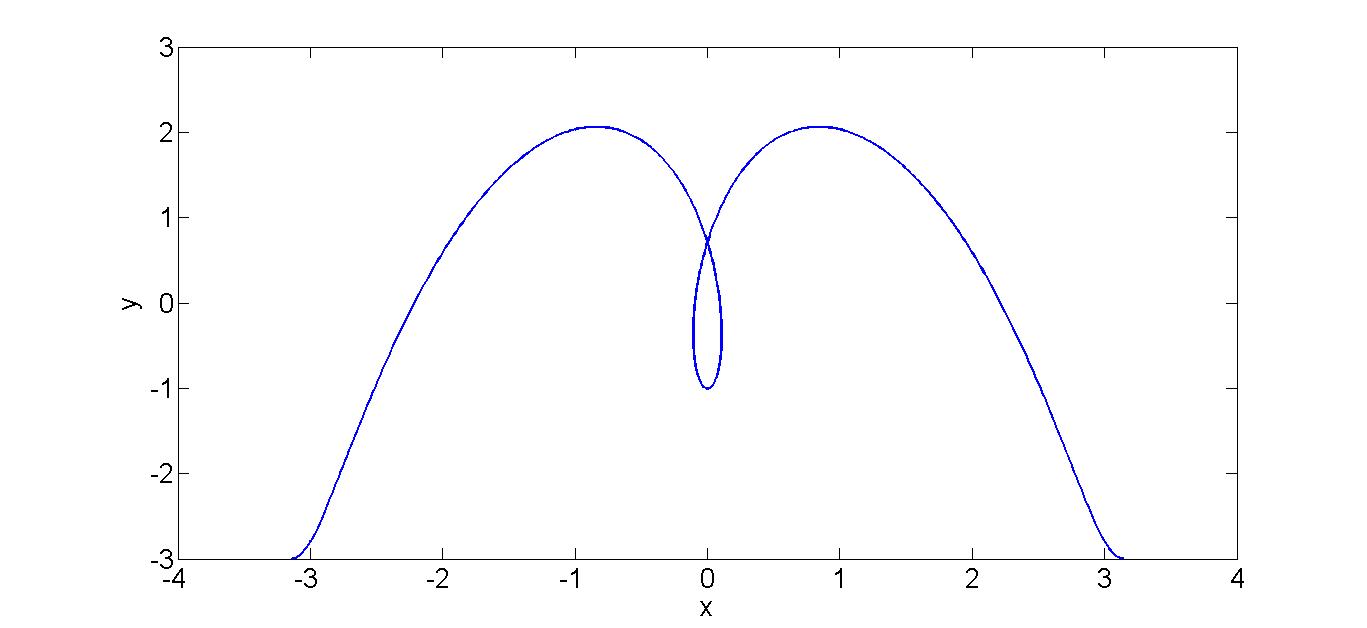}
\caption{Two examples of non splash curves.}
\label{nosplash}
\end{figure}

Now we will show a careful deduction of the equations in the tilde domain. From the definition of $\tilde{z}$ we have that

\begin{equation}
\label{deriv_ztilde_al}
\tilde{z}_{\al}(\al,t) = \nabla P(z(\al,t)) \cdot z_{\al}(\al,t)
\end{equation}
and
\begin{align}
 \tilde{z}_{t}(\al,t) & = \nabla P(z(\al,t)) \cdot z_t(\al,t) = \nabla P(z(\al,t)) \cdot (u(\al,t) + c(\al,t)z_{\al}(\al,t)) \nonumber\\
\label{deriv_ztilde_t}
 & = \nabla P(z(\al,t)) \cdot u(\al,t) + c \tilde{z}_{\al}(\al,t).
 \end{align}

Since $\phi = \tilde{\phi} \circ P$ and $v = \nabla \phi = \nabla(\tilde{\phi} \circ P)$, we obtain

\begin{equation}
\label{u-utilde}
v_{i} = \partial_{i} \phi = \partial_{i} (\tilde{\phi} \circ P) = \sum_{j}(\partial_{j} \tilde{\phi} \circ P)\frac{\partial P_{j}}{\partial x_i}
= \sum_{j} (\tilde{v}_{j} \circ P)\partial_{i} P_{j}.
\end{equation}

This implies that

\begin{equation}
\label{u-utilde2}
u(\al,t) = \nabla P(z(\al,t))^{T} \tilde{u}(\al,t).
\end{equation}

Plugging this into \eqref{deriv_ztilde_t} we get

\begin{equation}
\label{deriv_ztilde_t2}
 \tilde{z}_{t}(\al,t) = \nabla P(z(\al,t)) \cdot \nabla P(z(\al,t))^{T} \cdot \tilde{u}(\al,t) + c \tilde{z}_{\al}(\al,t).
\end{equation}

From the Cauchy-Riemann equations

\begin{equation}
\label{qscalar}
\nabla P(z(\al,t)) \cdot \nabla P(z(\al,t))^{T} = Q^2(\al,t) \cdot Id_{2}, \quad Q^2(\al,t) = \left|\frac{dP(z)}{dz}\right|^{2}.
\end{equation}

In this particular case, this means that

$$ Q^2(\al,t) = \left|\frac{1+\tilde{z}(\al,t)^4}{4\tilde{z}(\al,t)}\right|^{2}, \quad \tilde{z}(\al,t) = \tilde{z}_{1}(\al,t) + i \tilde{z}_{2}(\al,t).$$

Recall that $\tilde{\Phi}$ is the restriction of $\tilde{\phi}$ to the interface, i.e. $\tilde{\Phi}(\al,t) = \tilde{\phi}(\tilde{z}(\al,t),t)$. Then

\begin{equation}
\label{Phitilde}
\tilde{\Phi}(\al,t) = \tilde{\phi}(\tilde{z}(\al,t),t) = \phi (P^{-1}(\tilde{z}(\al,t)),t) = \phi(z(\al,t),t) = \Phi(\al,t)
\end{equation}

Thus, $\tilde{\Phi}$ satisfies

\begin{align}
\label{evolution-Phitilde}
\frac{\partial \tilde{\Phi}}{\partial t} & = \frac{1}{2}|u(\al,t)|^{2} + c(\al,t) u(\al,t) \cdot z_{\al}(\al,t) - gz_{2}(\al,t) \nonumber \\
& = \frac{1}{2}|\nabla P(z(\al,t))^{T} \cdot \tilde{u}(\al,t)|^{2} + c(\al,t) \tilde{u}(\al,t) \cdot \tilde{z}_{\al}(\al,t) - gP^{-1}_{2}(\tilde{z}(\al,t)),
\end{align}

where the subscript in the gravity term of the last line denotes the second component. Thus the system \eqref{Charlie19} in the new coordinates reads
\begin{align}
\label{tildastreamSection2}
\Delta \tilde{\psi}(x,y,t) & = 0 \quad \text{in $P(\Omega^2(t))$} \nonumber\\
\left.\partial_{n} \tilde{\psi}\right|_{\tilde{z}(\al,t)} & =- \frac{\tilde{\Phi}_{\al}(\al,t)}{|\tilde{z}_{\al}(\al,t)|} \nonumber\\
\tilde{v}&\equiv\nabla^\perp \tilde{\psi}\quad \text{in $P(\Omega^2(t))$}\nonumber\\
\tilde{z}_{t}(\al,t) & = Q^2(\al,t)\tilde{u}(\al,t) + c(\al,t)\tilde{z}_{\al}(\al,t) \nonumber\\
\tilde{\Phi}_{t}(\al,t) & = \frac{1}{2}Q^2(\al,t)|\tilde{u}(\al,t)|^2 + c(\al,t)\tilde{u}(\al,t) \cdot \tilde{z}_{\al}(\al,t) - gP^{-1}_2(\tilde{z}(\al,t)) + p^{*}(t)\nonumber\\
\tilde{z}(\alpha,0)&=\tilde{z}^0(\alpha)\nonumber\\
\tilde{\Phi}_\alpha(\alpha,0)&=\tilde{\Phi}_\alpha^0(\alpha) = \Phi_{\al}^{0}(\al).
\end{align}

We have seen that $\tilde{v}$ can be represented in the form

$$ \tilde{v}(\tilde{x},\tilde{y},t) = \nabla^{\perp} \tilde{\psi}(\tilde{x},\tilde{y},t) = \frac{1}{2\pi}P.V\int_{-\pi}^{\pi}\frac{(\tilde{x} - \tilde{z}_1(\al,t),\tilde{y} - \tilde{z}_2(\al,t))^{\perp}}{|(\tilde{x} - \tilde{z}_1(\al,t),\tilde{y} - \tilde{z}_2(\al,t))|^2}\tilde{\omega}(\al,t)d\al.$$
Taking limits from the fluid region we obtain
$$ \tilde{u}(\al,t) = BR(\tilde{z},\tilde{\omega}) + \frac{\tilde{\omega}}{2|\tilde{z}_{\al}|^2}\tilde{z}_{\al}.$$
The evolution of $\tilde{\omega}$ is calculated in the following way. First, let us recall the equations
\begin{align}
\label{equations-hou-tilde}
\tilde{z}_{t}(\al,t) & = Q^2(\al,t)\tilde{u}(\al,t) + c(\al,t)\tilde{z}_{\al}(\al,t) \nonumber\\
\tilde{\Phi}_{t}(\al,t) & = \frac{1}{2}Q^2(\al,t)|\tilde{u}(\al,t)|^2 + c(\al,t)\tilde{u}(\al,t) \cdot \tilde{z}_{\al}(\al,t) - gP^{-1}_2(\tilde{z}(\al,t)) \nonumber\\
\tilde{\Phi}_{\al}(\al,t) & = \tilde{u}(\al,t) \cdot \tilde{z}_{\al}(\al,t) \nonumber \\
\tilde{z}(\alpha,0)&=\tilde{z}^0(\alpha)\nonumber\\
\tilde{\Phi}_\alpha(\alpha,0)&=\tilde{\Phi}_\alpha^0(\alpha) = \Phi_{\al}^{0}(\al).
\end{align}

Substituting the expression for $\tilde{u}(\al,t)$ and performing the change $\tilde{c}(\al,t) = c(\al,t) + \frac{1}{2}Q^2(\al,t)\frac{\tilde{\om}(\al,t)}{|\tilde{z}_{\al}(\al,t)|^{2}}$ we obtain

\begin{align}\label{ikea}
\tilde{z}_{t}(\al,t) & = Q^2(\al,t)BR(\tilde{z},\tilde{\om})(\al,t) + \tilde{c}(\al,t)\tilde{z}_{\al}(\al,t) \nonumber\\
\tilde{\Phi}_{\al}(\al,t) & = BR(\tilde{z},\tilde{\om})(\al,t) \cdot \tilde{z}_{\al}(\al,t) + \frac{1}{2}\tilde{\om}(\al,t) \nonumber\\
\tilde{\Phi}_{t}(\al,t) & = \frac{1}{2}Q^2(\al,t)|\tilde{u}(\al,t)|^2 + c(\al,t)\tilde{u}(\al,t) \cdot \tilde{z}_{\al}(\al,t) - gP^{-1}_{2}(\tilde{z}(\al,t)) \nonumber\\
& = \frac{1}{2}Q^2(\al,t)|BR(\tilde{z},\tilde{\om})(\al,t)|^2 - \frac{Q^2(\al,t)}{8}\frac{\tilde{\om}(\al,t)^2}{|\tilde{z}_{\al}(\al,t)|^{2}}\nonumber \\
& + \tilde{c}(\al,t)BR(\tilde{z},\tilde{\om}) \cdot \tilde{z}_{\al}(\al,t)
+ \frac{1}{2}\tilde{c}(\al,t)\tilde{\om}(\al,t) - gP^{-1}_{2}(\tilde{z}(\al,t)).
\end{align}

On the one hand, by taking derivatives with respect to $t$ in the second equation follows
\begin{align}
\tilde{\Phi}_{\al t}(\al,t) & = \partial _{t} BR(\tilde{z},\tilde{\om})(\al,t) \cdot \tilde{z}_{\al}(\al,t) + BR(\tilde{z},\tilde{\om})(\al,t) \cdot \tilde{z}_{\al t}(\al,t) + \frac{\tilde{\om}_{t}(\al,t)}{2} \nonumber \\
& = \partial _{t} BR(\tilde{z},\tilde{\om}) \cdot \tilde{z}_{\al}(\al,t) + |BR(\tilde{z},\tilde{\om})|^{2} \partial_{\al}Q^{2}(\al,t)\nonumber \\
& + Q^{2}(\al,t) BR(\tilde{z},\tilde{\om}) \cdot \partial_{\al} BR(\tilde{z},\tilde{\om})
+ \tilde{c}_{\al}(\al,t) BR(\tilde{z},\tilde{\om}) \cdot \tilde{z}_{\al}(\al,t) \nonumber \\
& + \tilde{c}(\al,t) BR(\tilde{z},\tilde{\om}) \cdot \tilde{z}_{\al \al}(\al,t)
+ \frac{\tilde{\om}_{t}(\al,t)}{2}.
\end{align}

On the other, taking derivatives with respect to $\al$ in the third equation in \eqref{ikea} yields
\begin{align}
\tilde{\Phi}_{\al t}(\al,t) & = \frac{1}{2}|BR(\tilde{z},\tilde{\om})|^{2} \partial_{\al}Q^{2}(\al,t) + Q^{2}(\al,t) BR(\tilde{z},\tilde{\om}) \cdot \partial_{\al} BR(\tilde{z},\tilde{\om}) \nonumber \\
& - \frac{1}{2}\partial_{\al}\left(\frac{Q^2(\al,t)}{4}\frac{\tilde{\om}(\al,t)^2}{|\tilde{z}_{\al}(\al,t)|^{2}}\right)
 + \tilde{c}_{\al}(\al,t) BR(\tilde{z},\tilde{\om}) \cdot \tilde{z}_{\al}(\al,t)\nonumber \\
& + \tilde{c}(\al,t) \partial_{\al}BR(\tilde{z},\tilde{\om}) \cdot \tilde{z}_{\al}(\al,t)+ \tilde{c}(\al,t) BR(\tilde{z},\tilde{\om}) \cdot \tilde{z}_{\al \al}(\al,t)\nonumber \\
&  + \frac{1}{2}\partial_{\al}\left(\tilde{c}(\al,t)\tilde{\om}(\al,t)\right) - \partial_{\al}\left(gP^{-1}_{2}(\tilde{z}(\al,t))\right).
\end{align}

Combining both equations, we find that
\begin{align}
\tilde{\om}_{t}(\al,t) & = -2 \partial _{t} BR(\tilde{z},\tilde{\om})(\al,t) \cdot \tilde{z}_{\al}(\al,t) - |BR(\tilde{z},\tilde{\om})|^{2} \partial_{\al}Q^{2}(\al,t) \nonumber \\
& - \partial_{\al}\left(\frac{Q^2(\al,t)}{4}\frac{\tilde{\om}(\al,t)^2}{|\tilde{z}_{\al}(\al,t)|^{2}}\right)
 + 2\tilde{c}(\al,t) \partial_{\al}BR(\tilde{z},\tilde{\om}) \cdot \tilde{z}_{\al}(\al,t) \nonumber \\
& + \partial_{\al}\left(\tilde{c}(\al,t)\tilde{\om}(\al,t)\right)
- 2\partial_{\al}\left(gP^{-1}_{2}(\tilde{z}(\al,t))\right).
\end{align}

We will proceed in the following way: we will consider the evolution of the solutions in the tilde domain and see that everything works fine in the original domain. For example, the sign condition on the normal vectors in the non-tilde domain has an equivalent form in the tilde domain (i.e. the two normal components have negative sign).

In the non-tilde domain, this implies that the interface moves away from the branch removed from the square root, and therefore the interface touches neither the branch cut nor the conflictive points $q^{l}$ (see Condition 6 in Definition \ref{defsplash}). Hence $P$ and $P^{-1}$ will be well defined and one-to-one. (See Figure \ref{goodbad}).

Let us note that getting $\phi = \tilde{\phi} \circ P$ is not a problem since $\phi$ is bounded and harmonic. Moreover, as $\tilde{v} = \nabla^{\perp} \tilde{\psi}$ and
$$ v = \nabla P^{T} (\tilde{v} \circ P)$$
and $\nabla P$ has exponential decay at infinity, the velocity $v$ belongs to $L^{2}(\Omega^2(t) \cap [-\pi,\pi] \times \mathbb{R})$.

\begin{rem}
$\Psi, \Phi, u$ and $z$ have easy transformations to the tilde domain but $\om$ has not.
\end{rem}

We would like to discuss what happens to the amplitude of the vorticity $\om$ in the non-tilde domain as the curve approaches the splash.

If the vorticity belongs to $C([0,T_{splash}],C^{\delta}(\T))$, then the normal  velocity should be continuous at the
splash point and therefore the normal component of the restriction of the velocity to the curve from the water region cannot have
the same sign at $z(\al_1)$ and $z(\al_2)$ (see Theorem \ref{localexistencenontilde}). This means that the $C^{\delta}-$norm of the amplitude of the
vorticity becomes unbounded at the time of the splash.

We illustrate this phenomenon by plotting $1/\max{|\om|}$ (see Figure \ref{vorticityblowup}), where the blue curve is the calculated $\om$ and the red curve is a potential fitting to the data as numerical instabilities don't allow us to compute $\om$ with enough precision when we are in the regime which is close to the splash. Time has been reversed so that the splash occurs at time $t = 0$ and the interface separates from itself at $t > 0$.

\begin{figure}[h!]\centering
\includegraphics[scale=0.3]{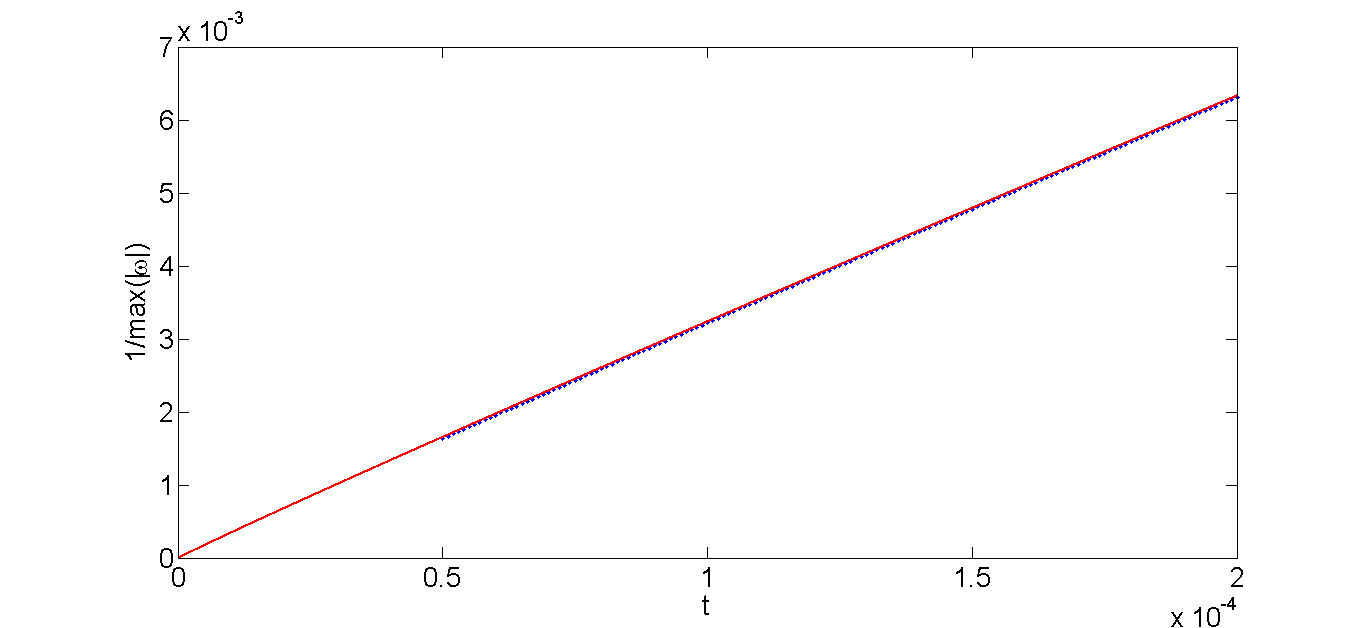}
\caption{Vorticity amplitude in the nontilde domain. The vorticity reaches infinity at a rate of approximately $\frac{1}{(T_{splash}-t)^{0.966}} \approx \frac{1}{(T_{splash}-t)}$. The fit is given by $F = 23.72 \cdot t^{0.966} - 1.476 \cdot 10^{-6}$.}
\label{vorticityblowup}
\end{figure}

We also have performed numerical simulations in order to get a blowup rate for the arc-chord condition. As in Figure \ref{vorticityblowup}, we plot the inverse of the arc-chord constant. The blue curve is made by the calculated points and the red curve is the interpolating one. We see a very good fitting. Time follows the same convention as before and the numerical evidence indicates a blowup of the arc-chord as $\frac{1}{T_{splash}-t}$. The results can be seen in Figure \ref{arcchordblowup}.

\begin{figure}[h!]\centering
\includegraphics[scale=0.3]{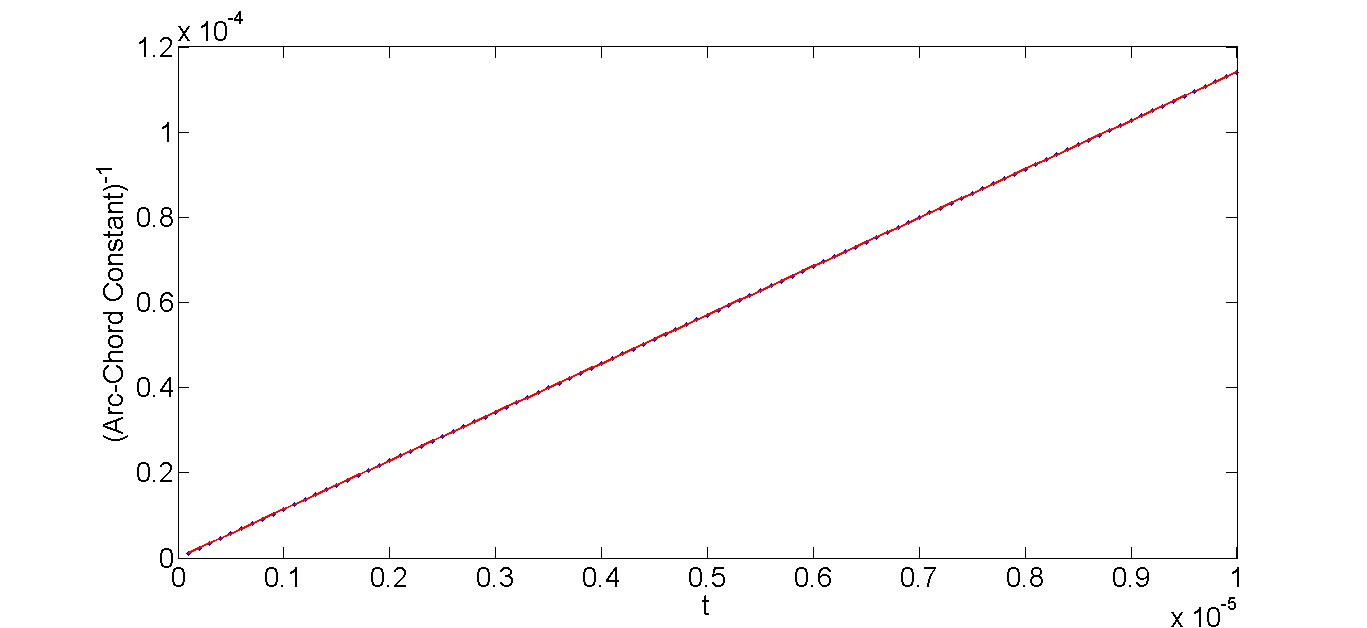}
\caption{Arc-chord condition in the non-tilde domain. The arc-chord reaches infinity at a rate of approximately $\frac{1}{(T_{splash}-t)}$. The fit is given by $F = 11.41 \cdot t + 5.104 \cdot 10^{-9}$.}
\label{arcchordblowup}
\end{figure}

We also kept track of the energy conservation. If we consider the following energy (not to be confused with the one in Section \ref{SectionLocalSobolev}):

\begin{equation}
E_S(t) = \frac{1}{2}\int_{\Omega^2_f(t)}|v(x,y,t)|^2dxdy + \frac{1}{2}\int_{-\pi}^{\pi}g(z_2(\alpha,t))^2\da z_{1}(\alpha,t) d\alpha \equiv E_k(t) + E_p(t)
\end{equation}

where $z(\alpha,t) = (z_{1}(\alpha,t), z_{2}(\alpha,t)), u(\alpha,t) = v(z(\alpha,t),t)$, and $\Omega^{2}_f(t)=\Omega^{2}(t) \cap [-\pi,\pi] \times \mathbb{R}$ is a fundamental domain in the water region in a period, then we can see that the energy is conserved; this is a check of the accuracy of our numerics.

\begin{align}
\frac{dE_k(t)}{dt} & = \int_{\Omega^2_f(t)}v(x,y,t)(v_t(x,y,t) + v(x,y,t)\cdot \nabla v(x,y,t))dxdy \nonumber \\
& = \int_{\Omega^2_f(t)}v(x,y,t)(-\nabla p(x,y,t) - g(0,1))dxdy \nonumber \\
& = -\int_{\Omega^2_f(t)}v(x,y,t)(\nabla (p(x,y,t) + gy))dxdy \nonumber \\
& =-
\int_{\partial (\Omega^{2}_f(t))}v(x,y,t)\cdot \overrightarrow{n} gy ds \nonumber \\
& = -\int_{-\pi}^{\pi}gz_{2}(\alpha,t)u(\alpha,t)\cdot \da z^\bot(\alpha,t)d\alpha
\end{align}
where we have used the incompressibility of the fluid ($\nabla \cdot v = 0$) and the continuity of the pressure on the interface ($\left.p^{*}(t)\right|_{\partial\Omega^{2}(t)} = 0$). Next
\begin{align}
\frac{dE_p(t)}{dt} & = \int_{-\pi}^{\pi}gz_2(\alpha,t)\partial_t z_2(\alpha,t)\da z_{1}(\alpha,t)d\alpha + \frac{1}{2}\int_{-\pi}^{\pi}g(z_{2}(\alpha,t))^2
\dpt \da  z_{1}(\alpha,t)d\alpha \nonumber \\
& = \int_{-\pi}^{\pi}gz_2(\alpha,t)\partial_t z_2(\alpha,t)\da z_{1}(\alpha,t)d\alpha - \int_{-\pi}^{\pi}gz_{2}(\alpha,t)\da z_{2}(\alpha,t)\partial_t z_1(\alpha,t)d\alpha \nonumber \\
& = \int_{-\pi}^{\pi}gz_2(\alpha,t)u(\alpha,t)\cdot \da z^\bot(\alpha,t)d\alpha.
\end{align}

This proves that the energy is constant. Note that:

\begin{align}
\int_{\Omega^{2}_f(t)}|v(x,y,t)|^2dxdy & = \int_{\Omega^{2}_f(t)}|\nabla \phi(x,y,t)|^{2}dxdy \nonumber \\
& = -\int_{\Omega^{2}_f(t)}\phi(x,y,t)\Delta \phi(x,y,t)dxdy \nonumber \\
& + \int_{\partial (\Omega^{2}_f(t))}\phi(x,y,t) \nabla \phi(x,y,t) \cdot \overrightarrow{n}dxdy \nonumber \\
& \stackrel{\Delta \phi = 0}{=} \int_{\partial (\Omega^{2}_f(t))}\phi(x,y,t) \nabla \phi(x,y,t) \cdot \overrightarrow{n}dxdy
\end{align}

so the numerical calculation is restricted to the values at the boundary. We observe that the energy of our system is conserved, as we have
$$ \displaystyle E_S(t) \approx 38.3936, \quad \frac{\displaystyle \max_t E_S(t) - \min_t E_S(t)}{\displaystyle \min_{t}E_S(t)} \approx 6 \cdot 10^{-11}.$$

We now give the proof of Theorem \ref{localexistencenontilde} using Theorem \ref{localexistencetilde}.
\begin{proofthm}{localexistencenontilde}
Using the fact that there is local existence to the initial data in the tilde domain and applying $P^{-1}$ to the solution obtained there, we can get a curve $z(\al,t)$ that solves the water wave equation in the non tilde domain. Details on the local existence in the tilde domain are shown below. Note that the sign condition \eqref{Charlie111} assumed in Theorem \ref{localexistencenontilde} guarantees that 
for positive time $t$ the curve in the nontilde domain will separate (as depicted in Figure \ref{goodbadA}) instead of crossing itself (as depicted in Figure \ref{goodbadB}).
More precisely, we check that for small positive time $t$ the curve $\al \mapsto z(\al,t) = (z_1(\al,t), z_2(\al,t)) = P^{-1}(\tilde{z}(\al,t)) \in
\mathbb{R}/2\pi \mathbb{Z} \times \mathbb{R}$ is a simple closed curve, i.e. that $\al \mapsto z(\al,t)$ is one-to-one. Indeed, if not, there exist a sequence of positive times $t_{\nu} \to 0$ and points $\al_{\nu}^{'}, \al_{\nu}^{''}$ such that $\al_{\nu}^{'} \neq \al_{\nu}^{''} \mod 2\pi \mathbb{Z}$, but $z(\al_{\nu}^{'}, t_{\nu}) = z(\al_{\nu}^{''}, t_{\nu})$. Since the initial splash curve $\al \mapsto z(\al,0)$ satisfies the modified chord-arc condition described in Condition 2 of Definition \ref{defsplash}, we may assume without loss of generality that $\al_{\nu}^{'} \rightarrow \al_{1}$ and $\al_{\nu}^{''} \rightarrow \al_{2}$ (with $\al_1, \al_2$ as in Definition \ref{defsplash}). The sign condition \eqref{Charlie111} therefore guarantees that (for large $\nu$), $\tilde{z}(\al_{\nu}^{'}, t_{\nu})$ and $\tilde{z}(\al_{\nu}^{''}, t_{\nu})$ lie in the image of the (open) time-zero water region under the map $P$. Moreover (for large $\nu$), $\tilde{z}(\al_{\nu}^{'}, t_{\nu}) \neq \tilde{z}(\al_{\nu}^{''}, t_{\nu})$ since $\tilde{z}(\al_1,0) \neq \tilde{z}(\al_2,0)$.

Since $P^{-1}$ is one-to-one on the image of the open time-zero water region under $P$, it follows that (for large $\nu$) we have $z(\al_{\nu}^{'}, t_{\nu}) \neq z(\al_{\nu}^{''}, t_{\nu}) \in \mathbb{R}/2\pi \mathbb{Z} \times \mathbb{R}$, with $z(\al,t) \equiv P^{-1}(\tilde{z}(\al,t)).$ This contradicts the defining condition $z(\al_{\nu}^{'}, t_{\nu}) = z(\al_{\nu}^{''}, t_{\nu})$, completing the proof that $\al \mapsto z(\al,t)$ is a simple closed curve for small positive $t$.

The proof of Theorem \ref{localexistencenontilde} is complete.
\end{proofthm}

We end this section by defining a ``splat curve'', as promised in Section \ref{Section1Introduction}. To do so, we simply modify our Definition \ref{defsplash} for a splash curve, by replacing Condition 2 in that definition by the following

\underline{Condition $2'$:} We are given two disjoint closed non-degenerate intervals $I_1, I_2 \subset [0,2\pi)$ whose images under $\alpha \mapsto (z_1(\al), z_2(\al)) \in \mathbb{R}/2\pi \mathbb{Z} \times \mathbb{R}$ coincide.

The map $\al \mapsto (z_1(\al), z_2(\al)) \in \mathbb{R}/2\pi \mathbb{Z} \times \mathbb{R}$ satisfies the chord-arc condition when restricted to the complement of any open interval $J$ such that $J \supset I_1$ or $J \supset I_2$.

As promised, the curve depicted in Figure \ref{CharlieFig2b} is a splat curve. Observe that the curve in Figure \ref{CharlieFig2b} cannot be real-analytic.
\color{black}

\section{Proof of real-analytic short-time existence in tilde domain}

The main goal of this section is to prove Theorem \ref{localexistencetildeAnalytic}. In order to accomplish this task we will prove local well-posedness for the system \eqref{1paco} below. In this section, we will drop the tildes from the notation. The system arises from \eqref{equations-hou-tilde} taking $c=0$:
\begin{equation}
\label{1paco}
\left\{
\begin{array}{rcl}
z_t & = & \left|\frac{dP}{dw}(P^{-1}(z))\right|^{2}u \\
\Phi_t & = & \frac{1}{2}\left|\frac{dP}{dw}(P^{-1}(z))\right|^{2}|u|^{2} - gP_2^{-1}(z)\\
u & = & BR(z,\omega) + \frac{\omega}{2|z_{\al}|^{2}}z_{\al} \\
\Phi_{\al} & = & \frac{\omega}{2} + BR(z,\omega) \cdot z_{\al} \\
\left|\frac{dP}{dw}(P^{-1}(z(\al,t)))\right|^{2} & = & \frac{1}{16}\left|\frac{1+ (z_1(\al,t)+iz_2(\al,t))^{4}}{z_1(\al,t)+iz_2(\al,t)} \right|^2 \\
P_2^{-1}(z(\al,t)) & = & \log\left|\frac{i+(z_1(\al,t)+iz_2(\al,t))^2}{i-(z_1(\al,t)+iz_2(\al,t))^2}\right|.
\end{array}
\right.
\end{equation}

We demand that $z^0(\al) \neq (0,0)$ to find the function $\frac{dP}{dw}(P^{-1}(z(\al,t)))$ well defined. This condition is going to remain true for short time. We also consider $z^0(\al) \neq q_l$, $l=1,...,4$ in \eqref{points} to get $P_2^{-1}(z(\al,t))$ well defined. Again this is going to remain true for short time.

The main tool in this section is a Cauchy-Kowalewski theorem (see \cite[Section 5]{Castro-Cordoba-Fefferman-Gancedo-LopezFernandez:rayleigh-taylor-breakdown} for more details). We recall the following definitions
$$ S_{r}= \{\al + i \eta, |\eta| < r\},$$
$$ \|f\|^2_{L^2(\partial S_r)}=\sum_{\pm}\int_{-\pi}^{\pi}|f(\al\pm ir)|^{2}d\al,$$
$$ \|f\|_r^2=\|f\|^2_{L^2(\partial S_r)}+\|\da^3f\|^2_{L^2(\partial S_r)}, $$

the space
\begin{align*}
H^3(\partial S_r) = \left\{f \text{ analytic in } S_{r}, \|f\|_{r}^{2}<\infty, f \; 2\pi \text{-periodic}\right\}
\end{align*}
and we now take $(z_1, z_2,\Phi) \in (H^3(\partial S_r))^3 \equiv X_{r}$. We have the following theorem:
\begin{theorem}
\label{localexistenceanalyticsection3}
Let $z^{0}(\alpha)$ be a curve satisfying the arc-chord condition
$$ \frac{|z^0(\alpha) - z^{0}(\alpha - \beta)|^{2}}{|\beta|^{2}} > \frac{1}{M^2}$$
which doesn't touch the points $q_l$, $l=0,...,4$ in \eqref{points},and $(z^0, \Phi^0) \in X_{r_{0}}$ for some $r_0 > 0$. Then, there exist a time $T > 0$ and $0 < r < r_0$ such that there is a unique solution to the system \eqref{1paco} in $C([0,T],X_{r})$ with initial conditions $z(\al,0) = z^{0}(\al), \Phi(\al,0) = \Phi^{0}(\al)$, for all $\al \in \mathbb{T}$.
\end{theorem}
Equation \eqref{1paco} can be extended for complex variables:
\begin{equation*}
z_t(\al + i \xi,t) = F^1(z(\al + i\xi, t), \Phi(\al + i \xi, t)), \quad \Phi_t(\al + i \xi,t) = F^2(z(\al + i\xi, t), \Phi(\al + i \xi, t)).
\end{equation*}
Here
$$
F^1(z,\Phi)=\left|\frac{dP}{dw}(P^{-1}(z))\right|^{2}u
$$
where we abuse notation by writing
$$
\left|\frac{dP}{dw}(P^{-1}(z(\al+i\xi,t)))\right|^{2}=
\frac{1}{16}\frac{\Pi_{l=1}^4[(z_1(\al+i\xi,t)-q^l_1)^2+
(z_2(\al+i\xi,t)-q^l_2)^2]}{(z_1(\al+i\xi,t))^2
+(z_2(\al+i\xi,t))^2}
$$
and
$$
u(\al+i\xi,t)=BR(z(\al+i\xi,t),\omega(\al+i\xi,t))+ \frac{1}{2}\left(\frac{\omega(\al+i\xi,t)\da z(\al+i\xi,t)}{
(\da z_1(\al+i\xi,t))^2+(\da z_2(\al+i\xi,t))^2}\right)
$$
with
$$
BR(z(\al+i\xi,t),\omega(\al+i\xi,t))=
$$
$$
\frac{1}{2\pi}PV\int_\T \frac{(z_2(\al+i\xi-\beta,t)-z_2(\al+i\xi,t),z_1(\al+i\xi,t)-z_1(\al+i\xi-\beta,t))}
{(z_1(\al+i\xi,t)-z_1(\al+i\xi-\beta,t))^2+
(z_2(\al+i\xi,t)-z_2(\al+i\xi-\beta,t))^2}\omega(\al+i\xi-\beta,t)d\beta
$$
and $\omega$ given implicitly by
$$
\Phi_{\al}=\frac{\omega}{2} + BR(z,\omega)\cdot z_{\al}.
$$
We will also abuse notation by writing $|u|^{2}$ for $u_1^{2} + u_2^{2}$, even for complex $u = (u_1,u_2)$.
The operator $F^2$ is given by
$$
F^2(z,\Phi)=\frac12\left|\frac{dP}{dw}\left(P^{-1}(z)\right)\right|^{2}|u|^2-gP_2^{-1}(z)
$$
where
$$
P_2^{-1}(z(\al+i\xi,t))=\frac12 \sum_{l=1}^4(-1)^{l}\log [(z_1(\al+i\xi,t)-q^l_1)^2+
(z_2(\al+i\xi,t)-q^l_2)^2].
$$
Below we will use a strip of analyticity small enough so that the complex logarithm above is continuous.
We use the following proposition:
\begin{proposition}
\label{analyticestimates}
Consider $0 \leq r < r'$ and the open set $O \subset X_{r'}$ given by:

\begin{align*}
O =\{(z,\Phi) \in X_{r'}:\, &\|z_i\|_{r'}, \|\Phi\|_{r'} < R, \inf_{\al + i \xi \in S_r}|(z_1(\al+i\xi)-q^l_1)^2+(z_2(\al+i\xi)-q^l_2)^2| > R^{-2}, \\
&l=0,...,4,\,\inf_{\stackrel{\al + i \xi \in S_r}{\beta \in [-\pi,\pi]}}G(z)(\al + i \xi,\beta) > R^{-2}\}
\end{align*}

with

$$ G(z)(\al + i\xi,\beta) = \left|\frac{(z_1(\al + i\xi) - z_1(\al + i\xi - \beta))^2+(z_2(\al + i\xi) - z_2(\al + i\xi - \beta))^2}{\beta^2}\right| $$
then the function $F = (F^1,F^2)$ for $F: O \rightarrow X_{r}$ is a continuous mapping. In addition, there is a constant $C_R$ (depending on $R$ only) such that

\begin{equation}
\label{2paco}
\|F(z,\Phi)\|_{r} \leq \frac{C_R}{r'-r}\|(z,\Phi)\|_{r'}
\end{equation}
\begin{equation}
\label{3paco}
\|F(z^{2},\Phi^{2}) - F(z^{1},\Phi^{1})\|_{r} \leq \frac{C_R}{r'-r}\|(z^{2}-z^{1},\Phi^{2} - \Phi^{1})\|_{r'}
\end{equation}
and
\begin{equation}
\label{4paco}
\sup_{\stackrel{\al + i \xi \in S_r}{\beta \in [-\pi,\pi]}}|F^1(z,\Phi)(\al + i \xi) - F^1(z,\Phi)(\al + i \xi - \beta)| \leq C_R|\beta|
\end{equation}
for $z,z^{j},\Phi,\Phi^{j} \in O$.
\end{proposition}

\begin{proof} First we point out that $\omega$ is given in term of $\Phi_\al$ and $z$ by the implicit equation
$$\Phi_\al=\frac{\omega}2+BR(z,\omega)\cdot z_\al \equiv \frac12(I+J)(\omega).$$
It is well known that the operator $(I+J)$ is invertible on $L^2$ for real functions with mean zero (see \cite[Section 5]{Cordoba-Cordoba-Gancedo:interface-water-waves-2d} for more details). Writing
$$
\omega(\al\pm ir)=2\Phi_\al(\al\pm ir)-\frac1{2\pi}\int_\T\frac{(z(\al\pm ir)-z(\beta))^\bot\cdot z_\al(\al\pm ir)}{|z(\al\pm ir)-z(\beta)|^2}\omega(\beta)d\beta
$$
one can find that
$$
\|\omega\|_{L^2(\partial S_r)}\leq 2 \|\Phi_\al\|_{L^2(\partial S_r)}+C_R\|\omega\|_{L^2(\partial S_0)}
$$
(where $C_R$ depends on $R$) for $(z,\Phi)\in O$. The bound of $(I+J)^{-1}$ for real functions yields
$$
\|\omega\|_{L^2(\partial S_0)}\leq 2\|(I+J)^{-1}\|_{L^2\to L^2}\|\Phi_\al\|_{L^2(\partial S_0)}\leq C_R \|\Phi_\al\|_{L^2(\partial S_r)}.
$$

Thus
$$ \|\omega\|_{L^2(\partial S_r)} \leq C_R\|\Phi_{\al}\|_{L^2(\partial S_r)}.$$
Analogously, one finds that
$$ \|\da^2\omega\|_{L^2(\partial S_r)} \leq C_R\|\Phi\|_{r}.$$

This allows us to assert that $\omega$ is at the same level as $\Phi_{\al}$ in terms of derivatives:
\begin{equation}\label{qtsb}
\|\omega\|_{L^2(\partial S_{r})}+\|\da^2\omega\|_{L^2(\partial S_{r})}\leq
C_R \|\Phi\|_{r}\leq C_R \|\Phi\|_{r'}.
\end{equation}

 Then, inequality \eqref{2paco} follows as in \cite[Section 6.3]{Castro-Cordoba-Fefferman-Gancedo-LopezFernandez:rayleigh-taylor-breakdown}. We will see how to deal with the most singular terms. For the first term in the norm, it is easy to find that
\begin{equation}
\|F(z,\Phi)\|_{L^2(\partial S_{r})}\leq C_R\|(z,\Phi)\|_{r}\leq C_R\|(z,\Phi)\|_{r'}.
\end{equation}
In order to control the second one, we will show how to deal with $F^1$ as $F^2$ is analogous. Here we point out that the functions
$$
\left|\frac{dP}{dw}(P^{-1}(z(\al+i\xi,t)))\right|^{2}, \quad P_2^{-1}(z(\al+i\xi,t))
$$
have no loss of derivatives and they are regular as long as $(z,\Phi)\in O$. Therefore, in $\da^3 F^1$ the most singular term is given by
$$
\left|\frac{dP}{dw}(P^{-1}(z(\al+i\xi,t)))\right|^{2}\da^3 u(\al+i\xi,t)
$$
as the rest can be estimated in an easier manner (see \cite[Section 6.1]{Cordoba-Cordoba-Gancedo:interface-water-waves-2d} as an example with more details).
From the definition it is easy to bound $\left|\frac{dP}{dw}(P^{-1}(z))\right|^{2}$ in $L^\infty$, it remains to control $\da^3 u$ in $L^2(\partial S_r)$. To simplify the exposition we ignore the time dependence of the functions, we denote $\gamma=\alpha\pm ir$,
$$
(z_1(\gamma)-z_1(\gamma-\beta))^2+
(z_2(\gamma)-z_2(\gamma-\beta))^2 \equiv |z(\g)-z(\g-\beta)|^2_*,
$$
$$
(\da z_1(\gamma))^2+(\da z_2(\gamma))^2
\equiv |z_\al(\g)|^2_*,
$$
and
$$
(z_2(\gamma-\beta)-z_2(\gamma),z_1(\g)-z_1(\g-\beta)) \equiv (z(\g)-
z(\g-\beta))^{\bot}.
$$
Next, we split as follows
$$
\da^3 u=I_1+I_2+I_3+I_4+I_5+I_6+\mbox{l.o.t.}
$$
where l.o.t. denotes lower order terms which can be estimated in an easier manner. We have
$$
I_1=\frac{1}{2\pi}PV\int_{-\pi}^\pi\frac{(\da^3 z(\g)-\da^3 z(\g-\beta))^{\bot}}{|z(\g)-z(\g-\beta)|_*^2}\omega(\g-\beta)d\beta,
$$
$$
I_2=\frac{-1}{\pi}PV\int_{-\pi}^\pi\frac{(z(\g)\!-\!z(\g\!-\!\beta))^{\bot}}{(|z(\g)\!-\!z(\g\!-\!\beta)|_*^2)^2}
(z(\g)\!-\!z(\g\!-\!\beta))\cdot(\da^3 z(\g)\!-\!\da^3 z(\g\!-\!\beta))\omega(\g\!-\!\beta)d\beta,
$$
$$
I_3=\frac{1}{2\pi}PV\int_{-\pi}^\pi\frac{(z(\g)-z(\g-\beta))^{\bot}}{|z(\g)-z(\g-\beta)|_*^2}\da^3\omega(\g-\beta)d\beta,
$$
$$
I_4=\frac12\frac{\omega(\g) \da^4 z(\g)}{|z_\al(\g)|_*^2},
$$
$$
I_5=-\frac12\frac{\omega(\g) \da z(\g)}{(|z_\al(\g)|_*^2)^2}
\da z(\g)\cdot \da^4z(\g)$$
and
$$
I_6=\frac12\frac{\da^3\omega(\g) \da z(\g)}{|z_\al(\g)|_*^2}.
$$
For $I_6$ we find
$$
\|I_6\|_{L^2(\partial S_{r})}\leq \frac12 \|\da z\|_{L^\infty(S_{r})}\Big(\inf_{\stackrel{\g\in S_{r}}{\beta \in [-\pi,\pi]}}G(z)(\g,\beta)\Big)^{-1}\|\da^3\omega\|_{L^2(\partial S_{r})}
$$
and since $(z,\Phi) \in O$ we get
$$
\|I_6\|_{L^2(\partial S_{r})}\leq C_R\|\da^3\omega\|_{L^2(\partial S_{r})}
$$
by using Sobolev embedding.
A simple application of the Cauchy formula gives
$$
\|\da f\|_{L^2(\partial S_{r})}\leq \frac{C}{r'-r}\|f\|_{L^2(\partial S_{r'})}
$$
which allows us to find
$$
\|I_6\|_{L^2(\partial S_{r})}\leq \frac{C_R}{r'-r}\|\da^2\omega\|_{L^2(\partial S_{r'})}.
$$
The bound \eqref{qtsb} gives finally
$$
\|I_6\|_{L^2(\partial S_{r})}\leq \frac{C_R}{r'-r}\|\Phi\|_{r'}.
$$
In a similar way we obtain
$$
\|I_4\|_{L^2(\partial S_{r})}+\|I_5\|_{L^2(\partial S_{r})}\leq C_R
\|\da^4z\|_{L^2(\partial S_{r})}\leq \frac{C_R}{r'-r}\|z\|_{r'}.
$$
In $I_3$ we decompose further: $I_3=I_{3,1}+I_{3,2}$ where
$$
I_{3,1}=\frac{1}{2\pi}PV\int_{-\pi}^\pi K(\g,\beta)\da^3\omega(\g-\beta)d\beta,\quad I_{3,2}=\frac{1}{2}\frac{z^\bot_\al(\g)}{|z_\al(\g)|_*^2}H(\da^3\omega)(\g),
$$
where $H$ denotes the Hilbert transform
and the kernel $K$ is given by
$$
\frac{(z(\g)-z(\g-\beta))^{\bot}}{|z(\g)-z(\g-\beta)|_*^2}-\frac{z^\bot_\al(\g)}{|z_\al(\g)|_*^2}\frac{1}{2\tan(\beta/2)}.
$$

We can integrate by parts $\partial_\beta(-\da^2\omega(\g-\beta))$ in $I_{3,1}$ to find
$$
\|I_{3,1}\|_{L^2(\partial S_{r})}\leq C_R\|\da^2\omega\|_{L^2(\partial S_{r})}\leq C_R\|\Phi\|_{r'}
$$
(see \cite[Section 3]{Cordoba-Cordoba-Gancedo:interface-water-waves-2d} for more details). The term $I_{3,2}$ can be estimated by
$$
\|I_{3,2}\|_{L^2(\partial S_{r})}\leq C_R\|H(\da^3\omega)\|_{L^2(\partial S_{r})}=C_R\|\da^3\omega\|_{L^2(\partial S_{r})}\leq \frac{C_R}{r'-r}\|\Phi\|_{r'}.
$$
A similar splitting in $I_2=I_{2,1}+I_{2,2}$ with
$$
I_{2,1}=\frac{-1}{\pi}PV\int_{-\pi}^\pi L(\g,\beta)\cdot
(\da^3 z(\g)\!-\!\da^3 z(\g\!-\!\beta))d\beta,
$$
$$
I_{2,2}=-\frac{\omega(\g)z^{\bot}_\al(\g)}{(
|z_\al(\g)|_*^2)^2}z_\al(\g)\cdot \Lambda(\da^3z)(\g),
$$
(where $\Lambda=H\da$) gives the kernel $L$ as follows
\begin{align*}
L(\g,\beta)\cdot
(\da^3 z(\g)\!-\!\da^3 z(\g\!-\!\beta))=&
-\frac{\omega(\g)z^{\bot}_\al(\g)}{(
|z_\al(\g)|_*^2)^2}\frac{z_\al(\g)\cdot(\da^3 z(\g)\!-\!\da^3 z(\g\!-\!\beta))}{4\sin^2(\beta/2)}\\
+\frac{\omega(\g\!-\!\beta)(z(\g)\!-\!z(\g\!-\!\beta))^{\bot}}{(|z(\g)\!-\!z(\g\!-\!\beta)|_*^2)^2}&
(z(\g)\!-\!z(\g\!-\!\beta))\cdot
(\da^3 z(\g)\!-\!\da^3 z(\g\!-\!\beta)).
\end{align*}
Heuristically, we regard this operator as no better or no worse than a Hilbert transform of $\da^3 z$. It is easy to prove that
$$
\|I_{2,1}\|_{L^2(\partial S_{r})}\leq C_R\|\da^3z\|_{L^2(\partial S_{r})}\leq C_R\|\Phi\|_{r'}
$$
(see \cite[Section 6.1]{Cordoba-Cordoba-Gancedo:interface-water-waves-2d} for more details). The term $I_{2,2}$ can be bounded as follows
$$
\|I_{2,2}\|_{L^2(\partial S_{r})}\leq C_R\|\Lambda(\da^3z)\|_{L^2(\partial S_{r})}=C_R\|\da^4z\|_{L^2(\partial S_{r})}\leq \frac{C_R}{r'-r}\|z\|_{r'}.
$$
Analogously, for $I_1$ we find
$$
\|I_{1}\|_{L^2(\partial S_{r})} \leq \frac{C_R}{r'-r}\|z\|_{r'}.
$$
This strategy allows us to deal with $\da^3u$ and therefore with $\da^3F^1$. The same applies to $\da^3F^2$ and we can get finally \eqref{2paco}.

To get \eqref{3paco} we write

$$ \Phi_{\al}^{1} = \frac{1}{2}(I + J_{z^{1}})(\omega^{1}), \quad \Phi_{\al}^{2} = \frac{1}{2}(I + J_{z^{2}})(\omega^{2})$$
where
$$
J_{z^{j}}(\omega)=2BR(z^j,\omega)\cdot z^j_\al
$$
for $z^j\in O$ and $j=1,2$. This implies
$$\Phi_{\al}^{2} - \Phi_{\al}^{1} = \frac{\omega^{2} - \omega^{1}}{2} + BR(z^{2},\omega^{2} - \omega^{1}) \cdot z^{2}_{\al} +
BR(z^{2},\omega^{1}) \cdot z^{2}_{\al} - BR(z^{1},\omega^{1}) \cdot z^{1}_{\al}$$
which yields
$$(\omega^{2} - \omega^{1}) = 2(I+J_{z^{2}})^{-1}(\Phi_{\al}^{2} - \Phi_{\al}^{1}) - 2(I+J_{z^{2}})^{-1}(BR(z^{2},\omega^{1}) \cdot z^{2}_{\al} - BR(z^{1},\omega^{1}) \cdot z^{1}_{\al}).$$

This helps us to find
$$ \|\omega^{2} - \omega^{1}\|_{L^2(\partial S_{r})}+\|\da^2\omega^{2} - \da^2\omega^{1}\|_{L^2(\partial S_{r})}\leq C(R)(\|\Phi^{2} - \Phi^{1}\|_{r} + \|z^{2} - z^{1}\|_{r}).$$
We use a decomposition similar to the one used to prove \eqref{2paco} which allows us to get finally \eqref{3paco}. Inequality \eqref{4paco} follows in an easier manner.
\end{proof}

\begin{proofthm}{localexistenceanalyticsection3}
We apply the following result of Nirenberg \cite{Nirenberg:abstract-cauchy-kowalewski} and Nishida \cite{Nishida:theorem-Nirenberg}.
\vskip 0.5cm
\underline{Abstract Cauchy-Kowalewski Theorem:}

Consider the equation
\begin{equation}
\label{ACharlie}
\frac{du(t)}{dt} = F(u(t)) \text{ for } |t| < \delta
\end{equation}

with initial condition
\begin{equation}
\label{BCharlie}
u(0) = u^0 \in X_{r_0}
\end{equation}

For some numbers $\hat{C}, \hat{R} > 0$, assume the following hypothesis:

For every pair of numbers $r,r'$ such that $0 < r' < r < r_0$, $F$ is a Lipschitz map from $\{u \in X_{r} : \|u - u^{0}\|_{X_{r}} < \hat{R}\}$ into $X_{r'}$, with Lipschitz constant at most $\displaystyle \frac{\hat{C}}{r-r'}$. Then the equation \eqref{ACharlie} with initial condition \eqref{BCharlie} has a solution $u(t)$ in $C([-\delta,\delta],X_{r})$ for small enough $r,\delta > 0$.

The above Abstract Cauchy-Kowalewski Theorem is obviously equivalent to a special case of Nishida's Theorem \cite{Nishida:theorem-Nirenberg}, although our notation differs from that of \cite{Nirenberg:abstract-cauchy-kowalewski}. In place of \eqref{ACharlie}, Nirenberg and Nishida treat the more general equation
$\displaystyle \frac{du(t)}{dt} = F(u(t),t)$.

The proof of the Abstract Cauchy-Kowalewski Theorem in \cite{Nirenberg:abstract-cauchy-kowalewski} proceeds by showing that the obvious iteration scheme

$$ u^{k+1}(t) = u^{0} + \int_{0}^{t}F(u^{k}(s))ds$$

converges in $X_{r}$ for small enough $r$ (depending on $t$).

Our system \eqref{1paco} has the form $\displaystyle \frac{du}{dt} = F(u)$ for $u = (z,\Phi)$. Proposition
\ref{analyticestimates} tells us that the hypothesis of the Abstract Cauchy-Kowalewski Theorem holds for the system \eqref{1paco}. In particular, for $\hat{R} > 0$ small enough, we obtain the arc-chord condition for every $u = (z,\Phi)$ such that $\|(z,\Phi) - (z^{0},\Phi^{0})\|_{X_{r}} < \hat{R}$ for any (arbitrarily small) $r > 0$.

Hence, the conclusion of Theorem \ref{localexistenceanalyticsection3} follows from the Abstract Cauchy-Kowalewski Theorem.

\end{proofthm}

\begin{proofthm}{localexistencetildeAnalytic}

Applying Theorem \ref{localexistenceanalyticsection3}, we obtain a solution of the water wave equation, with the correct initial conditions, in the tilde domain. Passing from the tilde domain back to the original problem, we obtain a solution of the water wave equations as asserted in Theorem \ref{localexistencetildeAnalytic}.

We have to make sure that, for small positive time, the splash curve evolves as in Figure \ref{goodbadA}, rather than Figure \ref{goodbadB}.

\begin{figure}[ht]
\centering
\subfigure[Good]
{
\includegraphics[width=0.4\textwidth]{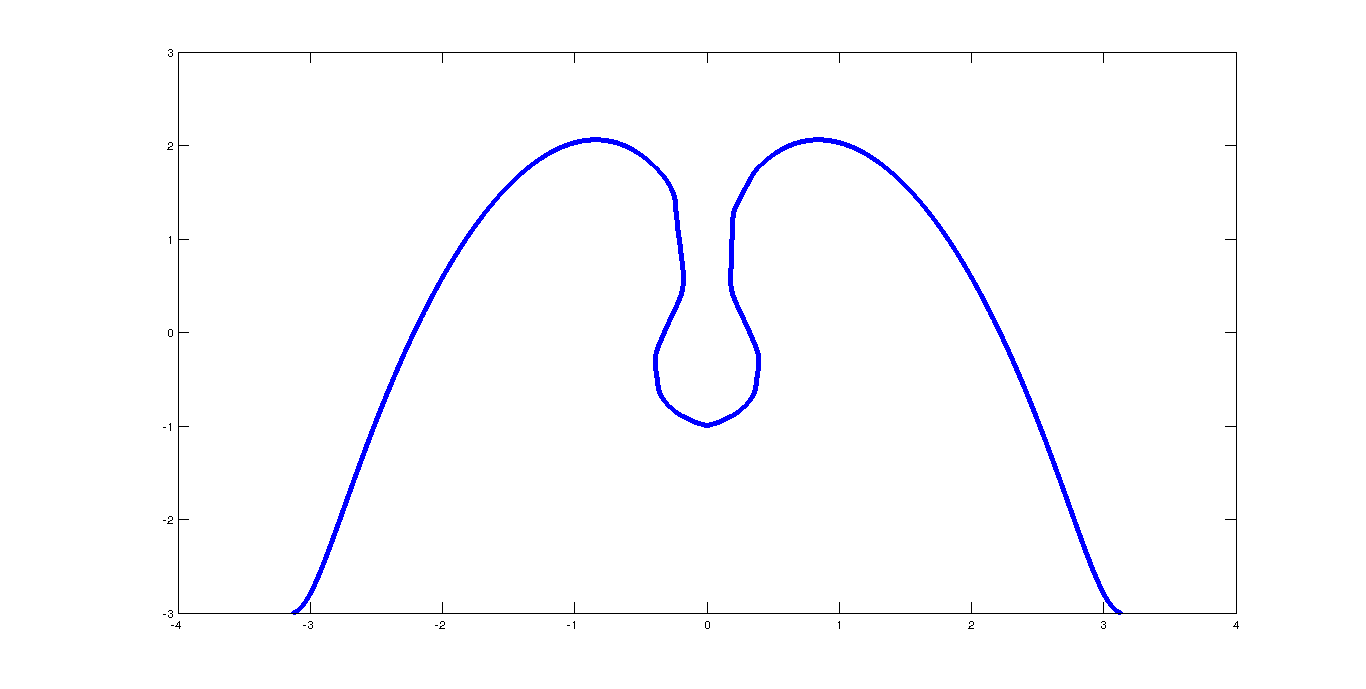}
\label{goodbadA}
}
\subfigure[Bad]
{
\includegraphics[width=0.4\textwidth]{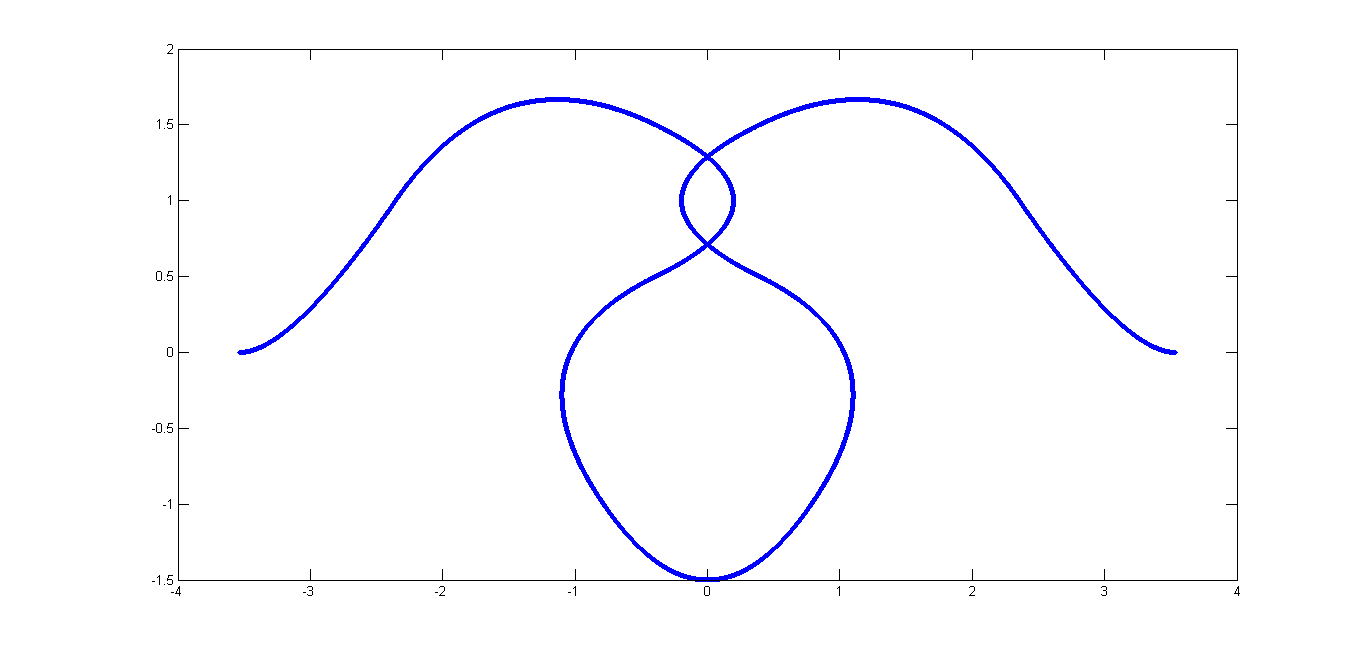}
\label{goodbadB}
}
\caption{Two different evolutions of the interface.}
\label{goodbad}
\end{figure}

This is guaranteed by the hypothesis of Theorem \ref{localexistencetildeAnalytic} regarding the sign of the normal component of the initial velocity at the splash point.

\end{proofthm}

\section{Proof of short-time existence in Sobolev spaces in the tilde domain}
\label{SectionLocalSobolev}
In this section we will show how to obtain a local existence theorem for the water wave equations in the tilde domain. The proof is based on energy estimates and uses the fact that the Rayleigh-Taylor function is positive.

\subsection{The Rayleigh-Taylor function in the tilde domain}

We begin by recalling the function $\tilde{\varphi}(\alpha,t)$, which will be studied in detail in Section \ref{subdefphi} and in the definition of the Rayleigh-Taylor condition, by the expression
\begin{equation}
\label{defphi}
\tilde{\varphi}(\alpha,t)=\frac{Q^2\tilde{\om}(\alpha,t)}{2|\ztil_\alpha(\alpha,t)|}-\tilde{c}|\ztil_\alpha(\alpha,t)|.
\end{equation}

Next we introduce the R-T function:
\begin{align}
\begin{split}\label{R-T}
 \sigma  \equiv& \left(BR_{t}(\tilde{z},\tilde{\omega}) + \frac{\tilde{\varphi}}{|\tilde{z}_{\al}|}BR_{\al}(\tilde{z},\tilde{\omega})\right) \cdot \tilde{z}_{\al}^{\perp} + \frac{\tilde{\omega}}{2|\tilde{z}_{\al}|^{2}}\left(\tilde{z}_{\al t} + \frac{\tilde{\varphi}}{|\tilde{z}_{\al}|}\tilde{z}_{\al \al}\right) \cdot \tilde{z}_{\al}^{\perp} \\
& + Q\left|BR(\tilde{z},\tilde{\omega}) + \frac{\tilde{\omega}}{2|\tilde{z}_{\al}|^{2}}\tilde{z}_{\al}\right|^{2}(\nabla Q)(\tilde{z}) \cdot \tilde{z}_{\al}^{\perp}
 + g(\nabla P_{2}^{-1})(\tilde{z}) \cdot \tilde{z}_{\al}^{\perp}.
\end{split}
\end{align}

This function $\sigma$ coincides with the expression $\tilde{z}^\perp(\alpha,t)\cdot \nabla \tilde{p}(\tilde{z}(\alpha,t),t)$, where $\tilde{p}=p\circ P^{-1}$. Indeed, it is easy to check that
\begin{equation}\label{rt1}
\pa_t\tilde{\phi}+\frac{Q^2}{2}\left|\tilde{v}\right|^2=-\tilde{p}-gP^{-1}_2 + p^*(t).
\end{equation}
And taking the gradient on the equation \eqref{rt1} yields
\begin{equation}\label{rt2}
\tilde{v}_t+\frac{1}{2}\left(\nabla Q^2\right)|\tilde{v}|^2+Q^2(\tilde{v}\cdot\nabla)\tilde{v}=-\nabla\tilde{p}-g\nabla P^{-1}_2.\end{equation}
In addition we know that
\begin{equation}\label{rt3}
\tilde{v}(\tilde{z}(\alpha,t),t)=BR(\tilde{z},\tilde{\om})(\alpha,t)+\frac{\tilde{\om}(\alpha,t)}{2|\tilde{z}_\alpha(\alpha,t)|^2}\tilde{z}_\alpha(\alpha,t)\end{equation}
and therefore
\begin{equation}\label{rt4}
 \frac{d}{dt}\tilde{v}(\ztil(\alpha,t),t)=\pa_t BR(\ztil,\tilde{\om})(\alpha,t)+\pa_t\left(\frac{\tilde{\om}(\alpha,t)}{2|\ztil_\alpha(\alpha,t)|^2}\right)\ztil_\alpha(\alpha,t)
+\frac{\tilde{\om}(\alpha,t)}{2|\ztil_\alpha(\alpha,t)|^2}\pa_t\ztil_\alpha(\alpha,t).
\end{equation}
On the other hand, by using \eqref{rt2} we have
\begin{align}
 \frac{d}{dt}\tilde{v}(\ztil(\alpha,t),t)=&\pa_t \tilde{v}(\ztil(\alpha,t),t)+(\pa_t\tilde{z}(\alpha,t)\cdot\nabla) \tilde{v}(\ztil(\alpha,t),t)\nonumber\\
=&-\frac{1}{2}\left(\nabla Q^2\right)|\tilde{v}(\ztil(\alpha,t),t)|^2-Q^2(\tilde{v}(\ztil(\alpha,t),t)\cdot\nabla)\tilde{v}(\ztil(\alpha,t),t)\nonumber\\
 &-\nabla \tilde{p}(\ztil(\alpha,t),t)-g\nabla P^{-1}_2(\ztil(\alpha,t))+(\pa_t\ztil(\alpha,t)\cdot\nabla )\tilde{v}(\ztil(\alpha,t),t).\label{rt5}
\end{align}
Furthermore the equation \eqref{ikea} together with \eqref{rt3} gives rise to
\begin{align}
\pa_t\ztil(\alpha,t)=& Q^2\tilde{v}(\ztil(\alpha,t),t)-\frac{Q^2\tilde{\om}(\alpha,t)}{2|\ztil_\alpha(\alpha,t)|^2}\ztil_\alpha(\alpha,t)+\tilde{c}\tilde{z}_\alpha(\alpha,t)\nonumber\\
=& Q^2\tilde{v}(\ztil(\alpha,t),t)-\frac{1}{|\ztil_\alpha(\alpha,t)|}\left(\frac{Q^2\tilde{\om}(\alpha,t)}{2|\ztil_\alpha(\alpha,t)|}
-\tilde{c}|\ztil_\alpha(\alpha,t)|\right)\ztil_\alpha(\alpha,t)\label{rt6}.
\end{align}
Therefore by \eqref{defphi}, we obtain
\begin{align}\label{rt7}
\pa_t \ztil(\alpha,t)=& Q^2\tilde{v}(\ztil(\alpha,t),t)-\tilde{\varphi}(\alpha,t)\frac{\ztil_\alpha(\alpha,t)}{|\ztil_\alpha(\alpha,t)|}.
\end{align}
By introducing \eqref{rt7} in \eqref{rt5} we have
\begin{align*}
 \frac{d}{dt}\tilde{v}(\ztil(\alpha,t),t)=&-\frac{1}{2}\left(\nabla Q^2\right)|\tilde{v}(\ztil(\alpha,t),t)|^2-Q^2(\tilde{v}(\ztil(\alpha,t),t)\cdot\nabla)\tilde{v}(\ztil(\alpha,t),t)\\
&-\nabla \tilde{p}(\ztil(\alpha,t),t)-g\nabla P^{-1}_2(\ztil(\alpha,t))\\&+Q^2(\tilde{v}(\ztil(\alpha,t),t)\cdot\nabla)
\tilde{v}(\ztil(\alpha,t),t)-\tilde{\varphi}(\alpha,t)\frac{\ztil_\alpha(\alpha,t)}{|\ztil_\alpha(\alpha,t)|} \cdot \nabla \tilde{v}(\tilde{z}(\al,t),t).
\end{align*}
Therefore
\begin{align}
\frac{d}{dt}\tilde{v}(\ztil(\alpha,t),t)=&-\frac{1}{2}\left(\nabla Q^2\right)|\tilde{v}(\ztil(\alpha,t),t)|^2
-\tilde{\varphi}(\alpha,t)\frac{\pa_\alpha \tilde{v}(\ztil(\alpha,t),t)}{|\ztil_\alpha(\alpha,t)|}\nonumber\\
& -\nabla \tilde{p}(\ztil(\alpha,t),t)-g\nabla P^{-1}_2(\ztil(\alpha,t)).\label{rt8}
\end{align}
Next we take a derivative with respect to $\alpha$ in the equation \eqref{rt3} to get
\begin{align}\label{rt9}
\pa_\alpha \tilde{v}(\ztil(\alpha,t),t)=\pa_\alpha BR(\ztil,\tilde{\om})(\alpha,t)+\pa_\alpha\left(\frac{\tilde{\om}(\alpha,t)}{2|\ztil_\alpha(\alpha,t)|^2}\right)\ztil_\alpha(\alpha,t)
+\left(\frac{\tilde{\om}(\alpha,t)}{2|\ztil_\alpha(\alpha,t)|^2}\right)\ztil_{\alpha\,\alpha}(\alpha,t).
\end{align}

Multiplying equation \eqref{rt8} by $\ztil^{\perp}_\alpha(\alpha,t)$ and using \eqref{rt9} we learn
\begin{align}
\left(\frac{d}{dt}\tilde{v}(\ztil(\alpha,t),t)\right)\cdot \ztil^\perp_\alpha(\alpha,t)=&-Q\nabla Q\cdot \ztil^\perp_\alpha(\alpha,t)|\tilde{v}(\ztil(\alpha,t),t)|^2
\nonumber\\&-\frac{\tilde{\varphi}(\alpha,t)}{|\ztil_\alpha(\alpha,t)|}\pa_\alpha BR(\ztil,\tilde{\om})(\alpha,t)\cdot \ztil^{\perp}_\alpha(\alpha,t)
\nonumber\\&-\frac{\tilde{\varphi}(\alpha,t)}{|\ztil_\alpha(\alpha,t)|}\left(\frac{\tilde{\om}(\alpha,t)}{2|\ztil_\alpha(\alpha,t)|^2}\right)
\ztil_{\alpha\,\alpha}(\alpha,t)\cdot \ztil^\perp_\alpha(\alpha,t)
 \nonumber\\&-\nabla \tilde{p}(\ztil(\alpha,t),t)\cdot \ztil^\perp_\alpha(\alpha,t)-g\nabla P^{-1}_2(\ztil(\alpha,t))\cdot \ztil^\perp_\alpha(\alpha,t).\label{rt10}
\end{align}
On the other hand, by multiplying \eqref{rt4} by  $\ztil^\perp_\alpha(\alpha,t)$ we have
\begin{align}\label{rt11}
\left(\frac{d}{dt}\tilde{v}(\ztil(\alpha,t),t)\right)\cdot \ztil^\perp_\alpha(\alpha,t)=& \pa_t BR(\ztil,\om)\cdot \ztil^\perp_\alpha(\alpha,t)
+\frac{\tilde{\om}}{2|\ztil_\alpha(\alpha,t)|^2}\pa_t\ztil_\alpha(\alpha,t)\cdot \ztil^\perp_\alpha(\alpha,t).
\end{align}
From \eqref{rt10} and \eqref{rt11} we find
\begin{align}
& \pa_t BR(\ztil,\om)\cdot \ztil^\perp_\alpha(\alpha,t)\nonumber
+\frac{\tilde{\om}}{2|\ztil_\alpha(\alpha,t)|^2}\pa_t\ztil_\alpha(\alpha,t)\cdot \ztil^\perp_\alpha(\alpha,t)\\
=&-Q\nabla Q\cdot \ztil^\perp_\alpha(\alpha,t)|\tilde{v}(\ztil(\alpha,t),t)|^2
-\frac{\tilde{\varphi}(\alpha,t)}{|\ztil_\alpha(\alpha,t)|}\pa_\alpha BR(\ztil,\tilde{\om})(\alpha,t)\cdot \ztil^\perp_\alpha(\alpha,t)
\nonumber\\&-\frac{\tilde{\varphi}(\alpha,t)}{|\ztil_\alpha(\alpha,t)|}\left(\frac{\tilde{\om}(\alpha,t)}{2|\ztil_\alpha(\alpha,t)|^2}\right)
\ztil_{\alpha\,\alpha}(\alpha,t)\cdot \ztil^\perp_\alpha(\alpha,t) \nonumber\\
& -\nabla \tilde{p}(\ztil(\alpha,t),t)\cdot \ztil^\perp_\alpha(\alpha,t)-g\nabla P^{-1}_2(\ztil(\alpha,t))\cdot \ztil^\perp_\alpha(\alpha,t).\label{rt12}
\end{align}
Finally, rearranging the terms in \eqref{rt12} yields
\begin{align*}
&-\nabla \tilde{p}(\ztil(\alpha,t),t)\cdot\ztil_\alpha^\perp(\alpha,t)=
\left(\pa_t BR(\ztil,\tilde{\om})(\alpha,t)+\frac{\tilde{\varphi}(\alpha,t)}{|\ztil_\alpha(\alpha,t)|}\pa_\alpha BR(\ztil,\tilde{\om})(\alpha,t)\right)\cdot \ztil_{\al}^\perp(\alpha,t)\\
&+\frac{\tilde{\om}(\alpha,t)}{2|\ztil_\alpha(\alpha,t)|^2}\left(\pa_t\ztil_\alpha(\alpha,t)+\frac{\tilde{\varphi}(\alpha,t)}{|\ztil_\alpha(\alpha,t)|}\ztil_{\alpha\,\alpha}\right)\cdot
 \ztil_{\al}^\perp(\alpha,t)+g\nabla P^{-1}_2\cdot \ztil_{\al}^\perp(\alpha,t)
\\&+Q\left|BR(\ztil,\tilde{\om})(\alpha,t)+\frac{\tilde{\om}(\alpha,t)}{2|\ztil_\alpha(\alpha,t)|^2}\ztil_\alpha(\alpha,t)\right|^2\left(\nabla Q\cdot \ztil_{\al}^\perp(\alpha,t)\right),
\end{align*}
and then, comparing with \eqref{R-T}, we obtain the desired result
$$-\nabla \tilde{p}(\ztil(\alpha,t),t)\cdot\ztil_\alpha^\perp(\alpha,t)=\sigma(\alpha,t).$$

Note that for the tilde domain, the Rayleigh-Taylor condition is the same as in the first domain, i.e:

$$ \nabla p(\al,t) \cdot z_{\al}^{\perp}(\al,t) = \nabla \tilde{p}(\al,t) \cdot \tilde{z}_{\al}^{\perp}(\al,t)$$

where $\tilde{p} = p \circ P^{-1}$ and

$$ \tilde{z}_{\al}(\al,t) = \nabla P(z(\al,t)) \cdot z_{\al}(\al,t)
\Rightarrow \tilde{z}_{\al}^{\perp}(\al,t) = (-J \nabla P(z(\al,t)) J )\cdot z_{\al}^{\perp}(\al,t)$$

where $J$ is the rotation matrix
$\left(
                                                              \begin{array}{cc}
                                                                0 & -1 \\
                                                                1 & 0 \\
                                                              \end{array}
                                                            \right)$. Together with the Cauchy-Riemann equations this implies that

$$ (-J \nabla P(z(\al,t)) J) = \nabla P(z(\al,t)).$$

Moreover

$$ \nabla p(\al,t) = \nabla P(z(\al,t))^{T} \nabla \tilde{p}(\al,t).$$

Hence
\begin{align} \langle\nabla p(\al,t),z_{\al}^{\perp}(\al,t)\rangle
& = \langle   \nabla P(z(\al,t))^{T} \nabla \tilde{p}(\al,t), (\nabla P(z(\al,t)))^{-1}\tilde{z}_{\al}^{\perp}(\al,t)\rangle \\
& = \langle   \nabla \tilde{p}(\al,t), \tilde{z}_{\al}^{\perp}(\al,t)\rangle.
\end{align}

By taking the divergence on the Euler equation (\ref{Charlie11}-\ref{Charlie12}) and because  the flow is irrotational in the interior of the regions $\Omega^{j}(t)$ follows
$$ -\Delta p = |\nabla v|^{2} \geq 0$$
which, together with the fact that the pressure is zero on the interface and 
$p(x,y,t) + gy = O(1)$ when $y$ tends to $-\infty$,
then follows by Hopf's lemma in $\Omega^{2}(t)$ that $$\sigma(\al,t) \equiv -|z_{\al}^{\perp}(\al,t)|\partial_{n} p(z(\al,t),t) > 0,$$
except in the case $v=0$. 
This argument was suggested by Hou and Caflisch (see  \cite{Wu:well-posedness-water-waves-3d}), although the proof of the positivity of the Rayleigh-Taylor condition in the nontilde domain for all time was first introduced by Wu in \cite{Wu:well-posedness-water-waves-2d}.   

The above proof shows that $\sigma > 0$ provided our domain $\tilde{\Omega}(t)$ arises by applying the map $P$ to a domain $\Omega(t)$ with smooth boundary. Here, $\partial \Omega(t)$ may be a splash curve, but we cannot allow boundaries $\partial \tilde{\Omega}(t)$ whose inverse images under $P$ look like figure \ref{goodbadB}.

Nevertheless, since $\sigma > 0$ for the image of $P$ applied to a splash curve, we know that $\sigma > 0$ at time $t = 0$ in the context of Theorem \ref{localexistencetilde}. Our estimates below will guarantee that the condition $\sigma > 0$ persists for a short time. Thus, in proving Theorem \ref{localexistencetilde}, we may use the positivity of $\sigma$.

\subsection{Definition of $c$ in the tilde domain}

From now on, we will drop the tildes from the notation for simplicity. We will choose the following tangential term:

\begin{equation}
\label{defc}
c  =\D \frac{\al + \pi}{2\pi}\int_{-\pi}^{\pi}(Q^2 BR)_{\beta}\cdot\frac{z_\beta}{|z_{\beta}|^{2}} d\beta
- \int_{-\pi}^{\al}(Q^2 BR)_{\beta}\cdot\frac{z_{\beta}}{|z_{\beta}|^{2}}d\beta.
\end{equation}
Here and in \eqref{ikea} we find
\begin{equation*}
P_2^{-1}=P_2^{-1}(z(\al,t)) =  \log\left(\left|\frac{i+(z_1(\al,t)+iz_2(\al,t))^2}{i-(z_1(\al,t)+iz_2(\al,t))^2}\right|\right)
\end{equation*}
and
\begin{equation*}
Q=Q(z(\al,t))= \frac{1}{4}\left|\frac{1+ (z_1(\al,t)+iz_2(\al,t))^{4}}{z_1(\al,t)+iz_2(\al,t)} \right|.
\end{equation*}
These functions are regular as long as $z(\al,t) \neq q^l$. We deal with initial data which satisfy the above condition and we will show that it's going to remain true for short time. In order to measure it we define
$$m(q^l)(t)=\min_{\al\in\T}|z(\al,t)-q^l|$$ for $l=0,...,4$.

 We also point out that, because of our choice of $c(\al,t)$, solutions of (\ref{zeq} - \ref{eqomega}) satisfy that
 $$
 |z_\al(\al,t)|^2=A(t)\quad \mbox{for any } \al\in\T
 $$
as in \cite[Equations (2.2 - 2.5)]{Cordoba-Cordoba-Gancedo:interface-heleshaw-muskat}.

\subsection{Time evolution of the function $\varphi$ in the tilde domain}
\label{subdefphi}

Recall that we have defined an auxiliary function $\varphi(\al,t)$ adapted to the tilde domain, which helps us to bound several of the terms that appear:
\begin{equation}\label{fvar}
\varphi(\al,t) = \frac{Q^2(\al,t)\omega(\al,t)}{2|z_{\al}(\al,t)|} - c(\al,t)|z_{\al}(\al,t)|.
\end{equation}

We will show how to find the evolution equation for $\vp_t$. We have
$$
\vp=\frac{Q^2 \om}{2|z_{\al}|} - c|z_{\al}|
$$
and therefore
$$
\frac{\vp^{2}}{Q^{2}}=\frac{Q^2 \om^{2}}{4|z_{\al}|^{2}} + \frac{c^{2}|z_{\al}|^{2}}{Q^{2}} - c\om
$$
that yields
$$
-\da\left(\frac{\vp^{2}}{Q^{2}}\right)= -\da\left(\frac{Q^2 \om^{2}}{4|z_{\al}|^{2}}\right) -\da\left( \frac{c^{2}|z_{\al}|^{2}}{Q^{2}}\right) + \da(c\om).
$$
The equation for $\om_t$ reads:
\begin{equation}\label{eeomt}
\om_{t}  = -2BR_t \cdot z_{\al} - 2QQ_{\al}|BR|^{2} + \underbrace{2c BR_{\al} \cdot z_{\al}}_{(1a)} -\da\left(\frac{\vp^{2}}{Q^{2}}\right)
\underbrace{+\da\left( \frac{c^{2}|z_{\al}|^{2}}{Q^{2}}\right)}_{(1b)} - 2\da\left(gP_{2}^{-1}\right).
\end{equation}

 For the quantity $(1) = (1a) + (1b)$ we write
\begin{align*}
(1)&=(1a) + (1b) = 2c BR_{\al} \cdot z_{\al}+\da\left( \frac{c^{2}|z_{\al}|^{2}}{Q^{2}}\right)=2c(BR_{\al} \cdot z_{\al}+\frac{c_\al|z_{\al}|^2}{Q^2}-\frac{c|z_\al|^2 Q_\al}{Q^3})\\
&=2c[BR_\al\cdot z_\al+\frac{|z_\al|^2}{2\pi Q^2}\int_{-\pi}^{\pi}(Q^2 BR)_{\beta}\cdot\frac{z_\beta}{|z_{\beta}|^{2}} d\beta- \frac{(Q^2BR)_\al\cdot z_\al}{Q^2}-\frac{c|z_\al|^2Q_\al}{Q^3}]\\
&=2c[\frac{|z_\al|^2}{2\pi Q^2}\int_{-\pi}^{\pi}(Q^2 BR)_{\beta}\cdot\frac{z_\beta}{|z_{\beta}|^{2}} d\beta-\frac{2Q_\al BR\cdot z_\al}{Q}-\frac{c|z_\al|^2Q_\al}{Q^3}]
\end{align*}
and then \eqref{eeomt} becomes
\begin{align}
\begin{split}\label{paraomt}
\om_{t}  =& -2BR_t \cdot z_{\al} - 2QQ_{\al}|BR|^{2} -\da\left(\frac{\vp^{2}}{Q^{2}}\right)\\
&+\frac{c|z_\al|^2}{\pi Q^2}\int_{-\pi}^{\pi}(Q^2 BR)_{\beta}\cdot\frac{z_\beta}{|z_{\beta}|^{2}} d\beta-\frac{4cQ_\al BR\cdot z_\al}{Q}-\frac{2c^2|z_\al|^2Q_\al}{Q^3}- 2\da\left(gP_{2}^{-1}\right).
\end{split}
\end{align}

Furthermore
\begin{align*}
\vp_{t} & = QQ_t \frac{\om}{|z_{\al}|} - \frac{Q^2 \om}{2|z_{\al}|^{3}}z_{\al} \cdot z_{\al t} + \frac{Q^2 \om_t}{2|z_{\al}|} - \partial_t(c|z_{\al}|) \\
& = QQ_t \frac{\om}{|z_{\al}|} - \frac{Q^2 \om}{2|z_{\al}|}\frac{1}{2\pi}\int_{-\pi}^\pi(Q^2BR)_\beta\cdot \frac{z_\beta}{|z_\beta|^2} d\beta \\
&\quad + \frac{Q^2}{2|z_{\al}|}\left[-2BR_t \cdot z_{\al} - 2QQ_{\al}|BR|^{2}-\da\left(\frac{\vp^{2}}{Q^{2}}\right)\right.\\
&\quad + \left.\frac{c|z_\al|^2}{\pi Q^2}\int_{-\pi}^\pi(Q^2BR)_\beta\cdot \frac{z_\beta}{|z_\beta|^2} d\beta-\frac{4cQ_\al BR\cdot z_\al}{Q}-\frac{2c^2|z_\al|^2Q_\al}{Q^3}- 2\da\left(gP_{2}^{-1}\right) \right] - \partial_t(c|z_{\al}|). \\
\end{align*}
We should remark that we have used that
$$
z_\al\cdot z_{\al t}=\frac{1}{2\pi}\int_{-\pi}^\pi (Q^2BR)_\beta\cdot z_\beta d\beta.
$$
For simplicity, we denote
\begin{align}
\label{defb}
B(t)  & = \frac{1}{2\pi}\int_{-\pi}^{\pi}(Q^{2} BR)_{\beta} \cdot \frac{z_{\beta}}{|z_{\beta}|^{2}}d\beta.
\end{align}

Computing
\begin{align*}
\vp_{t} & = QQ_t \frac{\om}{|z_{\al}|} \underbrace{-\frac{Q^2 \om}{2|z_{\al}|}B(t)}_{(2a)} - \frac{Q^2}{|z_{\al}|}BR_t \cdot z_{\al} - \frac{Q^3Q_{\al}}{|z_{\al}|}|BR|^{2}-\frac{Q^2}{2|z_{\al}|}\da\left(\frac{\vp^{2}}{Q^{2}}\right)\\
&\quad\underbrace{+c|z_\al|B(t)}_{(2b)} -\frac{2cQQ_\al}{|z_\al|}BR\cdot z_\al-\frac{c^2|z_\al|Q_\al}{Q}- \frac{Q^2}{|z_{\al}|}\da\left(gP_{2}^{-1}\right) - \partial_t(c|z_{\al}|)
\end{align*}
We can write
\begin{align*}
(2) = & (2a) + (2b) = -B(t)\vp,
\end{align*}
and it yields
\begin{align}
\begin{split}\label{evpt}
\vp_{t} & = - \vp B(t)-\frac{Q^2}{2|z_{\al}|}\da\left(\frac{\vp^{2}}{Q^{2}}\right)- Q^{2}\Big(BR_t \cdot \frac{z_{\al}}{|z_{\al}|}
+ \frac{\da\left(gP_{2}^{-1}\right)}{|z_{\al}|}\Big) \\
&\quad + QQ_t \frac{\om}{|z_{\al}|} - 2c BR \cdot \frac{z_{\al}}{|z_{\al}|} QQ_{\al} - \frac{Q_{\al}}{Q}c^{2} |z_{\al}|
- \frac{Q^3}{|z_{\al}|}Q_{\al}|BR|^{2}- \partial_t(c|z_{\al}|).
\end{split}
\end{align}

We will use the equation above to perform energy estimates.

\subsection{Definition and a priori estimates of the energy in the tilde domain}
\label{subsection4d}
Let us consider for $k\geq 4$ the following definition of energy $E(t)$:
\begin{align}\label{E}
\begin{split}
E(t)&=
1 + 
\|z\|^2_{H^{k-1}}(t)+\int_{-\pi}^\pi\frac{Q^2(z)\sigma}{|z_\al|^2}|\da^k z|^2 d\al+\|\F(z)\|^2_{L^\infty}(t)\\
&\quad+\|\om\|^2_{H^{k-2}}(t)+\|\varphi\|_{H^{k-\frac12}}^2(t)+\frac{|z_\al|^2}{m(Q^2\sigma)(t)}+\sum_{l=0}^4\frac{1}{m(q^l)(t)},
\end{split}
\end{align}
where
$$
\F(z)=\frac{|\beta|}{|z(\al)-z(\al-\beta)|},\quad \al,\beta\in [-\pi,\pi],
$$
and $m(Q^2\sigma)=\min_{\al\in\T}\{Q^2(z(\al,t))\sigma(\al,t)\}$. In the next section we shall show a
proof of the following lemma.
\begin{lemma}
\label{lemmaenergyestimates}
Let $z(\al,t)$ and $\om(\al,t)$ be a solution of (\ref{zeq}-\ref{eqomega}). Then, the following a priori estimate holds:
\begin{align}
\begin{split}\label{ntni}
\frac{d}{dt}E(t)&\leq CE^p(t),
\end{split}
\end{align}
for $k\geq 4$ and $C$ and $p$ constants depending only on $k$.
\end{lemma}

The following subsections are devoted to proving Lemma \ref{lemmaenergyestimates}  by showing the regularity of the
different elements involved in the problem: the Birkhoff-Rott
integral, $z_t(\al,t)$, $\om_t(\al,t)$, $\om(\al,t)$; $BR_t(\al,t)$, the R-T function $\sigma(\al,t)$ and its time derivative $\sigma_t(\al,t)$.

\subsubsection{Estimates for $BR$}
\label{subsubBR}
In this section we show that the Birkhoff-Rott integral is
as regular as $\da z$.

\begin{lemma}
\label{lemmaBR}
The following estimate holds
\begin{eqnarray}\label{nsibr}
\|BR(z,\om)\|_{H^k}\leq
C(\|\F(z)\|^2_{L^\infty}+\|z\|^2_{H^{k+1}}+\|\om\|^2_{H^{k}})^j,
\end{eqnarray}
for $k\geq 2$, where $C$ and $j$ are constants independent of $z$
and $\om$.
\end{lemma}
\begin{rem}
Using this estimate for $k=2$ we find easily that
\begin{eqnarray}\label{nliibr}
\|\da BR(z,\om)\|_{L^\infty}\leq
C(\|\F(z)\|^2_{L^\infty}+\|z\|^2_{H^{3}}+\|\om\|^2_{H^{2}})^j,
\end{eqnarray}
which shall be used throughout the paper, where $C$ and $j$ are universal constants.
\end{rem}
\begin{proof}
The proof can be done as in \cite[Section 6.1]{Cordoba-Cordoba-Gancedo:interface-water-waves-2d} since the definition for the Birkhoff-Rott operator is independent of the domain.
\end{proof}

\subsubsection{Estimates for $z_t$}
\label{subsubzt}
In this section we show that $z_t$ is as regular as
$\da z$.

\begin{lemma}The following estimate holds
\begin{eqnarray}\label{nszt}
\|z_t\|_{H^k}\leq
C\left(\|\F(z)\|^2_{L^\infty}+\|z\|^2_{H^{k+1}}+\|\om\|^2_{H^{k}}+\sum_{l=0}^4\frac{1}{m(q^l)(t)}\right)^j,
\end{eqnarray}
for $k\geq 2$, where $C$ and $j$ are constants that depend only on $k$.
\end{lemma}

 \begin{proof}
 It follows from \cite[Section 6.2]{Cordoba-Cordoba-Gancedo:interface-water-waves-2d}. The only additional thing we need to control is an $L^{\infty}$ norm of $Q^2$, which we can easily bound by the $m(q^l)$ terms which control the distance from the curve to the $q^l$ points, more precisely, the one that controls the distance from the origin.
 \end{proof}

\subsubsection{Estimates for $\om_t$}
\label{subsubomt}
This section is devoted to showing that $\om_t$ is as regular as $\da \om$.

\begin{lemma}
\label{lemmaomt}
The following estimate holds
\begin{eqnarray}\label{nswt}
\|\om_t\|_{H^k}\leq C
\left(\|\F(z)\|^2_{L^\infty}+\|z\|^2_{H^{k+2}}+\|\om\|^2_{H^{k+1}}+\|\varphi\|^2_{H^{k+1}}+\sum_{l=0}^4\frac{1}{m(q^l)(t)}\right)^j,
\end{eqnarray}
for $k\geq 1$, where $C$ and $j$ are constants that depend only on $k$.
\end{lemma}

\begin{proof}
We use formula \eqref{paraomt} and proceed as in \cite[Section 6.3]{Cordoba-Cordoba-Gancedo:interface-water-waves-2d}. Note that in \cite{Cordoba-Cordoba-Gancedo:interface-water-waves-2d} an exponential growth appears in the bound of the estimates for the nonlocal operator acting on $\om_t$ (see equation \eqref{paraomt}). However, in a recent paper \cite{Cordoba-Cordoba-Gancedo:muskat-3d}
 the authors get a polynomial growth for the operator in both 2 and 3 dimensions. Note that even the exponential growth is still good enough to prove Theorem \ref{localexistencetilde}.

\end{proof}

\subsubsection{Estimates for $\om$}
\label{subsubom}
 In this section we show that the amplitude of the vorticity $\om$ lies at the same level as $\da z$. We shall consider $z\in
H^k(\T)$, $\varphi\in H^{k-\frac12}(\T)$ and $\om\in H^{k-2}(\T)$ as part of the energy estimates. The inequality below yields
$\om\in H^{k-1}(\T)$.

\begin{lemma}The following estimate holds
\begin{eqnarray}\label{nsw}
\|\om \|_{H^{k-1}}\leq
C\left(\|\F(z)\|^2_{L^\infty}+\|z\|^2_{H^{k}}+\|\om\|^2_{H^{k-2}}+\|\varphi\|^2_{H^{k-1}}+\sum_{l=0}^4\frac{1}{m(q^l)(t)}\right)^j,
\end{eqnarray}
for $k\geq 3$, where $C$ and $j$ are constants that depend only on $k$.
\end{lemma}

\begin{proof}
We can apply the same techniques as in \cite[Section 6.4]{Cordoba-Cordoba-Gancedo:interface-water-waves-2d} since the most singular terms are treated there and the other terms are harmless and can be easily estimated. The impact of $Q$ is now taken into account by the $m(q^{l})$ terms (which now cover all of the points $q^0,...,q^4$).
\end{proof}

\subsubsection{Estimates for $BR_t$.}
\label{subsubBRt}
Here we prove that the time derivative of the Birkhoff-Rott integral is at the same level as $\da^2 z$.

\begin{lemma}
\label{lemmaBRt}
The following estimate holds
\begin{eqnarray}\label{BRt}
\|BR_t\|_{H^k}\leq
C\left(\|\F(z)\|^2_{L^\infty}+\|z\|^2_{H^{k+2}}+\|\om\|^2_{H^{k+1}}+\|\varphi\|^2_{H^{k+1}}+\sum_{l=0}^4\frac{1}{m(q^l)(t)}\right)^j,
\end{eqnarray}
for $k\geq 2$, where $C$ and $j$ are constants that depend only on $k$.
\end{lemma}
\begin{proof}
We proceed as in \cite[Section 6.5]{Cordoba-Cordoba-Gancedo:interface-water-waves-2d}, where $BR_t$ appears in the formula \eqref{R-T}. We use \eqref{nszt} and \eqref{nswt} to bound $z_t$ and $\omega_t$ in $BR_t$ respectively.
\end{proof}

\subsubsection{Estimates for the Rayleigh-Taylor function $\sigma$}
\label{subsubsigma}
Here we prove that the Rayleigh-Taylor function is at the same level as $\da^2 z$.

\begin{lemma}
\label{lemmasigmaHk}
The following estimate holds
\begin{eqnarray}\label{nswbis}
\|\sigma \|_{H^k}\leq
C\left(\|\F(z)\|^2_{L^\infty}+\|z\|^2_{H^{k+2}}+\|\om\|^2_{H^{k+1}}+\|\varphi\|^2_{H^{k+1}}+\sum_{l=0}^4\frac{1}{m(q^l)(t)}\right)^j,
\end{eqnarray}
for $k\geq 2$, where $C$ and $j$ are constants that depend only on $k$.
\end{lemma}
\begin{proof}
We proceed as in \cite[Section 6.5]{Cordoba-Cordoba-Gancedo:interface-water-waves-2d} using formula \eqref{R-T}. There is a new term in the definition of $\sigma$, namely $Q\left|BR(z,\om) + \frac{\om}{|z_{\al}|^{2}}z_{\al}\right|^{2}(\nabla Q)(z) \cdot z_{\al}^{\perp}$, but this term is less singular than $BR_t(z,\om) \cdot z_{\al}^{\perp}$. Hence, the new term causes no trouble.
\end{proof}

\subsubsection{Estimates for $\sigma_t$}
\label{subsubsigmat}
In this section we obtain an upper bound for the $L^\infty$ norm of $\sigma_t$ that will be used in the energy inequalities and in the treatment of the Rayleigh-Taylor condition.

\begin{lemma}
\label{lemmasigmat}
 The following estimate holds
\begin{eqnarray}\label{nlinftyst}
\|\sigma_t \|_{L^\infty}\leq
C\left(\|\F(z)\|^2_{L^\infty}+\|z\|^2_{H^4}+\|\om\|^2_{H^3}+\|\varphi\|^2_{H^3}+\sum_{l=0}^4\frac{1}{m(q^l)(t)}\right)^j,
\end{eqnarray}
where $C$ and $j$ are universal constants.
\end{lemma}

\begin{proof}
Again, as in the previous subsection, the new term is less singular than the terms treated in \cite[Section 6.6]{Cordoba-Cordoba-Gancedo:interface-water-waves-2d}. Hence we deal with them with no problem.
\end{proof}

\subsubsection{Energy estimates on the curve}

In this section we give the proof of the following  lemma when,
again, $k=4$. The case $k>4$ is left to the reader. Regarding  $\|\da^4 z\|^2_{L^2}$ let us remark that we have
\begin{equation}
\label{controlzH4}
\|\da^4 z\|^2_{L^2}(t)=\int_\T
\frac{Q^2\sigma |z_\al|^2}{Q^2\sigma|z_\al|^2}|\da^4 z|^2d\al\leq
\frac{|z_\al|^2}{m(Q^2\sigma)(t)} \int_\T \frac{Q^2\sigma}{|z_\al|^2}|\da^4 z|^2d\al.
\end{equation}

\begin{lemma}
\label{lemmaenergyS}
Let $z(\al,t)$ and $\om(\al,t)$ be a solution of (\ref{zeq}-\ref{eqomega}).
 Then, the following a priori estimate holds:
\begin{align}
\begin{split}\label{eec}
\frac{d}{dt}\Big(\|z\|^2_{H^{k-1}}+\int_{-\pi}^\pi\frac{Q^2\sigma}{|z_\al|^2}|\da^k z|^2 d\al+\|\F(z)\|_{L^\infty}^2\Big)&\leq S(t)+CE^p(t),
\end{split}
\end{align}
for
\begin{equation}\label{fS}
S(t)=\int_{-\pi}^{\pi}2Q^2\sigma\frac{\da^k z\cdot z_\al^{\perp}}{|
z_\al|^3}\la(\da^{k-1}\varphi)d\al,
\end{equation}
and $k\geq 4$, where $C$ and $p$ are constants that depend only on $k$.
\end{lemma}
(The term $S(t)$ is uncontrolled but it will appear
in the equation of the evolution of $\varphi$ with the opposite
sign.)\\

\begin{proof} Using \eqref{nszt} and \eqref{controlzH4} one gets easily
\begin{align*}
\begin{split}
\frac{d}{dt}\|z\|^2_{H^3}&\leq C\int_{-\pi}^{\pi}(|z(\al)||z_t(\al)|+|\da^3z(\al)||\da^3z_t(\al)|)d\al\\
&\leq CE^p(t).
\end{split}
\end{align*}
We obtain
\begin{align*}
\begin{split}
\frac{d}{dt}\|\F(z)\|^2_{L^\infty}&\leq CE^p(t)
\end{split}
\end{align*}
in a similar manner as in \cite[Section 7.2]{Cordoba-Cordoba-Gancedo:interface-water-waves-2d}. It remains to deal with the quantity
\begin{align*}
\begin{split}
\frac{d}{dt} \int_{-\pi}^\pi\frac{Q^2\sigma}{|z_\al|^2}|\partial_{\al}^4 z|^2d\al
= &\int_{-\pi}^{\pi}\Big(\frac{Q^2\sigma}{|z_\al|^2}\Big)_t
|\partial_{\al}^4 z|^2d\al+\int_{-\pi}^\pi\frac{2Q^2\sigma}{|z_\al|^2}\partial_{\al}^4 z\cdot\partial_{\al}^4 z_t d\al \\
=&I_1+I_2.
\end{split}
\end{align*}

The bounds \eqref{nszt}, \eqref{controlzH4} and \eqref{nlinftyst} give us
$$I_1\leq C E^p(t).$$

Next for $I_2$ we write
\begin{align*}
\begin{split}
I_2&=\int_{-\pi}^\pi\frac{2Q^2\sigma}{|z_\al|^2} \partial_{\al}^4 z\cdot
\partial_{\al}^4 (Q^2BR) d\al+
\int_{-\pi}^\pi\frac{2Q^2\sigma}{|z_\al|^2}\partial_{\al}^4 z\cdot \da^4(c\da z) d\al \\
&=J_1+J_2.
\end{split}
\end{align*}
The most singular terms in $J_1$ are given by $K_1$, $K_2$, $K_3$ and $K_4$:

$$
K_1=\frac{1}{\pi}PV\int_{-\pi}^{\pi}\int_{-\pi}^\pi
\frac{Q^4\sigma}{|z_\al|^2}\partial_{\al}^4z\cdot\frac{(\da^4z-\da^4z')^\bot}{|z-z'|^2}\om'd\beta
d\alpha,$$

\begin{align*}
\begin{split}
K_2&=-\frac{2}{\pi}PV\int_{-\pi}^\pi\int_{-\pi}^\pi\frac{Q^4\sigma}{|z_\al|^2}\partial_{\al}^4z\cdot\frac{(z-z')^{\bot}}{|z-z'|^4}
(z-z')\cdot(\da^4z-\da^4z')\om'd\beta d\alpha,
\end{split}
\end{align*}
$$
K_3=\frac{1}{\pi}PV\int_{-\pi}^\pi\int_{-\pi}^\pi
\frac{Q^4\sigma}{|z_\al|^2} \da^4z\cdot
\frac{(z-z')^{\bot}}{|z-z'|^2}\da^4\om'd\beta d\al,
$$
and
$$
K_4=\frac{2}{\pi}\int_{-\pi}^\pi
\frac{Q^3\sigma}{|z_\al|^2} \da^4z\cdot BR \,\grad Q(z)\cdot \da^4 z d\al,
$$
where the prime denotes a function in the variable $\al-\beta$, i.e.
$f'=f(\al-\beta)$.

Then we write:
\begin{align*}
\begin{split}
K_1&=\frac{1}{\pi}PV\int_{-\pi}^\pi\int_{-\pi}^\pi
\frac{Q^4(\al)\sigma(\al)}{|z_\al|^2}
\partial_{\al}^4
z(\al)\cdot\frac{(\da^4z(\al)-\da^4z(\beta))^\bot}{|z(\al)-z(\beta)|^2}\om(\beta)d\beta
d\alpha\\
&=\frac{1}{\pi|z_\al|^2}PV\int_{-\pi}^\pi\int_{-\pi}^\pi\partial_{\al}^4
z(\al)\cdot\frac{(\da^4z(\al)-\da^4z(\beta))^\bot}{|z(\al)-z(\beta)|^2}\frac{Q^4(\al)\sigma(\al)\om(\beta)+Q^4(\beta)\sigma(\beta)\om(\alpha)}{2}d\beta
d\alpha\\
&\quad+\frac{1}{\pi|z_\al|^2}PV\int_{-\pi}^\pi\int_{-\pi}^\pi\partial_{\al}^4
z(\al)\cdot\frac{(\da^4z(\al)-\da^4z(\beta))^\bot}{|z(\al)-z(\beta)|^2}\frac{Q^4(\al)\sigma(\al)\om(\beta)-Q^4(\beta)\sigma(\beta)\om(\alpha)}{2}d\beta
d\alpha\\
&=L_1+L_2.
\end{split}
\end{align*}
That is, we have performed a manipulation in $K_1$,
allowing us to show that $L_1$, its  most singular term,  vanishes:
\begin{align*}
\begin{split}
L_1&=\frac{-1}{\pi|z_\al|^2}PV\int_{-\pi}^\pi\int_{-\pi}^\pi\partial_{\al}^4
z(\beta)\cdot\frac{(\da^4z(\al)-\da^4z(\beta))^\bot}{|z(\al)-z(\beta)|^2}\frac{Q^4(\al)\sigma(\al)\om(\beta)+Q^4(\beta)
\sigma(\beta)\om(\alpha)}{2}d\beta d\alpha\\
&=\frac{1}{2\pi|z_\al|^2}PV\int_{-\pi}^\pi\int_{-\pi}^\pi(\partial_{\al}^4
z(\al)\!-\!\partial_{\al}^4 z(\beta))\!\cdot\!
\frac{(\da^4z(\al)\!-\!\da^4z(\beta))^\bot}{|z(\al)\!-\!z(\beta)|^2}\frac{Q^4(\al)\sigma(\al)\om(\beta)\!+\!
Q^4(\beta)\sigma(\beta)\om(\alpha)}{2}d\beta d\alpha\\
&=0.
\end{split}
\end{align*}
The term $L_2$ involves a S.I.O. (Singular Integral Operator) acting on $\da^{4} z(\al)$ thanks to the minus sign between the two terms $Q^4 \sigma \om$. One can show that
$$
L_2\leq C\|\F(z)\|^2_{L^\infty}\|z\|^k_{H^3}\|\om\|_{C^{1,\delta}}
\|\sigma\|_{C^{1,\delta}}\|Q^4\|_{C^{1,\delta}}\|\da^4 z\|^2_{L^2}\leq CE^p(t).
$$

Inside $K_2$ we find that $(z-z')\cdot(\da^4z-\da^4z')$ can be written as follows:
\begin{align}
\begin{split}\label{decomextra}
(z-z')\cdot(\da^4z-\da^4z')&=(z-z'-z_\al\beta)\cdot(\da^4z-\da^4z')\\
  &\quad -\beta (z_\al-z'_\al)\cdot \da^4 z'\\
  &\quad +\beta (z_\al\cdot \da^4 z-z'_\al\cdot \da^4z'),
\end{split}
\end{align}
then using that
\begin{equation}
\label{zdz4trick}
z_\al\cdot \da^4 z=-3\da^2 z\cdot \da^3 z,
\end{equation}
we can split $K_2$ as a sum of S.I.O.s operating on
$\da^4z(\al)$, plus a kernel of the form $\frac{\eta(\alpha,\beta)}{\beta^{2}}$ acting on $\da^{2} z \cdot \da^{3} z$ with $\eta \in \mathcal{C}^{2}$ allowing us to obtain again the estimate
$$K_2\leq C E^p(t).$$
Note that below we will also use a variant of \eqref{zdz4trick}, namely
\begin{equation}
\label{zdz4trick2}
z_\al\cdot (\da^{4} z - \da^{4} z') = (z'_\al-z_\al)\cdot\da^{4} z'-3(\da^2 z\cdot \da^3z-\da^2 z'\cdot \da^3z') .
\end{equation}
The term $K_3$ is a sum of
$$
L_3=\frac{1}{\pi}\int_{-\pi}^\pi\frac{Q^4\sigma}{|z_\al|^2} \da^4z\cdot\int_{-\pi}^\pi
\left[\frac{(z-z')^{\bot}}{|z-z'|^2}-\frac{z^\bot_\al}{| z_\al|^22\tan(\beta/2)}\right]\da^4\om'd\beta d\al,
$$
plus the following term:
$$
L_4=\int_{-\pi}^\pi Q^4\sigma\frac{\da^4 z\cdot z^{\bot}_\al}{|z_\al|^4}H(\da^4\om)d\al.
$$
We can integrate by parts on $L_3$ with respect to $\beta$ since
$\da^4 \om'=-\partial_{\beta}(\da^3\om')$. This
calculation gives a S.I.O. acting on $\da^{3} \om$ which can be estimated as before.

 Next in  $L_4$ we write
\begin{align*}
\begin{split}
L_4&=\int_{-\pi}^\pi Q^4\sigma\frac{\da^4 z\cdot z^{\bot}_\al}{|z_\al|^4} \la(\da^3\om) d\al
\end{split}
\end{align*}
and decompose further
\begin{align*}
\begin{split}
L_4&=\int_{-\pi}^\pi 2Q^2\sigma\frac{\da^4 z\cdot z^{\bot}_\al}{|z_\al|^3} \left[\frac{Q^2}{2|z_\al|}\la(\da^3\om)-\la(\da^3(\frac{Q^2}{2|z_\al|}\om))\right] d\al\\
&\quad+\int_{-\pi}^\pi 2Q^2\sigma\frac{\da^4 z\cdot z^{\bot}_\al}{|z_\al|^3}\la(\da^3\vp)d\al\\
&\quad+\int_{-\pi}^\pi 2Q^2\sigma\frac{\da^4 z\cdot z^{\bot}_\al}{|z_\al|^3}\la(\da^3(c|z_\al|))d\al\\
&=M_{-1}+S+M_1,
\end{split}
\end{align*}
for $S(t)$ given by \eqref{fS}. In $M_{-1}$ we find a commutator that allows us to obtain
$$
M_{-1}\leq CE^p(t).
$$
Using \eqref{defc} for $M_1$ we have
\begin{align*}
\begin{split}
M_1&=-2\int_{-\pi}^\pi H\big(Q^2\sigma\frac{\da^4 z\cdot z^{\bot}_\al}{|z_\al|^3}\big)\da^4(c|z_\al|))d\al=N_1+N_2+N_3+N_4,
\end{split}
\end{align*}
where
$$N_1=2\int_{-\pi}^\pi H\big(Q^2\sigma\frac{\da^4 z\cdot z^{\bot}_\al}{|z_\al|^3}\big)Q^2
BR_\al\cdot \frac{\da^4 z}{|z_\al|} d\al,$$
$$N_2=2\int_{-\pi}^\pi H\big(Q^2\sigma\frac{\da^4 z\cdot z^{\bot}_\al}{|z_\al|^3}\big)Q^2
\da^4 BR\cdot \frac{z_\al}{|z_\al|} d\al,$$
$$N_3=4\int_{-\pi}^\pi H\big(Q^2\sigma\frac{\da^4 z\cdot z^{\bot}_\al}{|z_\al|^3}\big)Q \grad Q\cdot \da^4 z
  BR\cdot\frac{z_\al}{|z_\al|} d\al,$$
and $N_4$ is given by the rest
of the terms which can be controlled easily by the estimates from Section \ref{subsubBR} for the Birkhoff-Rott integral.

 Regarding  $N_1$ a
straightforward calculation gives
\begin{align*}
\begin{split}
N_1&\leq CE^p(t),
\end{split}
\end{align*}
and analogously for $N_3$
\begin{align*}
\begin{split}
N_3&\leq CE^p(t).
\end{split}
\end{align*}
Again, in $N_2$ we consider the most singular terms given by
$$
O_1=\int_{-\pi}^\pi H\big(Q^2\sigma\frac{\da^4 z\cdot z^{\bot}_\al}{|z_\al|^3}\big)\frac{z_\al}{|z_\al|} \cdot
\frac{Q^2}{\pi}PV\int_{-\pi}^\pi \frac{(\da^4 z-\da^4
z')^{\bot}}{|z-z'|^2}\om'd\beta d\al,
$$
$$
O_2=-\int_{-\pi}^\pi H\big(Q^2\sigma\frac{\da^4 z\cdot z^{\bot}_\al}{|z_\al|^3}\big)\frac{z_\al}{|z_\al|} \cdot
\frac{Q^2}{2\pi}PV\int_{-\pi}^\pi
\frac{(z-z')^{\bot}}{|z-z'|^4}(z-z')\cdot(\da^4z-\da^4z')\om'
d\al d\beta,
$$
$$
O_3=2\int_{-\pi}^\pi H\big(Q^2\sigma\frac{\da^4 z\cdot z^{\bot}_\al}{|z_\al|^3}\big)Q^2\frac{z_\al}{|z_\al|} \cdot
BR(z,\da^4\om) d\al.
$$
Using the decomposition \eqref{decomextra} we can easily
estimate $O_2$ as in our discussion of $K_2$.

In $O_3$ we find
$$z_\al\cdot BR(z,\da^4 \om)=\frac{z_\al}{2\pi}\cdot \int_{-\pi}^{\pi}\frac{(z- z'-z_\al\beta)^{\perp}}{|z- z'|^{2}}\da^4\omega'd\beta.$$
Above we can integrate by parts as in our discussion of $L_3$. We find that
$$ O_3 \leq CE^{p}(t).$$

Next we split $O_1$ into a S.I.O. acting on
$(\da^4 z)^{\bot}$, which can be estimated as before, plus the
term
$$
P_1=\int_{-\pi}^\pi H\big(Q^2\sigma\frac{\da^4 z\cdot z^{\bot}_\al}{|z_\al|^3}\big)Q^2\om\frac{z_\al}{|z_\al|^3} \cdot
\la((\da^4 z)^{\bot}) d\al.$$

Then the following estimate for the commutator
$$\|Q^2\om\frac{z_\al}{|z_\al|^3} \cdot
\la((\da^4 z)^{\bot})-\la(Q^2\om\frac{z_\al}{|z_\al|^3} \cdot (\da^4 z)^{\bot})\|_{L^2}\leq CE^p(t)$$
yields
$$
P_1\leq CE^p(t)+R
$$
where
$$
R=-\int_{-\pi}^\pi Q^2\sigma\frac{\da^4 z\cdot z^{\bot}_\al}{|z_\al|^3}\partial_\al(Q^2\om\frac{z_\al}{|z_\al|^3}
(\da^4 z)^{\bot}) d\al.
$$
Using that
$$
\int_{-\pi}^\pi Hf(\al)\la g(\al)d\al=-\int_{-\pi}^\pi f(\al)\da
g(\al)d\al.
$$
We can write
$$
R=\int_{-\pi}^\pi Q^2\sigma\frac{\da^4 z\cdot z^{\bot}_\al}{|z_\al|^3}\partial_\al(Q^2\om)\frac{\da^4 z\cdot z^{\bot}_\al}{|z_\al|^3} d\al+\int_{-\pi}^\pi Q^2\sigma\frac{\da^4 z\cdot z^{\bot}_\al}{|z_\al|^3}Q^2\om \partial_\al(\frac{\da^4 z\cdot z^{\bot}_\al}{|z_\al|^3}) d\al
$$
and a straightforward integration by parts let us control $R$. This calculation allows us to get
$$
P_1\leq CE^p(t).
$$
We can easily show that
$$
K_4\leq CE^p(t)
$$
because we can bound $Q^3\sigma BR\grad Q$ in $L^\infty$. So finally we have controlled  $J_1$ in the following manner:

$$
J_1\leq CE^p(t)+ S.
$$

To finish the proof  let us observe that the term $J_2$ can be
estimated integrating by parts,  using the identity $\da^4
z\cdot \da z=-3\da^3 z\cdot \da^2 z$ to
treat its most singular component. We have obtained
$$
\int_{\T}\frac{Q^2\sigma}{|z_\al|^2}\partial_{\al}^4 z\cdot\da
z \da^4c d\al =
3\int_{\T}\frac{1}{|z_\al|^2}\da(Q^2\sigma\partial_{\al}^3 z\cdot\da^2
z)\da^3c d\al
$$
and this yields the desired control.
\end{proof}

\subsubsection{Energy estimates for $\om$}

In this section we show the following result.

\begin{lemma}
Let $z(\al,t)$ and $\om(\al,t)$ be a solution of
(\ref{zeq}-\ref{eqomega}).  Then, the following
a priori estimate holds:
\begin{align}
\begin{split}\label{eec}
\frac{d}{dt}\|\om\|^2_{H^{k-2}}(t)&\leq CE^p(t),
\end{split}
\end{align}
for $k\geq 4$, where $C$ and $p$ are constants that depend only on $k$.
\end{lemma}

\begin{proof}
We will discuss the case $k=4$, leaving the
other cases to the reader. Formula \eqref{nswt} shows easily that
$$
\frac{d}{dt}\|\om\|^2_{H^2}(t)\leq  \left(\|\F(z)\|^2_{L^\infty}(t)+\|z\|^2_{H^{4}}(t)+\|\om\|^2_{H^{3}}(t)+\|\varphi\|^2_{H^{3}}(t)+\sum_{l=0}^4\frac{1}{m(q^l)(t)}\right)^j
$$
which together with \eqref{nsw} yields
$$
\frac{d}{dt}\|\om\|^2_{H^2}(t)\leq  CE^p(t).
$$
\end{proof}

\subsubsection{Finding the Rayleigh-Taylor function in the equation for $\da\vp_{t}$.}

In this section we get the R-T function in the evolution equation for $\da\vp_t$.

\begin{lemma}
\label{lemmaphit}
Let $z(\al,t)$ and $\om(\al,t)$ be a solution of
(\ref{zeq}-\ref{eqomega}).  Then, the following identity holds:
\begin{align}
\begin{split}\label{eec}
\vp_{\al t}=\text{NICE}-\frac{\vp}{|z_\al|}\vp_{\al\al}-Q^2\sigma\frac{z_{\al\al}\cdot z^{\bot}_\al}{|z_\al|^3}
\end{split}
\end{align}
where $\text{NICE}$ satisfies
\begin{equation}\label{nicee}
\int_{-\pi}^\pi \Lambda(\da^{k-1}\vp)\da^{k-2}(\text{NICE})d\al\leq CE^p(t)
\end{equation}
and
\begin{equation}\label{nicen}
\|\text{NICE}\|_{H^{k-2}}\leq CE^p(t)
\end{equation}
for $k\geq 4$, where $C$ and $p$ are constants that depend only on $k$.
\end{lemma}

\begin{proof} We will give the proof for $k=4$. From now on, when we show that a term $f$ satisfies
\begin{equation*}
\int_{-\pi}^\pi \Lambda(\da^{3}\vp)\da^{2}f d\al\leq CE^p(t)\quad\mbox{and}\,\|f\|_{H^{2}}\leq CE^p(t)
\end{equation*}
we say that this term is ``NICE". Then, $f$ becomes part of NICE and by abuse of notation we denote $f$ by NICE. Notice that, whenever we can estimate the $L^{2}$ norm of $\Lambda^{1/2} \da^{2} f$ by $CE^{p}(t)$, then $f$ is NICE.

We use \eqref{evpt} to compute
\begin{align*}
\vp_{\al t} = & -B(t) \vp_{\al} - \da\left(\frac{Q^{2}}{2|z_\al|}\left(\frac{\vp^{2}}{Q^{2}}\right)_{\al}\right) \underbrace{-\big(Q^2\big(BR_t\cdot\frac{z_\al}{|z_\al|}}_{(3a)}
+\frac{(gP^{-1}_{2}(z))_{\al}}{|z_\al|} \big)\big)_{\al} \\
& + \left(Q Q_{t} \frac{\om}{|z_\al|}\right)_{\al} - \left(2c BR \cdot \frac{z_\al}{|z_\al|} Q Q_{\al}\right)_{\al} - \left(\frac{Q_{\al}}{Q}c^{2}|z_\al|\right)_{\al}
-\left(\frac{Q^3}{|z_\al|}|BR|^{2}Q_{\al}\right)_{\al} \\
& \underbrace{- (c|z_\al|)_{\al t}}_{(3b)}.
\end{align*}

Expanding $(3) = (3a) + (3b)$:
\begin{align*}
(3) = & (3a) + (3b) = -\left(Q^{2}BR_t\cdot\frac{z_\al}{|z_\al|}\right)_{\al} - (c|z_\al|)_{\al t} \\
= & -\left(Q^{2}BR_t\right)_{\al}\cdot\frac{z_\al}{|z_\al|}-Q^{2}BR_t\cdot\left(\frac{z_\al}{|z_\al|}\right)_{\al}
- \left(|z_\al|B(t) - (Q^{2} BR)_{\al} \cdot \frac{z_\al}{|z_\al|}\right)_{t} \\
= &- Q^{2}BR_{t} \cdot  \left(\frac{z_\al}{|z_\al|}\right)_{\al} - \left(|z_\al|B(t)\right)_{t} +   (Q^{2} BR)_{\al} \cdot \left(\frac{z_\al}{|z_\al|}\right)_{t}
+ 2(Q Q_{t} BR)_{\al} \cdot\frac{z_\al}{|z_\al|}.
\end{align*}
We use that
$$  \left(\frac{z_\al}{|z_\al|}\right)_{\al} = \frac{z_{\al \al} \cdot z_\al^{\perp}}{|z_\al|^{2}} \cdot\frac{z_\al^{\perp}}{|z_\al|};
\quad \left(\frac{z_\al}{|z_\al|}\right)_{t} = \frac{z_{\al t} \cdot z_\al^{\perp}}{|z_\al|^{2}} \cdot\frac{z_\al^{\perp}}{|z_\al|}$$
to find

\begin{align}
\vp_{\al t} = & \underbrace{-B(t) \vp_{\al}}_{(4)} - \underbrace{\frac{\da^{2}(\vp^{2})}{2|z_\al|}}_{(5)}
+ \underbrace{\da\left(\frac{Q_{\al}}{|z_\al|Q}\vp^{2}\right)}_{(6)}
- Q^{2} BR_t \cdot z_\al^{\perp} \frac{z_{\al \al} \cdot z_\al^{\perp}}{|z_\al|^{3}} -
(|z_\al|B(t))_{t} \nonumber \\
& + \underbrace{(Q^{2} BR)_{\al} \cdot z_\al^{\perp} \frac{z_{\al t} \cdot z_\al^{\perp}}{|z_\al|^{3}}}_{(13)}
+ \underbrace{2(QQ_{t} BR)_{\al}\cdot \frac{z_\al}{|z_\al|}}_{(7)}
- \underbrace{\left(Q^{2}\frac{(gP^{-1}_{2}(z))_{\al}}{|z_\al|}\right)_{\al}}_{(8)}
+ \underbrace{\left(QQ_{t} \frac{\om}{|z_\al|}\right)_{\al}}_{(9)} \nonumber \\
& - \underbrace{\left(2c BR \cdot \frac{z_\al}{|z_\al|}Q Q_{\al}\right)_{\al}}_{(10)}
- \underbrace{\left(\frac{Q_{\al}}{Q}c^{2}|z_\al|\right)_{\al}}_{(11)}
-\underbrace{\left(\frac{Q^{3}}{|z_\al|}|BR|^{2}Q_{\al}\right)_{\al}}_{(12)}.
\label{phiat412}
\end{align}

The term $(|z_\al|B(t))_{t}$ depends only on $t$ so it is going to be part of NICE.

$$ (4) = -B(t) \vp_{\al} \text{ is NICE (at the level of } \vp_{\al}).$$

$$ (5) = -\frac{\da^{2}(\vp^{2})}{2|z_\al|}=-\frac{\vp^2_\al}{|z_\al|}-\frac{\vp}{|z_\al|}\vp_{\al\al}.$$
The first term is at the level of $\vp_\al$ so it is NICE. The second one is the transport term which appears in \eqref{eec}.
$$ (6) = \da\left(\frac{Q_{\al}}{|z_\al|Q}\vp^{2}\right) = -\frac{Q^{2}_{\al}\vp^{2}}{|z_\al|Q^{2}}
+ \frac{2Q_{\al}\vp \vp_{\al}}{|z_\al|Q} + \frac{\vp^{2}}{Q}\left(\frac{Q_{\al}}{|z_\al|}\right)_{\al}.$$
Above we find the first term at the level of $z_\al$ so it is NICE. The second term is at the level of $\vp_{\al}$ so it is NICE. We write the last one as
$$\frac{\vp^{2}}{Q}\left(\frac{Q_{\al}}{|z_\al|}\right)_{\al}
= \frac{\vp^{2}}{Q}z_\al \cdot \left(\nabla^{2}Q(z)\cdot \frac{z_\al}{|z_\al|}\right)
+ \frac{\vp^{2}}{Q} \nabla Q \cdot z_\al^{\perp} \frac{z_{\al \al} \cdot z_\al^{\perp}}{|z_\al|^{3}}.$$

The first term is at the level of $z_\al$ so it is NICE. For the second term we have used that

$$  \left(\frac{z_\al}{|z_\al|}\right)_{\al} = \frac{z_{\al \al} \cdot z_\al^{\perp}}{|z_\al|^{2}} \cdot\frac{z_\al^{\perp}}{|z_\al|}.$$
Finally:

$$ (6) = \text{NICE } + \frac{\vp^{2}}{Q} \nabla Q \cdot z_\al^{\perp}
\frac{z_{\al \al} \cdot z_\al^{\perp}}{|z_\al|^{3}}.$$

$$ (7) = 2(Q Q_{t} BR)_{\al}\cdot \frac{z_\al}{|z_\al|}
= 2Q_{\al} Q_{t} BR \cdot  \frac{z_\al}{|z_\al|}
+ 2Q \left(\frac{Q_{t}}{|z_\al|}\right)_{\al} BR \cdot z_\al
+ 2QQ_{t} BR_{\al} \cdot \frac{z_\al}{|z_\al|}.$$

The first term is at the level of $z_\al, z_{t}, BR \sim z_\al$ so it is NICE. We use that

$$ \frac{Q_{t\al}}{|z_\al|} = \frac{Q_{\al t}}{|z_\al|}
= \frac{(\nabla Q(z) \cdot z_\al)_{t}}{|z_\al|}
= \left(\nabla Q(z) \cdot \frac{z_\al}{|z_\al|}\right)_{t} - \nabla Q(z) \cdot z_\al \left(\frac{1}{|z_\al|}\right)_{t}.$$

Using equation \eqref{defb}

$$ \frac{z_\al \cdot z_{\al t}}{|z_\al|^{2}} = B(t)$$
and
$$  \left(\frac{z_\al}{|z_\al|}\right)_{t} = \frac{z_{\al t} \cdot z_\al^{\perp}}{|z_\al|^{2}} \cdot\frac{z_\al^{\perp}}{|z_\al|}$$
we find that
\begin{align}
\begin{split}\label{star}
\frac{Q_{t\al}}{|z_\al|}=
& z_t \cdot \left(\nabla^2 Q(z) \cdot \frac{z_\al}{|z_\al|}\right) + \nabla Q(z) \cdot z_\al^{\perp} \frac{z_{\al t} \cdot z_\al^{\perp}}{|z_\al|^{3}}
\\
&+ \nabla Q(z) \cdot\frac{z_\al}{|z_\al|} B(t).
\end{split}
\end{align}

That yields

\begin{align*}
(7) & = 2(QQ_{t} BR)_{\al} \frac{z_\al}{|z_\al|}
= \text{NICE } + \underbrace{2Q BR \cdot z_\al\, z_t \cdot \left(\nabla^2 Q(z) \cdot \frac{z_\al}{|z_\al|}\right)}_{\text{NICE (at the level of }z_\al, z_t, BR)} \\
& + 2Q BR \cdot z_\al \nabla Q(z) \cdot z_\al^{\perp} \frac{z_{\al t} \cdot z_\al^{\perp}}{|z_\al|^{3}}
+ \underbrace{2Q BR \cdot z_\al \nabla Q(z) \cdot\frac{z_\al}{|z_\al|} B(t)}_{\text{NICE (at the level of }z_\al, z_t, BR)} \\
& + 2QQ_{t} BR_{\al} \cdot \frac{z_\al}{|z_\al|}.
\end{align*}

Finally:

\begin{align*}
(7) & = 2(QQ_{t} BR)_{\al} \frac{z_\al}{|z_\al|}
= \text{NICE } + 2Q BR \cdot z_\al \nabla Q(z) \cdot z_\al^{\perp} \frac{z_{\al t} \cdot z_\al^{\perp}}{|z_\al|^{3}}
+ 2QQ_{t} BR_{\al} \cdot \frac{z_\al}{|z_\al|}.
\end{align*}

\begin{align*}
(8) & = -\left(Q^2\frac{(gP^{-1}_{2}(z))_{\al}}{|z_\al|}\right)_{\al}
= -\left(Q^2\nabla gP^{-1}_{2}(z)\cdot \frac{z_\al}{|z_\al|}\right)_{\al}\\
&= -\underbrace{2 Q \nabla Q\cdot z_\al \nabla gP^{-1}_{2}(z) \cdot \frac{z_\al}{|z_\al|}}_{\text{NICE (at the level of }z_\al)}
-\underbrace{Q^{2} z_\al \cdot \left(\nabla^{2}gP^{-1}_{2}(z) \cdot \frac{z_\al}{|z_\al|}\right)}_{\text{NICE (at the level of }z_\al)}\\
&\quad- Q^{2} \nabla gP^{-1}_{2}(z) \cdot z_\al^{\perp} \frac{z_{\al \al} \cdot z_\al^{\perp}}{|z_\al|^{3}},
\end{align*}

which means

\begin{align*}
(8) & = -\left(Q^{2}\frac{(gP^{-1}_{2}(z))_{\al}}{|z_\al|}\right)_{\al}
= \text{NICE } - Q^{2} \nabla gP^{-1}_{2}(z) \cdot z_\al^{\perp} \frac{z_{\al \al} \cdot z_\al^{\perp}}{|z_\al|^{3}}.
\end{align*}
Next
\begin{align*}
(9) & = \left(QQ_{t} \frac{\om}{|z_\al|}\right)_{\al}
= \underbrace{Q_{\al} Q_{t}\frac{\om}{|z_\al|}}_{\text{NICE (at the level of }z_\al, z_t)} + Q \frac{Q_{\al t}}{|z_\al|}\om
+ QQ_{t} \left(\frac{\om}{|z_\al|}\right)_{\al}.
\end{align*}

We use \eqref{star} to deal with $\frac{Q_{\al t}}{|z_\al|}$. We find that
\begin{align*}
(9) & = \left(QQ_{t} \frac{\om}{|z_\al|}\right)_{\al}
= \text{NICE } + Q\om \nabla Q(z) \cdot z_\al^{\perp} \frac{z_{\al t} \cdot z_\al^{\perp}}{|z_\al|^{3}}
+ QQ_{t} \left(\frac{\om}{|z_\al|}\right)_{\al}.
\end{align*}
For the next term
\begin{align*}
(10) & = - \left(2c BR \cdot \frac{z_\al}{|z_\al|}Q Q_{\al}\right)_{\al}
= \underbrace{- 2c BR \cdot \frac{z_\al}{|z_\al|}Q^{2}_{\al}}_{\text{NICE as before}}
- \left(2c BR \cdot \frac{z_\al}{|z_\al|}\right)_{\al}QQ_{\al} \\
&\quad - 2c BR \cdot z_\al Q \nabla Q(z) \cdot z_\al^{\perp} \frac{z_{\al \al} \cdot z_\al^{\perp}}{|z_\al|^{3}}
 \underbrace{- 2c BR \cdot \frac{z_\al}{|z_\al|}Q z_\al \cdot (\nabla^{2} Q(z)) \cdot z_\al}_{\text{NICE as before}}.
\end{align*}
Therefore

\begin{align*}
(10)  = - \left(2c BR \cdot \frac{z_\al}{|z_\al|}Q Q_{\al}\right)_{\al}
& = \text{NICE }
- \left(2c BR \cdot \frac{z_\al}{|z_\al|}\right)_{\al}QQ_{\al} \\
&\quad -2c BR \cdot z_\al Q \nabla Q \cdot z_\al^{\perp} \frac{z_{\al \al} \cdot z_\al^{\perp}}{|z_\al|^{3}}.
\end{align*}
Next
\begin{align*}
(11)  = - \left(\frac{Q_{\al}}{Q}c^{2}|z_\al|\right)_{\al}
& = - \left(c^{2}|z_\al|\right)_{\al}\frac{Q_{\al}}{Q}
- \frac{c^{2}|z_\al|^2}{Q}\nabla Q(z) \cdot z_\al^{\perp} \frac{z_{\al \al} \cdot z_\al^{\perp}}{|z_\al|^{3}} \\
&\quad - \frac{z_\al \cdot (\nabla^{2} Q(z) \cdot z_\al)}{Q}c^2|z_\al| + \frac{Q_{\al}^2}{Q^2}c^2|z_\al|.
\end{align*}

The fact that the last two terms are NICE, allows us to find that

\begin{align*}
(11)  = - \left(\frac{Q_{\al}}{Q}c^{2}|z_\al|\right)_{\al}
& = \text{NICE } - \left(c^{2}|z_\al|\right)_{\al}\frac{Q_{\al}}{Q}
- \frac{c^{2}|z_\al|^2}{Q}\nabla Q(z) \cdot z_\al^{\perp} \frac{z_{\al \al} \cdot z_\al^{\perp}}{|z_\al|^{3}}.
\end{align*}

Finally:
\begin{align*}
(12) = - \left(\frac{Q^{3}}{|z_\al|}|BR|^{2}Q_{\al}\right)_{\al}
& = \underbrace{- \frac{3Q^2Q_{\al}^{2}|BR|^{2}}{|z_{\al}|}}_{\text{NICE}} - \frac{Q^3}{|z_\al|}(|BR|^2)_{\al}Q_{\al} \\
& \quad\underbrace{- \frac{Q^3}{|z_\al|}|BR|^2z_\al \cdot(\nabla^{2}Q(z) \cdot z_\al)}_{\text{NICE}} - Q^{3}|BR|^{2}\nabla Q(z) \cdot z_\al^{\perp} \frac{z_{\al \al} \cdot z_\al^{\perp}}{|z_\al|^{3}}
\end{align*}
which implies that

\begin{align*}
(12) = - \left(\frac{Q^{3}}{|z_\al|}|BR|^{2}Q_{\al}\right)_{\al}
& = \text{NICE } - \frac{Q^3}{|z_\al|}(|BR|^2)_{\al}Q_{\al} - Q^{3}|BR|^{2}\nabla Q(z) \cdot z_\al^{\perp} \frac{z_{\al \al} \cdot z_\al^{\perp}}{|z_\al|^{3}}.
\end{align*}
We gather all the formulas from (4) to (12), keeping term (13) unchanged. They yield:
\begin{align*}
\vp_{\al t}= & \text{NICE}-\frac{\vp}{|z_\al|}\vp_{\al\al}\\
& + \underbrace{\frac{\vp^{2}}{Q} \nabla Q(z) \cdot z_\al^{\perp} \frac{z_{\al \al} \cdot z_\al^{\perp}}{|z_\al|^{3}}}_{(16a)} \underbrace{- Q^{2} BR_t \cdot z_\al^{\perp} \frac{z_{\al \al} \cdot z_\al^{\perp}}{|z_\al|^{3}}}_{(15a)}
\underbrace{- Q^{2} \nabla gP_{2}^{-1}(z) \cdot z_\al^{\perp} \frac{z_{\al \al} \cdot z_\al^{\perp}}{|z_\al|^{3}}}_{(15b)}
\end{align*}
\begin{align*}
\qquad & + \underbrace{Q \om \nabla Q(z) \cdot z_\al^{\perp} \frac{z_{\al t} \cdot z_\al^{\perp}}{|z_\al|^{3}}}_{(18a)}
+ \underbrace{QQ_{t} \left(\frac{\om}{|z_\al|}\right)_{\al}}_{(14a)}
+ \underbrace{2Q BR \cdot z_\al \nabla Q(z) \cdot z_\al^{\perp} \frac{z_{\al t} \cdot z_\al^{\perp}}{|z_\al|^{3}}}_{(18b)} \\
& + \underbrace{QQ_{t} 2 BR_{\al} \cdot \frac{z_\al}{|z_\al|}}_{(14b)}
\underbrace{- \left(2 c BR \cdot \frac{z_\al}{|z_\al|}\right)_{\al}QQ_{\al}}_{(17a)}
\underbrace{-2c BR \cdot z_\al Q \nabla Q(z) \cdot z_\al^{\perp} \frac{z_{\al \al} \cdot z_\al^{\perp}}{|z_\al|^{3}}}_{(16b)}
\end{align*}
\begin{align*}
\qquad& \underbrace{-(c^{2}|z_\al|)_{\al}\frac{Q_{\al}}{Q}}_{(17b)}
 \underbrace{-\frac{c^2|z_\al|^{2}}{Q}\nabla Q(z) \cdot z_\al^{\perp} \frac{z_{\al \al} \cdot z_\al^{\perp}}{|z_\al|^{3}}}_{(16c)}
 \underbrace{-\frac{Q^3}{|z_\al|}(|BR|^{2})_{\al}Q_{\al}}_{(17c)}\\
& \underbrace{-Q^3|BR|^{2}\nabla Q(z) \cdot z_\al^{\perp} \frac{z_{\al \al} \cdot z_\al^{\perp}}{|z_\al|^{3}}}_{(16d)}
 + (Q^2 BR)_{\al} \cdot z_\al^{\perp} \frac{z_{\al t} \cdot z_\al^{\perp}}{|z_\al|^{3}}.
\end{align*}

We compute
\begin{align*}
(14) & =(14a)+(14b)= QQ_{t} \left(\frac{\om}{|z_\al|}\right)_{\al}
+ QQ_{t} 2 BR_{\al} \cdot \frac{z_\al}{|z_\al|}\\
&= 2 \frac{Q_{t}}{Q}Q^2\left(\frac{\om}{2|z_\al|}\right)_{\al}
+ 2\frac{Q_{t}}{Q}Q^2 BR_{\al} \cdot \frac{z_\al}{|z_\al|} \\
& = 2\frac{Q_{t}}{Q}\vp_{\al} - 2\frac{Q_{t}}{Q}(Q^{2})_{\al}\frac{\om}{2|z_\al|}
- 2\frac{Q_{t}}{Q}(Q^{2})_{\al}BR\cdot \frac{z_\al}{|z_\al|}
+ 2\frac{Q_{t}}{Q}(|z_\al|B(t)).
\end{align*}
The last formula allows us to conclude that (14)=NICE.

We reorganize gathering $$(15)=(15a)+(15b),$$ $$(16)=(16a)+(16b)+(16c)+(16d),$$ $$(17)=(17a)+(17b)+(17c)$$ and $$(18)=(18a)+(18b)$$
as follows:
\begin{align*}
\vp_{\al t}  = &\text{NICE }-\frac{\vp}{|z_\al|}\vp_{\al\al}\underbrace{- Q^2(BR_t \cdot z_\al^{\perp} + \nabla gP_{2}^{-1}(z) \cdot z_\al^{\perp})\frac{z_{\al \al} \cdot z_\al^{\perp}}{|z_\al|^{3}}}_{(15)} \\
& \underbrace{- Q^3\left(|BR|^{2} + \frac{c^2|z_\al|^{2}}{Q^{4}}+2c \frac{BR \cdot z_\al}{Q^{2}} - \frac{\vp^{2}}{Q^4}\right)\nabla Q(z) \cdot z_\al^{\perp} \frac{z_{\al \al} \cdot z_\al^{\perp}}{|z_\al|^{3}}}_{(16)} \\
& + (Q^{2} BR)_{\al} \cdot z_\al^{\perp} \frac{z_{\al t} \cdot z_\al^{\perp}}{|z_\al|^{3}}
+\underbrace{(Q \om + 2Q BR \cdot z_\al)\nabla Q(z) \cdot z_\al^{\perp}
\frac{z_{\al t} \cdot z_\al^{\perp}}{|z_\al|^{3}}}_{(18)} \\
& \underbrace{- \left(\frac{Q^{3}(|BR|^{2})_{\al}}{|z_\al|} + \frac{(c^{2}|z_\al|)_{\al}}{Q}
+ \left(2c BR \cdot \frac{z_\al}{|z_\al|}\right)_{\al}Q\right)Q_{\al}}_{(17)}.
\end{align*}

We add and subtract terms in order to find the R-T condition. We recall here that
\begin{align*}
\si & \equiv \left(BR_t + \frac{\vp}{|z_{\al}|}BR_{\al}\right) \cdot z_{\al}^{\perp} + \frac{\om}{2|z_{\al}|^{2}}\left(z_{\al t} + \frac{\vp}{|z_{\al}|}z_{\al \al}\right)\cdot z_{\al}^{\perp}\\
&\quad + Q\left|BR + \frac{\om}{2|z_{\al}|^{2}}z_{\al}\right|^{2}\nabla Q(z) \cdot z_{\al}^{\perp} + \nabla gP_{2}^{-1}(z) \cdot z_{\al}^{\perp}.
\end{align*}

Then, we find

\begin{align*}
\vp_{\al t}  = & \text{NICE }-\frac{\vp}{|z_\al|}\vp_{\al\al}\\
& - Q^{2}\left(\left(BR_t + \frac{\vp}{|z_\al|}BR_{\al}\right) \cdot z_\al^{\perp}
+ \frac{\om}{2|z_\al|^{2}}\left(z_{\al t} + \frac{\vp}{|z_\al|}z_{\al \al}\right)\cdot z_\al^{\perp}
+ \nabla gP_{2}^{-1}(z) \cdot z_\al^{\perp}\right)\frac{z_{\al \al} \cdot z_\al^{\perp}}{|z_\al|^{3}} \\
& \underbrace{+ (Q^{2} BR)_{\al} \cdot z_\al^{\perp} \frac{z_{\al t} \cdot z_\al^{\perp}}{|z_\al|^{3}}
+ Q^{2}\left(\frac{\vp}{|z_\al|}BR_{\al} \cdot z_\al^{\perp} + \frac{\om}{2|z_\al|^{2}}\left(z_{\al t} + \frac{\vp}{|z_\al|}z_{\al \al}\right) \cdot z_\al^{\perp}\right)\frac{z_{\al \al} \cdot z_\al^{\perp}}{|z_\al|^{3}}}_{(19)}\\
& - Q^3\left(|BR|^{2} + \frac{c^2|z_\al|^{2}}{Q^{4}}+2c \frac{BR \cdot z_\al}{Q^{2}} - \frac{\vp^{2}}{Q^4}\right)\nabla Q(z) \cdot z_\al^{\perp} \frac{z_{\al \al} \cdot z_\al^{\perp}}{|z_\al|^{3}} \\
& + (Q \om + 2Q BR \cdot z_\al)\nabla Q(z) \cdot z_\al^{\perp} \frac{z_{\al t} \cdot z_\al^{\perp}}{|z_\al|^{3}} \\
& - \left(\frac{Q^{3}(|BR|^{2})_{\al}}{|z_\al|} + \frac{(c^{2}|z_\al|)_{\al}}{Q}
+ \left(2c BR \cdot \frac{z_\al}{|z_\al|}\right)_{\al}Q\right)Q_{\al}.
\end{align*}

Line (19) can be written as

\begin{align*}
(19)= &  (Q^{2} BR)_{\al} \cdot z_\al^{\perp} \frac{z_{\al t} \cdot z_\al^{\perp}}{|z_\al|^{3}}
+ Q^{2}BR_{\al} \cdot z_\al^{\perp}\frac{\vp}{|z_\al|}\frac{z_{\al \al} \cdot z_\al^{\perp}}{|z_\al|^{3}}
+ \frac{Q^{2}\om}{2|z_\al|^{2}}\left(z_{\al t} + \frac{\vp}{|z_\al|}z_{\al \al}\right)
 \cdot z_\al^{\perp}\frac{z_{\al \al} \cdot z_\al^{\perp}}{|z_\al|^{3}} \\
 =& (Q^{2} BR)_{\al} \cdot z_\al^{\perp} \frac{z_{\al t} \cdot z_\al^{\perp}}{|z_\al|^{3}}
+ (Q^{2}BR)_{\al} \cdot z_\al^{\perp}\frac{\vp}{|z_\al|}\frac{z_{\al \al} \cdot z_\al^{\perp}}{|z_\al|^{3}}
+ \frac{Q^{2}\om}{2|z_\al|^{2}}\left(z_{\al t}\cdot z_\al^{\perp} + \frac{\vp}{|z_\al|}z_{\al \al}\cdot z_\al^{\perp}\right)
\frac{z_{\al \al} \cdot z_\al^{\perp}}{|z_\al|^{3}} \\
& - 2QQ_{\al} BR \cdot z_\al^{\perp} \frac{\vp}{|z_\al|}\frac{z_{\al \al} \cdot z_\al^{\perp}}{|z_\al|^{3}} \\
=&(Q^{2} BR)_{\al} \cdot z_\al^{\perp}\frac{1}{|z_\al|^{3}}
\left( z_{\al t} \cdot z_\al^{\perp}
+ \frac{\vp}{|z_\al|}z_{\al \al} \cdot z_\al^{\perp}\right)
+ \frac{Q^{2}\om}{2|z_\al|^{2}}\frac{1}{|z_\al|^{3}}\left(z_{\al t}\cdot z_\al^{\perp} + \frac{\vp}{|z_\al|}z_{\al \al}\cdot z_\al^{\perp}\right)
z_{\al \al} \cdot z_\al^{\perp} \\
& - 2QQ_{\al} BR \cdot z_\al^{\perp} \frac{\vp}{|z_\al|}\frac{z_{\al \al} \cdot z_\al^{\perp}}{|z_\al|^{3}} \\
=& \frac{1}{|z_\al|^{3}}\left( z_{\al t} \cdot z_\al^{\perp}
+ \frac{\vp}{|z_\al|}z_{\al \al} \cdot z_\al^{\perp}\right)\left((Q^{2} BR)_{\al} \cdot z_\al^{\perp}
+ \frac{Q^{2}\om}{2|z_\al|^{2}}z_{\al \al} \cdot z_\al^{\perp}\right)\\
&- 2QQ_{\al} BR \cdot z_\al^{\perp} \frac{\vp}{|z_\al|}\frac{z_{\al \al} \cdot z_\al^{\perp}}{|z_\al|^{3}}.
\end{align*}

We expand $z_{\al t}$ to find
\begin{align*}
(19) =& \frac{1}{|z_\al|^{3}}\left((Q^{2} BR)_{\al} \cdot z_\al^{\perp}
+ \frac{Q^{2}\om}{2|z_\al|^{2}}z_{\al \al} \cdot z_\al^{\perp}\right)^2\\
& - 2QQ_{\al} BR \cdot z_\al^{\perp} \frac{\vp}{|z_\al|}\frac{z_{\al \al} \cdot z_\al^{\perp}}{|z_\al|^{3}}.
\end{align*}

We denote
\begin{equation}
\label{spiral}
D(\al) = (Q^{2} BR)_{\al} \cdot z_\al^{\perp}
+ \frac{Q^{2}\om}{2|z_\al|^{2}}z_{\al \al} \cdot z_\al^{\perp}.
\end{equation}

We claim that
\begin{equation}\label{fD}
D(\al) = \text{AN}(\al) + |z_\al|H(\da \vp)
\end{equation}

where
\begin{equation}
\label{defAN}
\|\text{AN}\|_{H^3}\leq C E^p(t).
\end{equation}

That means
$$ (D(\al))^2 = \text{NICE}.$$

Thus
\begin{align*}
(19) = \text{NICE} - 2QQ_{\al} BR \cdot z_\al^{\perp} \frac{\vp}{|z_\al|}\frac{z_{\al \al} \cdot z_\al^{\perp}}{|z_\al|^{3}}.
\end{align*}

We write
\begin{align*}
D(\al)= &  \underbrace{2QQ_{\al} BR \cdot z_\al^{\perp}}_{\text{part of AN, at the level of }z_\al}
+ \underbrace{Q^{2}\frac{1}{2\pi}PV\int \frac{(z_\al - z'_\al) \cdot z_\al}{|z-z'|^{2}}\om'd\beta}_{\text{part of AN, we use \eqref{zdz4trick2} } } \\
& \underbrace{- Q^{2}\frac{1}{\pi}PV\int \frac{(z- z') \cdot z_\al}{|z-z'|^{4}}(z- z')\cdot(z_\al - z_\al')\om'd\beta}_{\text{part of AN, we use \eqref{decomextra} and \eqref{zdz4trick}}} \\
& + \underbrace{Q^2 BR(z,\om_{\al}) \cdot z_\al^{\perp}}_{\text{AN + $\frac{Q^2}{2} H(\omega_\al) $}} + \frac{Q^{2}\om}{2|z_\al|^{2}}z_{\al \al} \cdot z_\al^{\perp}.
\end{align*}

Therefore
\begin{align*}
D(\al) & = \text{AN} + |z_\al|Q^{2}H\left(\left(\frac{\om}{2|z_\al|}\right)_{\al}\right) + \frac{Q^{2}\om}{2|z_\al|^{2}}z_{\al \al} \cdot z_\al^{\perp} \\
& = \text{AN} + |z_\al|H\left(\left(\frac{Q^{2}\om}{2|z_\al|}\right)_{\al}\right) + \frac{Q^{2}\om}{2|z_\al|^{2}}z_{\al \al} \cdot z_\al^{\perp} \\
& = \text{AN} + |z_\al|H(\da \vp) + H\left((c |z_\al|^{2})_{\al}\right)+ \frac{Q^{2}\om}{2|z_\al|^{2}}z_{\al \al} \cdot z_\al^{\perp} \\
& = \text{AN} + |z_\al|H(\vp_{\al}) - H\left((Q^{2}BR)_{\al} \cdot z_\al\right)+ \frac{Q^{2}\om}{2|z_\al|^{2}}z_{\al \al} \cdot z_\al^{\perp}.
\end{align*}
We have
\begin{align*}
(Q^{2}BR)_{\al} \cdot z_\al  = &\underbrace{2QQ_{\al} BR \cdot z_\al}_{\text{AN}}
+ Q^{2}\frac{1}{2\pi}PV\int \frac{(z_\al - z'_\al)^{\perp} \cdot z_\al}{|z-z'|^{2}}\om'd\beta \\
& \underbrace{- Q^{2}\frac{1}{\pi}PV\int \frac{(z- z')^{\perp} \cdot z_\al}{|z-z'|^{4}}(z-z')\cdot(z_\al - z'_\al)\om'd\beta}_{\text{AN, we use that }(z- z')^{\perp} \cdot z_\al = (z - z' - \beta z_{\al})^{\perp} \cdot z_{\al}} \\
& + \underbrace{Q^{2}\frac{1}{2\pi}PV\int \frac{(z - z')^{\perp} \cdot z_\al}{|z-z'|^{2}}\om'_\al d\beta}_{\text{AN, we use that }(z- z')^{\perp} \cdot z_\al = (z - z' - \beta z_{\al})^{\perp} \cdot z_{\al}}.
\end{align*}
For the second term on the right one finds
\begin{align*}
\da^3\Big(\frac{Q^{2}}{2\pi}PV\int &\frac{(z_\al - z'_\al)^{\perp} \cdot z_\al}{|z-z'|^{2}}\om'd\beta\Big)=
\frac{Q^{2}}{2\pi}PV\int \frac{(\da^4z_\al - \da^4z')^{\perp} \cdot z_\al}{|z-z'|^{2}}\om'd\beta\\
&+\frac{Q^{2}}{2\pi}PV\int \frac{(z_\al - z'_\al)^{\perp}}{|z-z'|^{2}}\om'd\beta\cdot \da^4z+\frac{Q^{2}}{2\pi}PV\int \frac{(z_\al - z'_\al)^{\perp} \cdot z_\al}{|z-z'|^{2}}\da^3\om'd\beta\\
&-\frac{Q^{2}}{\pi}PV\int \frac{(z_\al - z'_\al)^{\perp} \cdot z_\al}{|z-z'|^{4}}(z-z')\cdot(\da^3z-\da^3z')\om'd\beta+\mbox{l.o.t.},
\end{align*}
where in l.o.t. we gather the terms of lower order. Then, all the terms above can be estimated in $L^2$ but the first one on the right. That is equal to
$$
\frac{1}{2}H\left(Q^2\frac{\da^5z^{\perp}\cdot z_\al}{|z_\al|^{2}}\om\right)
$$
plus a commutator which can be estimated in $L^2$. This means that
\begin{align*}
(Q^{2}BR)_{\al} \cdot z_\al & = \text{AN} + \frac{1}{2}H\left(Q^2\frac{z^{\perp}_{\al\al} \cdot z_\al}{|z_\al|^{2}}\om\right).
\end{align*}
Taking Hilbert transforms:
\begin{align*}
-H\left((Q^{2}BR)_{\al} \cdot z_\al\right) & = \text{AN} - \frac{1}{2}H^2\left(Q^2\frac{z^{\perp}_{\al\al} \cdot z_\al}{|z_\al|^{2}}\om\right)
= \text{AN} + \frac{1}{2}Q^2\frac{z^{\perp}_{\al\al} \cdot z_\al}{|z_\al|^{2}}\om.
\end{align*}
Using that $z^{\perp}_{\al\al} \cdot z_\al = -z_{\al\al} \cdot z_\al^{\perp}$ we complete the proof of \eqref{fD}. Thus (19) yields

\begin{align*}
\vp_{\al t}  =& \text{NICE}-\frac{\vp}{|z_\al|}\vp_{\al\al}\\
& - Q^{2}\left(\left(BR_t + \frac{\vp}{|z_\al|}BR_{\al}\right) \cdot z_\al^{\perp}
+ \frac{\om}{2|z_\al|^{2}}\left(z_{\al t} + \frac{\vp}{|z_\al|}z_{\al \al}\right)\cdot z_\al^{\perp}
+ \nabla gP_{2}^{-1}(z) \cdot z_\al^{\perp}\right)\frac{z_{\al \al} \cdot z_\al^{\perp}}{|z_\al|^{3}} \\
& - Q^3\left(|BR|^{2} + \frac{c^2|z_\al|^{2}}{Q^{4}}+2c \frac{BR \cdot z_\al}{Q^{2}} - \frac{\vp^{2}}{Q^4}\right)\nabla Q(z) \cdot z_\al^{\perp} \frac{z_{\al \al} \cdot z_\al^{\perp}}{|z_\al|^{3}} \\
& + (Q \om + 2Q BR \cdot z_\al)\nabla Q(z) \cdot z_\al^{\perp} \frac{z_{\al t} \cdot z_\al^{\perp}}{|z_\al|^{3}} \\
& \underbrace{- \left(\frac{Q^{3}(|BR|^{2})_{\al}}{|z_\al|} + \frac{(c^{2}|z_\al|)_{\al}}{Q}
+ \left(2c BR \cdot \frac{z_\al}{|z_\al|}\right)_{\al}Q\right)Q_{\al}}_{(20)} \\
& \underbrace{- 2QQ_{\al} BR \cdot z_\al^{\perp} \frac{\vp}{|z_\al|} \frac{z_{\al \al} \cdot z_\al^{\perp}}{|z_\al|^{3}}}_{(21)}.
\end{align*}

For (20) we write
\begin{align*}
|z_t|^{2} & = Q^{4}|BR|^{2} + c^{2}|z_\al|^{2} + 2 Q^{2} c BR \cdot z_\al \\
\Rightarrow \frac{|z_t|^{2}}{Q|z_\al|} & =  \frac{Q^{3}|BR|^{2}}{|z_\al|} + \frac{c^{2}|z_\al|}{Q} + 2 Q c BR \cdot \frac{z_\al}{|z_\al|}.\\
\end{align*}

Now
\begin{align*}
(20) = \text{NICE } - \frac{(|z_t|^{2})_{\al}}{Q|z_\al|}Q_{\al},
\end{align*}
which means
\begin{align*}
(20)+(21) & = \text{NICE} - \frac{(|z_t|^{2})_{\al}}{Q|z_\al|}Q_{\al}
- 2QQ_{\al} BR \cdot z_\al^{\perp} \frac{\vp}{|z_\al|} \frac{z_{\al \al} \cdot z_\al^{\perp}}{|z_\al|^{3}}.
\end{align*}

We write
\begin{align*}
z_{\al t} & = \underbrace{(z_{\al t} \cdot z_\al)}_{\text{only depends on }t} \frac{z_\al}{|z_\al|^{2}}
+ (z_{\al t} \cdot z_\al^{\perp}) \frac{z_\al^{\perp}}{|z_\al|^{2}} \\
& = \underbrace{B(t)}_{\text{See } \eqref{defb}}z_\al + ((Q^{2} BR)_{\al} \cdot z_\al^{\perp} + c z_{\al \al} \cdot z_\al^{\perp})\frac{z_\al^{\perp}}{|z_\al|^{2}} \\
& = B(t)z_\al +\underbrace{D(\al)}_{\text{ as in }\eqref{spiral}}\frac{z_\al^{\perp}}{|z_\al|^{2}}- \frac{\vp}{|z_\al|}z_{\al \al} \cdot z_\al^{\perp}
\frac{z_\al^{\perp}}{|z_\al|^{2}}.
\end{align*}

Writing $z_t = Q^{2}BR + c z_\al$ we compute
\begin{align*}
z_{\al t} \cdot z_{t} & = \underbrace{Q^{2} BR \cdot z_\al B(t)}_{\text{NICE}} + \underbrace{D Q^{2} BR \cdot \frac{z_\al^{\perp}}{|z_\al|^{2}}}_{\text{NICE because }D \text{ is nice}} \\
& - \frac{\vp}{|z_\al|}z_{\al \al} \cdot z_\al^{\perp} Q^{2} BR \cdot \frac{z_\al^{\perp}}{|z_\al|^{2}}
+\underbrace{cB(t)|z_\al|^{2}}_{\text{NICE}}.
\end{align*}
To simplify we write
$$ z_{\al t} \cdot z_{t} = \text{NICE}- \frac{\vp}{|z_\al|}z_{\al \al} \cdot z_\al^{\perp} Q^{2} BR \cdot \frac{z_\al^{\perp}}{|z_\al|^{2}}.$$

Setting the above formula in the expression of (20)+(21) allows us to find
$$ (20) + (21) = \text{NICE }.$$

This yields
\begin{align*}
\vp_{\al t}  =& \text{NICE }-\frac{\vp}{|z_\al|}\vp_{\al\al}\\
&- Q^{2}\left(\left(BR_t + \frac{\vp}{|z_\al|}BR_{\al}\right)
+ \frac{\om}{2|z_\al|^{2}}\left(z_{\al t} + \frac{\vp}{|z_\al|}z_{\al \al}\right)
+ \nabla gP_{2}^{-1}(z)\right)\cdot z_\al^{\perp}\frac{z_{\al \al} \cdot z_\al^{\perp}}{|z_\al|^{3}} \\
&\quad - Q^3\left(|BR|^{2} + \frac{c^2|z_\al|^{2}}{Q^{4}}+2c \frac{BR \cdot z_\al}{Q^{2}} - \frac{\vp^{2}}{Q^4}\right)\nabla Q(z) \cdot z_\al^{\perp} \frac{z_{\al \al} \cdot z_\al^{\perp}}{|z_\al|^{3}} \\
&\quad + (Q \om + 2Q BR \cdot z_\al)\nabla Q(z) \cdot z_\al^{\perp} \frac{z_{\al t} \cdot z_\al^{\perp}}{|z_\al|^{3}}.
\end{align*}

We now complete the formula for $\si$ in \eqref{R-T} to find
\begin{align*}
\vp_{\al t} = & \text{NICE}-\frac{\vp}{|z_\al|}\vp_{\al\al} - Q^{2}\si\frac{z_{\al \al} \cdot z_\al^{\perp}}{|z_\al|^{3}} \\
& + \underbrace{Q^3 \left|BR + \frac{\om}{2|z_\al|^{2}}z_\al\right|^{2}
\nabla Q \cdot z_\al^{\perp} \frac{z_{\al \al} \cdot z_\al^{\perp}}{|z_\al|^{3}}}_{(22)} \\
& + \underbrace{Q^3\left(-|BR|^{2} - \frac{c^2|z_\al|^{2}}{Q^{4}}-2c \frac{BR \cdot z_\al}{Q^{2}} + \frac{\vp^{2}}{Q^4}\right)\nabla Q(z) \cdot z_\al^{\perp} \frac{z_{\al \al} \cdot z_\al^{\perp}}{|z_\al|^{3}}}_{(23)} \\
& + \underbrace{(Q\om + 2Q BR \cdot z_\al)\nabla Q(z) \cdot z_\al^{\perp} \frac{z_{\al t} \cdot z_\al^{\perp}}{|z_\al|^{3}}}_{(24)}.
\end{align*}

Expanding
$$ \frac{\vp^{2}}{Q^{4}} = \frac{\om^{2}}{4|z_\al|^{2}} + \frac{c^{2}|z_\al|^{2}}{Q^{4}} - \frac{\om c}{Q^{2}}$$

we find

\begin{align*}
(22) + (23) = Q^3\left(\frac{\om^{2}}{2|z_\al|^{2}} + BR \cdot z_\al \frac{\om}{|z_\al|^{2}}-2c \frac{BR \cdot z_\al}{Q^{2}}
- \frac{\om c}{Q^{2}}\right)\nabla Q(z) \cdot z_\al^{\perp} \frac{z_{\al \al} \cdot z_\al^{\perp}}{|z_\al|^{3}}.
\end{align*}

Writing
\begin{align*}
z_{\al t} \cdot z_\al^{\perp} = (Q^{2} BR)_{\al}\cdot z_\al^{\perp} + c z_{\al \al} \cdot z_\al^{\perp}
\end{align*}
we obtain that
\begin{align*}
(24)=&   \left(Q \om + 2Q BR \cdot z_\al\right)\nabla Q(z) \cdot z_\al^{\perp} \frac{(Q^{2} BR)_{\al} \cdot z_\al^{\perp}}{|z_\al|^{3}} \\
& + \left(Q \om + 2Q BR \cdot z_\al\right)\nabla Q(z) \cdot z_\al^{\perp} c\frac{z_{\al \al} \cdot z_\al^{\perp}}{|z_\al|^{3}}.
\end{align*}

Thus
\begin{align*}
(22)+(23)+(24) & = Q^3\left(\frac{\om^{2}}{2|z_\al|^{2}} + BR \cdot z_\al \frac{\om}{|z_\al|^{2}}
\right)\nabla Q(z) \cdot z_\al^{\perp} \frac{z_{\al \al} \cdot z_\al^{\perp}}{|z_\al|^{3}} \\
&\quad + \left(Q \om + 2Q BR \cdot z_\al\right)\nabla Q(z) \cdot z_\al^{\perp} \frac{(Q^{2} BR)_{\al} \cdot z_\al^{\perp}}{|z_\al|^{3}}\\
& = Q\nabla Q(z) \cdot z_\al^{\perp} \left(\om + 2BR \cdot z_\al\right)\left(\frac{Q^{2}\om}{2|z_\al|^{2}}\frac{z_{\al \al} \cdot z_\al^{\perp}}{|z_\al|^{3}} + \frac{(Q^{2}BR)_{\al} \cdot z_\al^{\perp}}{|z_\al|^{3}}\right)\\
& = Q \nabla Q(z) \cdot z_\al^{\perp} \left(\om + 2BR \cdot z_\al\right)\frac{1}{|z_\al|^{3}}D(\al)\\
& = \text{NICE}.
\end{align*}

Finally, we obtain
$$ \vp_{\al t} = \text{NICE}-\frac{\vp}{|z_\al|}\vp_{\al\al}-Q^{2}\si\frac{z_{\al \al} \cdot z_\al^{\perp}}{|z_\al|^{3}}.$$

\end{proof}

\begin{corollary}
If we disregard the condition on the $H^{k-2}$ norm for the definition of the NICE terms, imposing only the first condition, then
$$ \vp_{\al t} = \text{NICE}-Q^{2}\si\frac{z_{\al \al} \cdot z_\al^{\perp}}{|z_\al|^{3}}.$$
\end{corollary}

\subsubsection{Higher order derivatives of $\sigma$}
In this section we deal with the highest order derivative of the R-T function. We show that
\begin{lemma}
Let $z(\al,t)$ and $\om(\al,t)$ be a solution of
(\ref{zeq}-\ref{eqomega}).  Then, the following identity holds:
\begin{align}
\begin{split}\label{rtc}
\da^{k-1}(Q^2\sigma)=\text{ANN}+|z_\al|H(\da^{k-1}\vp_t)+\vp H(\da^k\vp)
\end{split}
\end{align}
where $\text{ANN}$ satisfies
\begin{equation}\label{niceep}
\|\text{ANN}\|_{L^2}\leq CE^p(t)
\end{equation}
for $k\geq 4$, where $C$ and $p$ are constants that depend only on $k$.
\end{lemma}

\begin{proof}
We show the proof for $k=4$. From now on, if a term $f$ satisfies

$$ \|f\|_{L^{2}} \leq CE^{p}(t)$$

we say that this term becomes part of ANN. By abuse of notation we will denote $f$ by ANN. We recall

\begin{align*}
Q^{2}\si =&  Q^{2}\left(BR_t + \frac{\vp}{|z_{\al}|}BR_{\al}\right) \cdot z_{\al}^{\perp}
+ \frac{Q^2\om}{2|z_{\al}|^{2}}\left(z_{\al t} + \frac{\vp}{|z_{\al}|}z_{\al \al}\right) \cdot z_{\al}^{\perp} \\
& + \underbrace{Q^{3}\left|BR + \frac{\om}{2|z_{\al}|^{2}}z_{\al}\right|^{2}\nabla Q \cdot z_{\al}^{\perp}}_{\text{this term is in }H^3 \text{ so its third derivative is in  ANN}} + \underbrace{Q^{2} \nabla gP_{2}^{-1}(z) \cdot z_{\al}^{\perp}}_{\text{this term is also in }H^3}.
\end{align*}

We write
\begin{align*}
\frac{Q^2\om}{2|z_{\al}|^{2}}\left(z_{\al t} + \frac{\vp}{|z_{\al}|}z_{\al \al}\right) \cdot z_{\al}^{\perp}
& = \frac{Q^2\om}{2|z_{\al}|^{2}}\left((Q^2 BR)_{\al} \cdot z_{\al}^{\perp} + c z_{\al \al} \cdot z_{\al}^{\perp}
+\frac{\vp}{|z_{\al}|}z_{\al \al} \cdot z_{\al}^{\perp}\right) \\
& = \frac{Q^2\om}{2|z_{\al}|^{2}}\left((Q^2 BR)_{\al} \cdot z_{\al}^{\perp} + \left(c + \frac{\vp}{|z_{\al}|}\right) z_{\al \al} \cdot z_{\al}^{\perp}\right)\\
& = \frac{Q^2\om}{2|z_{\al}|^{2}}\left((Q^2 BR)_{\al} \cdot z_{\al}^{\perp} + \frac{Q^{2}\om}{2|z_{\al}|^{2}} z_{\al \al} \cdot z_{\al}^{\perp}\right)\\
& = \frac{Q^2\om}{2|z_{\al}|^{2}}D(\al).
\end{align*}
Above we use \eqref{spiral} and \eqref{fD} to find
\begin{align}\label{mqiaf}
\frac{Q^2\om}{2|z_{\al}|^{2}}\left(z_{\al t} + \frac{\vp}{|z_{\al}|}z_{\al \al}\right) \cdot z_{\al}^{\perp}
& =\text{AN}+\frac{Q^2\om}{2|z_\al|}H(\vp_\al),
\end{align}
where AN is as in \eqref{defAN}. The remaining terms in $Q^{2}\si$ are

$$ L = Q^{2} BR_t \cdot z_{\al}^{\perp} + \frac{Q^{2}\vp}{|z_{\al}|}BR_{\al} \cdot z_{\al}^{\perp}.$$

We take 3 derivatives and consider the most dangerous characters:
$$
\da^{3}(L) = M_1 + M_2 + M_3 + \text{ANN},
$$
where
$$
M_1 = Q^{2} BR(z,\da^{3} \om_t) \cdot z_{\al}^{\perp} + \frac{Q^{2}\vp}{|z_{\al}|}BR(z,\da^{4}\om) \cdot z_{\al}^{\perp},
$$
\begin{align*}
M_2 =&  Q^{2} \frac{1}{2\pi}\int_{-\pi}^{\pi}\frac{(\da^{3} z_{t} - \da^{3}z'_{t})\cdot z_{\al}}{|z - z'|^{2}}
\om'd\beta \\
& + \frac{Q^{2}\vp}{|z_{\al}|} \frac{1}{2\pi}\int_{-\pi}^{\pi}\frac{(\da^{4} z - \da^{4}z')\cdot z_{\al}}{|z - z'|^{2}}
\om'd\beta,
\end{align*}
\begin{align*}
M_3= &  -\frac{Q^2}{\pi}\int_{-\pi}^{\pi}\frac{(z-z') \cdot z_{\al}}{|z-z'|^{4}}(z-z') \cdot
 (\da^{3}z_t-\da^{3}z'_t)\om'd\beta \\
&  -\frac{Q^2\vp}{|z_{\al}|\pi}\int_{-\pi}^{\pi}\frac{(z-z') \cdot z_{\al}}{|z-z'|^{4}}(z-z') \cdot
(\da^{4}z-\da^{4}z') \om'd\beta.
\end{align*}
Here we point out that in order to deal with $BR_t$ in the less singular terms we proceed using estimate \eqref{BRt}. In $M_2$ we find
$$ M_2 = \frac{Q^2\om}{2|z_{\al}|^{2}}\Lambda(\da^{3}z_t \cdot z_{\al})
+ \frac{Q^{2}\vp\om}{2|z_{\al}|^{3}}\Lambda(\da^{4}z \cdot z_{\al}) + \text{ANN}.$$

For the second term we use the usual trick
$$ \da^{4} z \cdot z_{\al} = -3\da^{3} z \cdot z_{\al \al}.$$

For the first term we recall that
\begin{align*}
& |z_{\al}|^{2} = A(t) \Rightarrow z_{\al} \cdot z_{\al t} = \frac{1}{2}A'(t) \Rightarrow (z_{\al} \cdot z_{\al t})_{\al} = 0 \\
& \Rightarrow z_{\al \al} \cdot z_{\al t} + z_{\al} \cdot z_{\al \al t} = 0 \Rightarrow
z_{\al \al \al} \cdot z_{\al t} + 2z_{\al \al} \cdot z_{\al \al t} + z_{\al} \cdot z_{\al \al \al t} = 0\\
& \Rightarrow z_{\al} \cdot z_{\al \al \al t} = - 2 z_{\al \al} \cdot z_{\al \al t} - z_{\al \al \al} \cdot z_{\al t}.
\end{align*}

This allows us to control $M_2$. For $M_3$ we find

$$ M_3 = -\frac{Q^2\om}{|z_{\al}|^{2}}\Lambda(z_{\al} \cdot \da^{3} z_t)
- \frac{Q^{2}\vp\om}{|z_{\al}|^{3}}\Lambda(z_{\al} \cdot \da^{4}z) + \text{ANN}$$

so it can be estimated as $M_2$. There remains $M_1$. Using that $(z-z')^{\perp}\cdot z_{\al}^{\perp}=(z-z')\cdot z_{\al}$ we find
\begin{align}
\label{ouroboros}
M_1 = \frac{Q^{2}}{2}H(\da^{3}\om_{t}) + \frac{Q^{2} \vp}{2|z_{\al}|}H(\da^{4}\om) + \text{ANN}.
\end{align}

We compute
\begin{align}
\frac{Q^{2}}{2}H(\da^{3}\om_{t}) & = H\left(\da^{3}\left(\frac{Q^{2}\om}{2}\right)_{t}\right) + \text{ANN} \nonumber \\
& = H(\da^{3}(|z_{\al}|\vp)_{t}) + H(\da^{3}(|z_{\al}|c)_{t}) +
\text{ANN} \nonumber \\
& = |z_{\al}|H(\da^{3}\vp_{t}) + H(\da^{2}\partial_t(-(Q^{2} BR)_{\al} \cdot z_{\al})) + \text{ANN}. \label{anchor}
\end{align}

We compute the most singular term in
\begin{align*}
\da^{2}\partial_t(-(Q^{2} BR)_{\al} \cdot z_{\al})  =& -\frac{Q^{2}}{2\pi}\int_{-\pi}^{\pi}\frac{(\da^{3}z_t- \da^{3}z'_t)^{\perp}\cdot z_{\al}}{|z-z'|^{2}}\om'd\beta \\
& +\underbrace{\frac{Q^{2}}{\pi}\int_{-\pi}^{\pi}\frac{(z-z')^{\perp} \cdot z_{\al}}{|z-z'|^{4}}(z-z')\cdot(\da^{3} z_{t}-\da^{3} z'_{t}) \om'd\beta}_{\text{extra cancelation in } (z-z')^{\perp} \cdot z_{\al}=(z-z'-z_\al\beta)^{\perp} \cdot z_{\al} } \\
& - \underbrace{\frac{Q^{2}}{2\pi}PV\int_{-\pi}^{\pi}\frac{(z-z')^{\perp} \cdot z_{\al}}{|z-z'|^{2}}\da^{3}\om'_{t}d\beta}_{\text{extra cancelation as above}} + \text{ANN}.
\end{align*}

This shows that
\begin{align*}
\da^{2}\partial_t(-(Q^{2} BR)_{\al} \cdot z_{\al}) = -\frac{Q^{2}\om}{2|z_{\al}|^{2}}\Lambda(\da^{3} z_{t}^{\perp} \cdot z_{\al})  +\text{ANN}.
\end{align*}

That gives
\begin{align*}
\da^{2}\partial_t(-(Q^{2} BR)_{\al} \cdot z_{\al}) = -\Lambda\left(\frac{Q^{2}\om}{2|z_{\al}|^{2}}\da^{3} z_{t}^{\perp} \cdot z_{\al}\right)  +\text{ANN},
\end{align*}

which implies
\begin{align*}
H(\da^{2}\partial_t(-(Q^{2} BR)_{\al} \cdot z_{\al})) = \da\left(\frac{Q^{2}\om}{2|z_{\al}|^{2}}\da^{3} z_{t}^{\perp} \cdot z_{\al}\right)  + \text{ANN} = -\frac{Q^{2}\om}{2|z_{\al}|^{2}}\da\left(\da^{3} z_{t} \cdot z_{\al}^{\perp}\right)+\text{ANN}.
\end{align*}

Plugging the above formula in \eqref{anchor} we find that
\begin{align*}
\frac{Q^{2}}{2}H(\da^{3}\om_{t}) & = |z_{\al}|H(\da^{3}\vp_{t}) -
\frac{Q^{2}\om}{2|z_{\al}|^{2}}\da\left(\da^{3} z_{t} \cdot z_{\al}^{\perp}\right)
+\text{ANN}\\
& = |z_{\al}|H(\da^{3}\vp_{t}) - \frac{Q^{2}\om}{2|z_{\al}|^{2}}\da\left(\da^{3}(Q^{2} BR) \cdot z_{\al}^{\perp}\right)
- \frac{Q^{2}\om}{2|z_{\al}|^{2}}\da\left(c \da^{4} z \cdot z_{\al}^{\perp}\right)\\
&\quad  +\text{ANN}.
\end{align*}

As we did before, we expand $\da(\da^{3}(Q^{2}BR) \cdot z_{\al}^{\perp})$ to find
\begin{align*}
\da(\da^{3}(Q^{2}BR) \cdot z_{\al}^{\perp})=&2Q\grad Q(z)\cdot\da^4z BR\cdot z_{\al}^{\perp}+
\frac{Q^{2}}{2\pi}PV\int\frac{(\da^4z-\da^4z') \cdot z_{\al}}{|z-z'|^{2}}\om'd\beta\\
&-\frac{Q^{2}}{\pi}PV\int\frac{(z-z') \cdot z_{\al}}{|z-z'|^{4}}(z-z')\cdot(\da^4z-\da^4z')\om'd\beta\\
&+\frac{Q^{2}}{2\pi}PV\int\frac{(z-z')\cdot z_{\al}}{|z-z'|^{2}}\da^{4}\om'd\beta+\text{ANN}.
\end{align*}
Therefore, we can use \eqref{decomextra},\eqref{zdz4trick} and \eqref{zdz4trick} to show that the most dangerous term is given by $Q^{2}\frac{1}{2}H(\da^{4}\om)$. It implies
\begin{align*}
\da(\da^{3}(Q^{2}BR) \cdot z_{\al}^{\perp}) & = Q^{2}\frac{1}{2}H(\da^{4}\om) +\text{ANN}
\end{align*}
and therefore
\begin{align*}
\frac{Q^{2}}{2}H(\da^{3}\om_{t}) = |z_{\al}|H(\da^{3}\vp_{t}) - \frac{Q^{2}\om}{2|z_{\al}|^{2}}\frac{Q^{2}}{2}H(\da^{4}\om)
- \frac{Q^{2}\om}{2|z_{\al}|^{2}}c\da\left(\da^{4} z \cdot z_{\al}^{\perp}\right) + \text{ANN}.\\
\end{align*}

We use the above formula and expand $\varphi$ to find
\begin{align*}
M_1 & = \frac{Q^{2}}{2}H(\da^{3}\om_{t}) + \frac{Q^{2}\vp}{2|z_{\al}|}H(\da^{4}\om) + \text{ANN}
 \\
 & = |z_{\al}|H(\da^{3}\vp_{t}) - \frac{Q^{2}}{2}cH(\da^{4}\om)
- \frac{Q^{2}\om}{2|z_{\al}|^{2}}c\da\left(\da^{4} z \cdot z_{\al}^{\perp}\right) + \text{ANN}\\
& = |z_{\al}|H(\da^{3}\vp_{t}) - c|z_{\al}|H\left(\da^{4}\left(\frac{Q^{2}\om}{2|z_{\al}|}\right)\right)
- \frac{Q^{2}\om}{2|z_{\al}|^{2}}c\da\left(\da^{4} z \cdot z_{\al}^{\perp}\right) + \text{ANN}\\
& = |z_{\al}|H(\da^{3}\vp_{t}) - c|z_{\al}|H\left(\da^{4}\vp\right) - c|z_{\al}|H(\da^{4}(c|z_{\al}|))
- \frac{Q^{2}\om}{2|z_{\al}|^{2}}c\da\left(\da^{4} z \cdot z_{\al}^{\perp}\right) + \text{ANN}.
\end{align*}

We will show that

\begin{equation}
\label{claimsmiley}
 - c|z_{\al}|H(\da^{4}(c|z_{\al}|))
- \frac{Q^{2}\om}{2|z_{\al}|^{2}}c\da\left(\da^{4} z \cdot z_{\al}^{\perp}\right)=\text{ANN}.
\end{equation}
It yields
$$
\frac{Q^{2}}{2}H(\da^{3}\om_{t}) + \frac{Q^{2}\vp}{2|z_{\al}|}H(\da^{4}\om)=|z_{\al}|H(\da^{3}\vp_{t}) - c|z_{\al}|H\left(\da^{4}\vp\right)+\text{ANN}
$$
that together with \eqref{mqiaf} allows us to obtain \eqref{rtc}. We have

\begin{align*}
- c|z_{\al}|H(\da^{4}(c|z_{\al}|))- \frac{Q^{2}\om}{2|z_{\al}|^{2}}c\da\left(\da^{4} z \cdot z_{\al}^{\perp}\right)
& = - cH(\da^{4}(c|z_{\al}|^{2}))- \frac{Q^{2}\om}{2|z_{\al}|^{2}}c\da\left(\da^{4} z \cdot z_{\al}^{\perp}\right) \\
& = cH(\da^{3}((Q^{2} BR)_{\al} \cdot z_{\al}))- \frac{Q^{2}\om}{2|z_{\al}|^{2}}c\da\left(\da^{4} z \cdot z_{\al}^{\perp}\right).
\end{align*}

We repeat the calculation for dealing with the most dangerous terms in
\begin{align*}
\da^{3}((Q^{2} BR)_{\al} \cdot z_{\al}) = \Lambda\left(\da^{4} z^{\perp} \cdot z_{\al} \frac{\om Q^{2}}{2|z_{\al}|^{2}}\right) + \text{ANN}.
\end{align*}

We recognized as before terms in ANN using that $(z-z')^{\perp}\cdot z_\al$ gives an extra cancellation. We find that
\begin{align*}
cH(\da^{3}((Q^{2} BR)_{\al} \cdot z_{\al})) & - \frac{Q^{2}\om}{2|z_{\al}|^{2}}c\da\left(\da^{4} z \cdot z_{\al}^{\perp}\right) \\
& = cH\left(\Lambda\left(\da^{4} z^{\perp} \cdot z_{\al} \frac{\om Q^{2}}{2|z_{\al}|^{2}}\right)\right)- \frac{Q^{2}\om}{2|z_{\al}|^{2}}c\da\left(\da^{4} z \cdot z_{\al}^{\perp}\right) + \text{ANN}\\
& = -c\da\left(\da^{4} z^{\perp} \cdot z_{\al} \frac{\om Q^{2}}{2|z_{\al}|^{2}}\right)- \frac{Q^{2}\om}{2|z_{\al}|^{2}}c\da\left(\da^{4} z \cdot z_{\al}^{\perp}\right) +\text{ANN}.
\end{align*}

Using that $\da^{4} z^{\perp} \cdot z_{\al} = - \da^{4} z \cdot z_{\al}^{\perp} $ we are done proving \eqref{claimsmiley}.
\end{proof}

\subsubsection{Energy estimates for $\varphi$}
\label{subsubenergyphi}

In this section we prove the following result.

\begin{lemma}
\label{lemmaenergyphiS}
Let $z(\al,t)$ and $\om(\al,t)$ be a solution of
(\ref{zeq}-\ref{eqomega}).  Then, the following
a priori estimate holds:
\begin{align}
\begin{split}\label{eec}
\frac{d}{dt}\|\varphi\|^2_{H^{k-\frac12}}(t)&\leq -S(t)+CE^p(t)
\end{split}
\end{align}
for $k\geq 4$, where $C$ and $p$ are constants that depend only on $k$.
\end{lemma}

\begin{proof} We shall present the details in the case $k=4$, leaving the
other cases to the reader.

Using the estimates obtained before one has
$$
\frac{d}{dt}\|\varphi\|^2_{L^2}(t)\leq  C E^p(t).
$$

Developing the derivative using Lemma \ref{lemmaphit}, we get that:

\begin{align}\label{pcch}
\frac{d}{dt}\|\la^{1/2}(\da^3 \varphi)\|^2_{L^2}(t)=&2\int_\T\Lambda(\da^3\varphi)\da^3\varphi_td\al\\
=&I_1+I_2+I_3,\nonumber
\end{align}
where
$$
I_1=2\int_\T \la(\da^3\varphi)\da^2(\text{NICE})d\al,\qquad I_2=-2\int_\T \la(\da^3\varphi)\da^2(\frac{\vp}{|z_\al|}\vp_{\al\al})d\al,
$$
$$
I_3=-2\int_\T \la(\da^3\varphi)\da^2(Q^2\sigma \frac{z_{\al\al}\cdot z^\perp_\al}{|z_\al|^3})d\al.
$$
We use \eqref{nicee} to control $I_1$. The most singular term in $I_2$ is the one given by

\begin{align*}
 -2\int_\T \frac{1}{|z_{\al}|}\la(\da^3\varphi)\da^4 \vp \vp d\al = & 2\int_\T \frac{1}{|z_{\al}|}\Lambda^{1/2}(\da^{3}\vp)\left[\vp \Lambda^{1/2}(\da^{4} \vp) - \Lambda^{1/2}(\vp \da^{4} \vp) \right]d\al \\
& + \int_{\T} \frac{1}{|z_{\al}|}\da \vp |\Lambda^{1/2}(\da^{3}\vp)|^{2} d\al.
\end{align*}

Using the commutator estimate
\begin{equation}
\label{commu}
\|g \Lambda^{1/2}(\da f) - \Lambda^{1/2}(g \da f)\|_{L^{2}} \leq \|g\|_{C^{2}}\|f\|_{H^{1/2}}
\end{equation}

we can bound $I_2$. In $I_3$ we split further considering the most singular terms
$$
J_1=-2\int_\T \la(\da^3\varphi)Q^2\sigma z_{\al\al}\cdot\da^2(\frac{z^\perp_\al}{|z_\al|^3})  d\al,
$$
$$
J_2=-2\int_\T \la(\da^3\varphi)Q^2\sigma \frac{\da^4z\cdot z^\perp_\al}{|z_\al|^3}d\al,
$$
$$
J_3=-2\int_\T \la(\da^3\varphi)\da^2(Q^2\sigma)  \frac{z_{\al\al}\cdot z^\perp_\al}{|z_\al|^3}d\al.
$$
The term $J_1$ can be estimated as before. Recalling \eqref{fS} we see that $J_2=-S(t)$. It remains to control $J_3$ in order to find \eqref{eec}.

We decompose $J_3=K_1+K_2$ where
$$
K_1=2\int_\T H(\da^3\varphi)\da^2(Q^2\sigma) \da( \frac{z_{\al\al}\cdot z^\perp_\al}{|z_\al|^3}) d\al
$$
and
$$
K_2=2\int_\T H(\da^3\varphi)\da^3 (Q^2\sigma) \frac{z_{\al\al}\cdot z^\perp_\al}{|z_\al|^3} d\al.
$$
Inequality \eqref{nswbis} for $k=4$ allows us to obtain
$$
K_1\leq C E^p(t).
$$
To finish the proof we use formula \eqref{rtc} for $k=4$ to find
$K_2=L_1+L_2+L_3$ where
$$
L_1=2\int_\T H(\da^3\varphi)\text{ANN} \frac{z_{\al\al}\cdot z^\perp_\al}{|z_\al|^3} d\al,
$$
$$
L_2=2\int_\T H(\da^3\varphi)|z_\al|H(\da^{3}\vp_t) \frac{z_{\al\al}\cdot z^\perp_\al}{|z_\al|^3} d\al,
$$
$$
L_3=2\int_\T H(\da^3\varphi)\vp H(\da^4\vp) \frac{z_{\al\al}\cdot z^\perp_\al}{|z_\al|^3} d\al.
$$
The term $L_1$ can be easily estimated using \eqref{niceep}.
For $L_2$ we substitute the expression \eqref{eec} for $\da^{3}\vp_t$ to get $L_2 = M_1 + M_2 + M_3$:

$$
M_1=2\int_\T H(\da^3\varphi)|z_\al|H(\da^{2}(\text{NICE})) \frac{z_{\al\al}\cdot z^\perp_\al}{|z_\al|^3} d\al.
$$
$$
M_2=-2\int_\T H(\da^3\varphi)|z_\al|H(\da^{2}(\frac{\vp \vp_{\al \al}}{|z_{\al}|})) \frac{z_{\al\al}\cdot z^\perp_\al}{|z_\al|^3} d\al.
$$
$$
M_3=-2\int_\T H(\da^3\varphi)|z_\al|H(\da^{2}(Q^{2}\si\frac{z_{\al \al} \cdot z_\al^{\perp}}{|z_\al|^{3}} )) \frac{z_{\al\al}\cdot z^\perp_\al}{|z_\al|^3} d\al.
$$
By equation \eqref{nicee}, $M_1$ is bounded. $M_2$ is bounded knowing that we have room for half derivative in the term which is not the third factor. Finally we can bound $M_3$ in virtue of Lemma \ref{lemmasigmaHk}. To finish, in $L_3$ we integrate by parts to find
$$
L_3=-\int_\T |H(\da^3\varphi)|^2\da(\vp \frac{z_{\al\al}\cdot z^\perp_\al}{|z_\al|^3}) d\al\leq CE^p(t)
$$
using Sobolev embedding.
\end{proof}

\subsubsection{Energy estimates for $\D\frac{|z_\al|^2}{m(Q^2\sigma)(t)}+\sum_{l=0}^4\frac{1}{m(q^l)(t)}$.}

\begin{lemma}
Let $z(\al,t)$ and $\om(\al,t)$ be a solution of
(\ref{zeq}-\ref{eqomega}).  Then, the following
a priori estimate holds:
\begin{align}
\label{edqp}
\frac{d}{dt}\Big(\D\frac{|z_\al|^2}{m(Q^2\sigma)(t)}+\sum_{l=0}^4\frac{1}{m(q^l)(t)}\Big)&\leq CE^p(t)
\end{align}
for $k\geq 4$, where $C$ and $p$ are constants that depend only on $k$.
\end{lemma}
\begin{proof}
Inequalities \eqref{nszt} and \eqref{nlinftyst} show that $(Q^2\sigma) \in C^{1}([0,T]\times[-\pi,\pi])$ for some $T$ and therefore $m(Q^2\sigma)(t)$ is a
Lipschitz function differentiable almost everywhere by
Rademacher's theorem.  Let

$$m(Q^2\sigma)(t)=\D\min_{\al\in [-\pi,\pi]} (Q^2\sigma)(\al,t)=(Q^2\sigma)(\al_t,t).$$
We can calculate the derivative of
$m(Q^2\sigma)(t)$, to obtain
$$
(m(Q^2\sigma))'(t)=(Q^2\sigma)_t(\al_t,t)
$$
for almost every $t$. Then it follows that:
$$
\frac{d}{dt}\left(\frac{1}{m(Q^2\sigma)}\right)(t)=-\frac{(Q^2\sigma)_{t}(\al_t,t)}{(m(Q^2\sigma))^2(t)}
$$
 almost everywhere. By using the previous a priori estimates for the $L^{\infty}$ bounds, we get to
 $$ \frac{d}{dt}\left(\frac{|z_\al|^2}{m(Q^2\sigma)}\right)(t) \leq CE^p(t).$$
 On the other hand, we can apply the same argument to $\D \frac{1}{m(q^{l})(t)}$. Denoting again by $\al_t$ the point where the minimum is attained we have that:
$$
\frac{d}{dt}\left(\frac{1}{m(q^{l})}\right)(t)=-\frac{z_t(\al_t,t) \cdot (z(\al_t,t) - q^{l})}{(m(q^{l}))^3(t)}
$$

which again can be easily bounded and we get \eqref{edqp}, as desired.

\end{proof}

\subsection{Proof of short-time existence (Theorem \ref{localexistencetilde})}

To conclude the proof of the local existence, we shall use the previous a priori estimates. We now introduce a regularized version of the evolution equation  which is well-posed for short time independently of the sign condition on $\sigma(\al,t)$ at $t=0$.  But for
$\sigma(\al,0)>0$, we shall find a time of existence  uniformly in
the regularization, allowing us to take the limit.

Now, let $z^{\ep,\delta,\mu} (\al,t)$ be a solution of the following system (compare with \eqref{paraomt}):
\begin{align}
\label{epsdelmuzt}
\begin{split}z^{\ep,\delta,\mu}_t(\al,t)&=\phi_{\delta} * \phi_{\delta} * \left(Q^2(z^{\ep,\delta,\mu})BR(z^{\ep,\delta,\mu},\om^{\ep,\delta,\mu})\right)(\al,t)+\phi_{\mu} * \left(c^{\ep,\delta,\mu}\left(\phi_{\mu} * \da z^{\ep,\delta,\mu}\right)\right)(\al,t),
\end{split}
\end{align}
\begin{align}
\begin{split}
\om^{\ep,\delta,\mu}_{t}  =& \frac{|\da z^{\ep,\delta,\mu}|}{Q^{2}(z^{\ep,\delta,\mu})}\phi_{\delta} * \phi_{\delta} * \left(\frac{Q^{2}(z^{\ep,\delta,\mu})}{|\da z^{\ep,\delta,\mu}(\al,t)|}\left[-2BR(z^{\ep,\delta,\mu},\om^{\ep,\delta,\mu})_t \cdot z^{\ep,\delta,\mu}_{\al}\right.\right. \\
 & - 2Q(z^{\ep,\delta,\mu})Q(z^{\ep,\delta,\mu})_{\al}|BR(z^{\ep,\delta,\mu},\om^{\ep,\delta,\mu})|^{2} -\da\left(\frac{(\vp^{\ep,\delta,\mu})^{2}}{Q(z^{\ep,\delta,\mu})^{2}}\right) \\
& +\frac{c^{\ep,\delta,\mu}|z^{\ep,\delta,\mu}_\al|^2}{\pi Q(z^{\ep,\delta,\mu})^2}\int_{-\pi}^\pi(Q(z^{\ep,\delta,\mu})^2BR(z^{\ep,\delta,\mu},\om^{\ep,\delta,\mu}))_\be\cdot \frac{z^{\ep,\delta,\mu}_\be}{|z^{\ep,\delta,\mu}_\be|^2} d\be \\
& - \frac{4c^{\ep,\delta,\mu}Q(z^{\ep,\delta,\mu})_\al BR(z^{\ep,\delta,\mu},\om^{\ep,\delta,\mu})\cdot z^{\ep,\delta,\mu}_\al}{Q(z^{\ep,\delta,\mu})}\\
& \left.\left.-\frac{2(c^{\ep,\delta,\mu})^2|z^{\ep,\delta,\mu}_\al|^2Q(z^{\ep,\delta,\mu})_\al}{Q(z^{\ep,\delta,\mu})^3}- 2\da\left(gP_{2}^{-1}(z^{\ep,\delta,\mu})\right)\right] \right) \\
& - \frac{2|\da z^{\ep,\delta,\mu}(\al,t)|}{Q^{2}(z^{\ep,\delta,\mu})}\left(Q(z^{\ep,\delta,\mu})\partial_t Q(z^{\ep,\delta,\mu})\frac{\om^{\ep,\delta,\mu}}{|\da z^{\ep,\delta,\mu}|} - \frac{Q^2(z^{\ep,\delta,\mu})\om^{\ep,\delta,\mu}}{2|\da z^{\ep,\delta,\mu}|^3} \da z^{\ep,\delta,\mu} \cdot \da \partial_t z^{\ep,\delta,\mu}\right) \\
& + \frac{2|\da z^{\ep,\delta,\mu}(\al,t)|}{Q^{2}(z^{\ep,\delta,\mu})} \phi_{\delta} * \phi_{\delta} * \left(Q(z^{\ep,\delta,\mu})\partial_t Q(z^{\ep,\delta,\mu})\frac{\om^{\ep,\delta,\mu}}{|\da z^{\ep,\delta,\mu}|} - \frac{Q^2(z^{\ep,\delta,\mu})\om^{\ep,\delta,\mu}}{2|\da z^{\ep,\delta,\mu}|^3} \da z^{\ep,\delta,\mu} \cdot \da \partial_t z^{\ep,\delta,\mu}\right) \\
& - 2\ep\frac{|\da z^{\ep,\delta,\mu}|}{Q^2(z^{\ep,\delta,\mu})}\Lambda (\phi_{\mu} * \phi_{\mu} * \varphi^{\ep,\delta,\mu}),
\label{epsdelmuwt}
\end{split}
\end{align}
$z^{\ep,\delta,\mu}(\al,0)=z_0(\al)$ and $\om^{\ep,\delta,\mu}(\al,0)=\om_0(\al)$ for $\ep>0, \delta > 0, \mu > 0$, $\phi_{\delta}$ and $\phi_{\mu}$ even mollifiers, and

\begin{align*}
c^{\ep,\delta,\mu}(\al)&=\frac{\al+\pi}{2\pi}\int_{-\pi}^\pi\frac{\partial_{\beta} z^{\ep,\delta,\mu}(\be))}{|\partial_{\beta}
z^{\ep,\delta,\mu}(\be)|^2}\cdot \phi_{\delta} * \phi_{\delta} * (\partial_{\beta} (Q^2(z^{\ep,\delta,\mu})(\be) BR(z^{\ep,\delta,\mu},\om^{\ep,\delta,\mu}))(\be)) d\be \\
& -\int_{-\pi}^\al
\frac{\partial_{\beta} z^{\ep,\delta,\mu}(\beta)}{|\partial_{\beta} z^{\ep,\delta,\mu}(\beta)|^2}\cdot\phi_{\delta} * \phi_{\delta} * (\partial_{\beta} (Q^2(z^{\ep,\delta,\mu})(\beta)
BR(z^{\ep,\delta,\mu},\om^{\ep,\delta,\mu}))(\beta)) d\beta,
\end{align*}
$$
\varphi^{\ep,\delta,\mu}=\D\frac{Q^2(z^{\ep,\delta,\mu})\om^{\ep,\delta,\mu}}{2|\da z^{\ep,\delta,\mu}|}-\mathcal{C}^{\ep,\delta,\mu},$$
$$B^{\ep,\delta,\mu}(t)=\frac{1}{2\pi}\int_{-\pi}^\pi \!\frac{\da z^{\ep,\delta,\mu}(\al,t)}{|\da
z^{\ep,\delta,\mu}(\al,t)|^2}\cdot \da(Q^2(z^{\ep,\delta,\mu}) BR(z^{\ep,\delta,\mu},\om^{\ep,\delta,\mu}))(\al,t) d\al,
$$
\begin{align*}
\mathcal{C}^{\ep,\delta,\mu} & = \phi_{\delta} * \phi_{\delta} * \left(\frac{\al+\pi}{2\pi}\int_{-\pi}^\pi\frac{\partial_{\beta} z^{\ep,\delta,\mu}(\be)}{|\partial_{\beta}
z^{\ep,\delta,\mu}(\be)|}\cdot   (\partial_{\beta} (Q^2(z^{\ep,\delta,\mu}) BR(z^{\ep,\delta,\mu},\om^{\ep,\delta,\mu})))(\be) d\be\right) \\
& -\phi_{\delta} * \phi_{\delta} * \left(\int_{-\pi}^\al
\frac{\partial_{\beta} z^{\ep,\delta,\mu}(\beta)}{|\partial_{\beta} z^{\ep,\delta,\mu}(\beta)|}\cdot (\partial_{\beta} (Q^2(z^{\ep,\delta,\mu})(\beta)
BR(z^{\ep,\delta,\mu},\om^{\ep,\delta,\mu}))(\beta)) d\beta\right).
\end{align*}

We start proving the following lemma:
\begin{lemma}
\label{omegaH3}
Let $z^{\ep,\delta,\mu}(\al,t) \in H^{4}(\mathbb{T})$, $\om^{\ep,\delta,\mu}(\al,t) \in H^2(\mathbb{T})$, $\varphi^{\ep,\delta,\mu}(\al,t) \in H^{3}(\mathbb{T})$. Then $\om^{\ep,\delta,\mu}(\al,t) \in H^{3}(\mathbb{T})$.
\end{lemma}
\begin{proof}
We can write $\om^{\ep,\delta,\mu}$ as:

$$ \om^{\ep,\delta,\mu} = \frac{2|\da z^{\ep,\delta,\mu}|}{Q^2(z^{\ep,\delta,\mu})}\left(\varphi^{\ep,\delta,\mu} + \mathcal{C}^{\ep,\delta,\mu}\right).$$
Taking three derivatives yields
\begin{align*}
\da^{3} \om^{\ep,\delta,\mu} & = \text{ SAFE } + \frac{2|\da z^{\ep,\delta,\mu}|}{Q^2(z^{\ep,\delta,\mu})} \da^{3} \mathcal{C}^{\ep,\delta,\mu} \\
& = \text{ SAFE } - \frac{2|\da z^{\ep,\delta,\mu}|}{Q^2(z^{\ep,\delta,\mu})} \phi_{\delta} * \phi_{\delta} * \da^{2}\left(
\frac{\da z^{\ep,\delta,\mu}}{|\da z^{\ep,\delta,\mu}|}\cdot (\da (Q^2(z^{\ep,\delta,\mu})
BR(z^{\ep,\delta,\mu},\om^{\ep,\delta,\mu})))\right) \\
& = \text{ SAFE } - \frac{2|\da z^{\ep,\delta,\mu}|}{Q^2(z^{\ep,\delta,\mu})}
\phi_{\delta} * \phi_{\delta} * \left(\frac{\da z^{\ep,\delta,\mu}}{|\da z^{\ep,\delta,\mu}|}\cdot (Q^2(z^{\ep,\delta,\mu})
\da^{3} BR(z^{\ep,\delta,\mu},\om^{\ep,\delta,\mu}))\right) \\
& = \text{ SAFE } - \frac{2|\da z^{\ep,\delta,\mu}|}{Q^2(z^{\ep,\delta,\mu})}
\phi_{\delta} * \phi_{\delta} *\left(
\frac{\da z^{\ep,\delta,\mu}}{|\da z^{\ep,\delta,\mu}|}\cdot (Q^2(z^{\ep,\delta,\mu})
 BR(z^{\ep,\delta,\mu},\da^{3} \om^{\ep,\delta,\mu}))\right)
\end{align*}
where SAFE means bounded in $L^{2}$. Using the representation
$$ BR(z^{\ep,\delta,\mu},\da^{3} \om^{\ep,\delta,\mu}) = \text{ SAFE }  + \frac{1}{2} \frac{(\da z^{\ep,\delta,\mu})^{\perp}}{|\da z^{\ep,\delta,\mu}|^{2}} H(\da^{3} \om^{\ep,\delta,\mu})$$

we get that
\begin{align*}
\da^{3} \om^{\ep,\delta,\mu} & = \text{ SAFE }- \frac{2|\da z^{\ep,\delta,\mu}|}{Q^2(z^{\ep,\delta,\mu})}
 \phi_{\delta} * \phi_{\delta} * \underbrace{ \left(Q^2(z^{\ep,\delta,\mu})
 \frac{1}{2} \frac{(\da z^{\ep,\delta,\mu})^{\perp}}{|\da z^{\ep,\delta,\mu}|^{2}} H(\da^{3} \om^{\ep,\delta,\mu})\cdot \frac{\da z^{\ep,\delta,\mu}}{|\da z^{\ep,\delta,\mu}|}\right)  }_{ = 0}
\end{align*}
and we are done. We should remark that the lemma holds independently of $\delta$, $\mu$ and $\ep$.
\end{proof}

We define a distance between data $(z,\om)$ and $(\underline{z}, \underline{\om})$ by taking
$$ d((z,\om),(\underline{z},\underline{\om})) = \|z-\underline{z}\|_{H^{4}} + \|\om-\underline{\om}\|_{H^{2}} + \|\varphi - \underline{\varphi}\|_{H^{3}}$$

where $\varphi$ and $\underline{\varphi}$ arise from $(z,\om)$ and $(\underline{z},\underline{\om})$ respectively by \eqref{fvar}. Let $XX$ denote the resulting metric space. The proof of Lemma \ref{omegaH3} gives also the following

\begin{corollary}
The map $(z,\om)\mapsto \om$ is Lipschitz from any ball in $XX$ into $H^{3}(\T)$.
\end{corollary}

We note that throughout this section we will repeatedly use the following commutator estimate for convolutions:

\begin{equation}
\label{commmol}
 \|\phi_{\delta} * (\da f g) - g \phi_{\delta} * (\da f)\|_{L^{2}} \leq C\|\da g\|_{L^{\infty}} \|f\|_{L^{2}},
\end{equation}

where the constant $C$ is independent of $\delta, f$ and $g$. We can now operate to get the following expression for $\varphi^{\ep,\delta,\mu}$:

\begin{align*}
\partial_t \varphi^{\ep,\delta,\mu} & = \frac{Q(z^{\ep,\delta,\mu})\partial_t Q(z^{\ep,\delta,\mu}) \om^{\ep,\delta,\mu}}{|\da z^{\ep,\delta,\mu}|}
- \frac{Q^2(z^{\ep,\delta,\mu})\om^{\ep,\delta,\mu}}{2|\da z^{\ep,\delta,\mu}|^{3}}\da z^{\ep,\delta,\mu} \cdot \partial_t \da z^{\ep,\delta,\mu}
+ \frac{Q^2(z^{\ep,\delta,\mu})\partial_t \om^{\ep,\delta,\mu}}{2|\da z^{\ep,\delta,\mu}|}
- \partial_t \mathcal{C}^{\ep,\delta,\mu} \\
& = \phi_{\delta} * \phi_{\delta} * \left(\frac{Q^{2}(z^{\ep,\delta,\mu})}{2|\da z^{\ep,\delta,\mu}(\al,t)|}\left[-2BR(z^{\ep,\delta,\mu},\om^{\ep,\delta,\mu})_t \cdot z^{\ep,\delta,\mu}_{\al}\right.\right. \\
 & - 2Q(z^{\ep,\delta,\mu})Q(z^{\ep,\delta,\mu})_{\al}|BR(z^{\ep,\delta,\mu},\om^{\ep,\delta,\mu})|^{2} -\da\left(\frac{(\vp^{\ep,\delta,\mu})^{2}}{Q(z^{\ep,\delta,\mu})^{2}}\right) \\
& +\frac{c^{\ep,\delta,\mu}|z^{\ep,\delta,\mu}_\al|^2}{\pi Q(z^{\ep,\delta,\mu})^2}\int_{-\pi}^\pi(Q(z^{\ep,\delta,\mu})^2BR(z^{\ep,\delta,\mu},\om^{\ep,\delta,\mu}))_\be\cdot \frac{z^{\ep,\delta,\mu}_\be}{|z^{\ep,\delta,\mu}_\be|^2} d\be \\
& - \frac{4c^{\ep,\delta,\mu}Q(z^{\ep,\delta,\mu})_\al BR(z^{\ep,\delta,\mu},\om^{\ep,\delta,\mu})\cdot z^{\ep,\delta,\mu}_\al}{Q(z^{\ep,\delta,\mu})}\\
& \left.\left.-\frac{2(c^{\ep,\delta,\mu})^2|z^{\ep,\delta,\mu}_\al|^2Q(z^{\ep,\delta,\mu})_\al}{Q(z^{\ep,\delta,\mu})^3}- 2\da\left(gP_{2}^{-1}(z^{\ep,\delta,\mu})\right)\right] \right) \\
& + \phi_{\delta} * \phi_{\delta} * \left(Q(z^{\ep,\delta,\mu})\partial_t Q(z^{\ep,\delta,\mu})\frac{\om^{\ep,\delta,\mu}}{|\da z^{\ep,\delta,\mu}|} - \frac{Q^2(z^{\ep,\delta,\mu})\om^{\ep,\delta,\mu}}{2|\da z^{\ep,\delta,\mu}|^3} \da z^{\ep,\delta,\mu} \cdot \da \partial_t z^{\ep,\delta,\mu}\right) \\
& - \ep\Lambda (\phi_{\mu} * \phi_{\mu} * \varphi^{\ep,\delta,\mu})
 - \partial_ t \mathcal{C}^{\ep,\delta,\mu}.
\end{align*}

The RHS of the evolution equations for $z^{\ep,\delta,\mu}$ and $\varphi^{\ep,\delta,\mu}$ are Lipschitz in the spaces $H^{4}(\T)$ and $H^{3+\frac{1}{2}}(\T)$ since they are mollified. For the case of $\omega^{\ep,\delta,\mu}$ (Lipschitz in the space $H^{2}(\mathbb{T})$) we use that for $\delta$ small enough $\phi_{\delta} * \phi_{\delta}$ is close to the identity and the a priori bounds. In all of the cases we have taken advantage of Lemma \ref{omegaH3}. Therefore we can solve (\ref{epsdelmuzt}-\ref{epsdelmuwt}) for short time, thanks to Picard's theorem.

Now, we can perform energy estimates as in the a priori case to get uniform bounds in $\mu$ and we can let $\mu$ go to zero. The energy estimates that we can get are the following:

\begin{align*}
& \frac{d}{dt}\left(\|z^{\ep,\delta,\mu}\|^2_{H^4}+\|\F(z^{\ep,\delta,\mu})\|^2_{L^\infty}+\|\om^{\ep,\delta,\mu}\|^2_{H^{2}}+\|\varphi^{\ep,\delta,\mu}\|^2_{H^{3+\frac{1}{2}}}
+\sum_{l=0}^{4}\frac{1}{m^{\ep,\delta,\mu}(q^l)}\right)(t) \\
&\leq
C(\ep,\delta)\left(\|z^{\ep,\delta,\mu}\|^2_{H^4}+\|\F(z^{\ep,\delta,\mu})\|^2_{L^\infty}+
\|\om^{\ep,\delta,\mu}\|^2_{H^{2}}+\|\varphi^{\ep,\delta,\mu}\|^2_{H^{3+\frac{1}{2}}}+\sum_{l=0}^4\frac{1}{m^{\ep,\delta,\mu}(q^l)}\right)^j(t).
\end{align*}

We should note that for the new system without the $\phi_\mu$ mollifier, the length of the tangent vector $|\da z^{\ep,\delta}|$ is now constant in space and depends only on time. Lemma \ref{omegaH3} still applies and we can still perform energy estimates as in the a priori case. The only difference relies on the fact that we should have to move the mollifiers and apply the estimate \eqref{commmol}. We should also remark that because of the dissipative term $\ep \Lambda \varphi^{\ep,\delta}$ it is enough to use the following estimate

$$ \frac{1}{2}\frac{d}{dt}\|\Lambda^{1/2}(\da^{3} \varphi^{\ep,\delta})\|_{L^{2}}^{2} = \int \Lambda(\da^{3} \varphi^{\ep,\delta}) \partial_t \da^{3} \varphi^{\ep,\delta}d\al
 \leq \frac{\ep}{64} \|\Lambda(\da^{3} \varphi^{\ep,\delta})\|_{L^{2}}^{2} + C(\ep) \| \partial_t \da^{3} \varphi^{\ep,\delta}\|_{L^{2}}^{2}$$

and hence require only that $\partial_t \varphi^{\ep,\delta} \in H^{3}(\mathbb{T})$ (instead of the $H^{3+\frac{1}{2}}(\mathbb{T})$ that was required before) except for the transport term that can be estimated as in subsection \ref{subsubenergyphi}. The estimations are performed following exactly the same steps of subsection \ref{subsection4d}. More precisely, we can get the following energy estimates:

\begin{align*}
& \frac{d}{dt}\left(\|z^{\ep,\delta}\|^2_{H^4}+\|\F(z^{\ep,\delta})\|^2_{L^\infty}+\|\om^{\ep,\delta}\|^2_{H^{2}}+\|\varphi^{\ep,\delta}\|^2_{H^{3+\frac{1}{2}}}
+\sum_{l=0}^{4}\frac{1}{m^{\ep,\delta}(q^l)}\right)(t) \\
&\leq
C(\ep)\left(\|z^{\ep,\delta}\|^2_{H^4}+\|\F(z^{\ep,\delta})\|^2_{L^\infty}+
\|\om^{\ep,\delta}\|^2_{H^{2}}+\|\varphi^{\ep,\delta}\|^2_{H^{3+\frac{1}{2}}}+\sum_{l=0}^4\frac{1}{m^{\ep,\delta}(q^l)}\right)^j(t).
\end{align*}

Under these conditions, we can let $\delta$ go to zero.

Finally, let $z^\ep (\al,t)$ be a solution of the following system (compare with \eqref{paraomt}):
\begin{align}
\label{epszt}
\begin{split}z^{\ep}_t(\al,t)&=Q^2(z^{\ep})(\al,t)BR(z^{\ep},\om^{\ep})(\al,t)+c^{\ep}(\al,t)\da z^{\ep}(\al,t),
\end{split}
\end{align}
\begin{align}
\begin{split}
\om^{\ep}_{t}  =& -2BR(z^{\ep},\om^{\ep})_t \cdot z^{\ep}_{\al} - 2Q(z^{\ep})Q(z^{\ep})_{\al}|BR(z^{\ep},\om^{\ep})|^{2} -\da\left(\frac{(\vp^{\ep})^{2}}{Q(z^{\ep})^{2}}\right)\\
&+\frac{c^{\ep}|z^{\ep}_\al|^2}{\pi Q(z^{\ep})^2}\int_{-\pi}^\pi(Q(z^{\ep})^2BR(z^{\ep},\om^{\ep}))_\be\cdot \frac{z^{\ep}_\be}{|z^{\ep}_\be|^2} d\be-\frac{4c^{\ep}Q(z^{\ep})_\al BR(z^{\ep},\om^{\ep})\cdot z^{\ep}_\al}{Q(z^{\ep})}\\
& -\frac{2(c^{\ep})^2|z^{\ep}_\al|^2Q(z^{\ep})_\al}{Q(z^{\ep})^3}- 2\da\left(gP_{2}^{-1}(z^{\ep})\right) - 2\ep\frac{|\da z^{\ep}|}{Q^2(z^{\ep})}\Lambda \varphi^{\ep},
\label{epswt}
\end{split}
\end{align}
$z^\ep(\al,0)=z_0(\al)$ and $\om^\ep(\al,0)=\om_0(\al)$ for $\ep>0$, where

\begin{align*}
c^{\ep}(\al)&=\frac{\al+\pi}{2\pi}\int_{-\pi}^\pi\frac{\partial_{\beta} z^{\ep}(\be))}{|\partial_{\beta}
z^{\ep}(\be)|^2}\cdot \partial_{\beta} (Q^2(z^{\ep})(\be) BR(z^{\ep},\om^{\ep})(\be)) d\be \\
& -\int_{-\pi}^\al
\frac{\partial_{\beta} z^{\ep}(\beta)}{|\partial_{\beta} z^{\ep}(\beta)|^2}\cdot\partial_{\beta} (Q^2(z^{\ep})(\beta)
BR(z^{\ep},\om^{\ep})(\beta)) d\beta,
\end{align*}
$$
\varphi^\ep=\D\frac{Q^2(z^{\ep})\om^\ep}{2|\da z^\ep|}-|\da z^\ep|c^\ep,\qquad B^\ep(t)=\frac{1}{2\pi}\int_{-\pi}^\pi \!\frac{\da z^\ep(\al,t)}{|\da
z^\ep(\al,t)|^2}\cdot \da(Q^2(z^{\ep}) BR(z^\ep,\om^\ep))(\al,t) d\al.
$$

Proceeding as in section \ref{subdefphi} (compare with equation \eqref{phiat412}) we find
\begin{align}
\begin{split}\label{fephitep}
\da \vp^{\ep}_{ t} = & -B^{\ep}(t) \vp^{\ep}_{\al} - \frac{\da^{2}((\vp^{\ep})^{2})}{2|\da z^{\ep}|}
+ \da\left(\frac{\da Q(z^{\ep})}{|\da z^{\ep}|Q(z^{\ep})}(\vp^{\ep})^{2}\right) \\
& - Q(z^{\ep})^{2} \partial_t BR(z^{\ep},\om^{\ep}) \cdot \dpa z^{\ep} \frac{\da^{2} z^{\ep} \cdot \dpa z^{\ep}}{|\da z^{\ep}|^{3}} -
\partial_t(|\da z^{\ep}|B^{\ep}(t)) \\
& + \da(Q(z^{\ep})^{2} BR(z^{\ep},\om^{\ep})) \cdot \dpa z^{\ep} \frac{\partial_t \da z^{\ep} \cdot \dpa z^{\ep}}{|\da z^{\ep}|^{3}}
+ 2\da (Q(z^{\ep})\partial_t Q(z^{\ep}) BR(z^{\ep},\om^{\ep}))\cdot \frac{\da z^{\ep}}{|\da z^{\ep}|} \\
& - \da \left(Q(z^{\ep})^{2}\frac{\da (gP^{-1}_{2}(z^{\ep}))}{|\da z^{\ep}|}\right)
+ \da \left(Q(z^{\ep})\partial_t Q(z^{\ep}) \frac{\om^{\ep}}{|\da z^{\ep}|}\right) \\
& - \da \left(2c^{\ep} BR(z^{\ep},\om^{\ep}) \cdot \frac{\da z^{\ep}}{|\da z^{\ep}|}Q(z^{\ep}) \da Q(z^{\ep})\right)
- \da \left(\frac{\da Q(z^{\ep})}{Q(z^{\ep})}(c^{\ep})^{2}|\da z^{\ep}|\right) \\
& - \da \left(\frac{Q(z^{\ep})^{3}}{|\da z^{\ep}|}|BR(z^{\ep},\om^{\ep})|^{2}\da Q(z^{\ep})\right) -\ep\Lambda\da\varphi^\ep.
\end{split}
\end{align}
We also define (compare with equation \eqref{R-T})
\begin{align*}
\begin{split}
\sigma^\ep &=(\partial_t BR(z^\ep,\om^\ep)+\frac{\varphi^\ep}{|\da z^\ep|}\da BR(z^\ep,\om^\ep))\cdot \dpa z^\ep
+\frac12\frac{\om^\ep}{|\da z^\ep|^2}(\da
z^\ep_t+\frac{\varphi^\ep}{|\da z^\ep|}\da^2 z^\ep)\cdot
\dpa z^\ep \\
&\quad + Q(z^{\ep})\left|BR(z^{\ep},\om^{\ep}) + \frac{\om^{\ep}}{2|\da z^{\ep}|^{2}}\da z^{\ep}\right|^{2}(\nabla Q(z^{\ep})) \cdot \dpa z^{\ep}
+g(\nabla P_{2}^{-1}(z^{\ep})) \cdot \dpa z^{\ep}.
\end{split}
\end{align*}

\begin{rem}
The system (\ref{epszt}-\ref{epswt}) is analogous to the system considered in \cite[Section 8]{Cordoba-Cordoba-Gancedo:interface-water-waves-2d}. We point out an unfortunate typographical error in that section; the Laplacian should have been written as the square root of the Laplacian.
\end{rem}

For this $\varepsilon$-system (\ref{epszt}-\ref{epswt}) we now know that there is local-existence for initial data
satisfying $\F(z_0)(\al,\beta)< \infty$ even if
$\sigma^\ep(\al,0)$ does not have the proper sign. In the
following we shall show briefly  how to obtain a solution of the
regularized system with $z^{\ep}\in C([0,T^{\ep}],H^{k}), \varphi^\ep\in
C([0,T^{\ep}],H^{k-\frac{1}{2}}), \om^{\ep}\in C([0,T^{\ep}],H^{k-2})$ for $k\geq 4$.

The next step is to integrate the system during a time $T$
independent of $\ep$. We will show that for this system we have
\begin{align}
\begin{split}\label{ntniepsilon}
\frac{d}{dt}E(t)&\leq CE^{p}(t),
\end{split}
\end{align}
where $E(t)$ is given by the analogous formula \eqref{E} for
the $\ep$-system, and $C$ and $p$ are constants independent of $\ep$.

  In the following we shall see what is the impact of the $\ep$ system on the a priori estimates and check that there is no practical impact for sufficiently small $\ep$. To do that, we will show the corresponding uniform estimates for $k=4$ and
leave to the reader the remaining easier cases. Let us consider
the one corresponding to $I_3$ in section \ref{subsubenergyphi}, we have
$$
I^\ep_3=-2\int_{-\pi}^\pi\frac{1}{|\da z^\ep|^3}\la(\da^3\varphi^\ep(\al))\da^2(Q^2(z^{\ep})\sigma^\ep\da^2z^\ep\cdot\dpa z^\ep)d\al.
$$
Proceeding in the same way as before, we can perform the same splittings and get uniform bounds such that
$I^\ep_3=-S^\ep+M_4^\ep+\mbox{``bounded terms''}$ where
$S^\ep$ corresponds to $S$ in \eqref{fS},
$$ |\text{``bounded terms''}| \leq CE^{p}(t),$$
and

$$
M_4^\ep = -2\ep\int_{-\pi}^\pi\frac{\da^2z^\ep\cdot\dpa z^\ep}{|\da z^\ep|^2}H(\da^3\varphi^\ep)H(\Lambda\da^3\varphi^\ep) d\al.
$$

Then we can write $M_4^{\ep}$ as follows
$$
M_4^\ep = -2\ep\int_{-\pi}^\pi\la^{\frac12}\Big(\frac{\da^2z^\ep\cdot\dpa
z^\ep}{|\da z^\ep|^2}
H(\da^3\varphi^\ep)\Big)\la^{\frac12}(H\da^3\varphi^\ep) d\al,
$$
and therefore, for small $\varepsilon$
$$
M_4^{\ep} \leq \|\la^{\frac12}\da^3\varphi^\ep\|^2_{L^2}+\mbox{``bounded terms''},
$$
which gives
\begin{align*}
\begin{split}
\frac{d}{dt}E(t)&\leq
CE^p(t)- \frac{\ep}{2}\|\Lambda(\da^{3}\varphi)\|_{L^{2}}^{2}.
\end{split}
\end{align*}
This finally shows \eqref{ntniepsilon} and therefore
$$
E(t)\leq (Ct(1-p)+E^{1-p}(0))^{1/(1-p)}.
$$
Now we are in position to extend the time of existence $T^\ep$ so
long as the above estimate works and obtain  a time $T$
dependent only on the initial data (arc-chord,
Rayleigh-Taylor, distance to the points $q^{0}, \ldots, q^{4}$, and Sobolev norms of $z, \omega,$ and $\varphi$). We can let $\ep$ tend to $0$, and get a solution of the original system. This concludes the proof.

\subsection*{{\bf Acknowledgements}}

\smallskip

 AC, DC, FG and JGS were partially supported by the grant {\sc MTM2008-03754} of the MCINN (Spain) and
the grant StG-203138CDSIF  of the ERC. CF was partially supported by
NSF grant DMS-0901040. FG was partially supported by NSF grant DMS-0901810. We are grateful to CCC of Universidad
Aut\'onoma de Madrid for computing facilities.

\bibliographystyle{abbrv}
\bibliography{references}

\begin{tabular}{ll}
\textbf{Angel Castro} &  \\
{\small D\'epartement de Math\'ematiques et Applications} & \\
{\small \'Ecole Normale Sup\'erieure} &\\
{\small 45, Rue d'Ulm, 75005 Paris} & \\
{\small Email: castro@dma.ens.fr} & \\
   & \\
\textbf{Diego C\'ordoba} &  \textbf{Charles Fefferman}\\
{\small Instituto de Ciencias Matem\'aticas} & {\small Department of Mathematics}\\
{\small Consejo Superior de Investigaciones Cient\'ificas} & {\small Princeton University}\\
{\small C/ Nicol\'{a}s Cabrera, 13-15} & {\small 1102 Fine Hall, Washington Rd, }\\
{\small Campus Cantoblanco UAM, 28049 Madrid} & {\small Princeton, NJ 08544, USA}\\
{\small Email: dcg@icmat.es} & {\small Email: cf@math.princeton.edu}\\
 & \\
\textbf{Francisco Gancedo} &  \textbf{Javier G\'omez-Serrano}\\
{\small Departamento de An\'alisis Matem\'atico} & {\small Instituto de Ciencias Matem\'aticas}\\
{\small Universidad de Sevilla} & {\small Consejo Superior de Investigaciones Cient\'ificas}\\
{\small C/ Tarfia, s/n } & {\small C/ Nicol\'{a}s Cabrera, 13-15} \\
{\small Campus Reina Mercedes, 41012 Sevilla}  & {\small Campus Cantoblanco UAM, 28049 Madrid} \\
{\small Email: fgancedo@us.es} & {\small Email: javier.gomez@icmat.es}\\
\end{tabular}

\end{document}